\documentclass[a4paper,12pt]{article}

\usepackage{amsmath,amssymb,amsthm, amscd}
\usepackage{bm}
\usepackage[all]{xy}

\allowdisplaybreaks
\numberwithin{equation}{section}

\newcommand{\Fg}{\mathfrak{g}}
\newcommand{\Fh}{\mathfrak{h}}

\newcommand{\BZ}{\mathbb{Z}}
\newcommand{\BQ}{\mathbb{Q}}
\newcommand{\BR}{\mathbb{R}}
\newcommand{\BC}{\mathbb{C}}
\newcommand{\BB}{\mathbb{B}}

\newcommand{\ve}{\varepsilon}
\newcommand{\vp}{\varphi}
\newcommand{\vpi}{\varpi}

\newcommand{\bzero}{{\bf 0}}

\newcommand{\bp}{\mathbf{p}}
\newcommand{\bq}{\mathbf{q}}
\newcommand{\bv}{\bm{v}}

\newcommand{\Hom}{\mathop{\rm Hom}\nolimits}

\newcommand{\wt}{\mathop{\rm wt}\nolimits}
\newcommand{\id}{\mathop{\rm id}\nolimits}
\newcommand{\gch}{\mathop{\rm gch}\nolimits}
\newcommand{\cl}{\mathop{\rm cl}\nolimits}

\newcommand{\af}{\mathrm{af}}

\newcommand{\Lus}{\mathop{\sf Lus}\nolimits}

\newcommand{\rr}{\Delta_{\af}}
\newcommand{\prr}{\Delta_{\af}^{+}}

\newcommand{\lng}{w_{\circ}}

\newcommand{\si}{\frac{\infty}{2}}
\newcommand{\sell}{\ell^{\frac{\infty}{2}}}
\newcommand{\sil}{\prec}
\newcommand{\sile}{\preceq}
\newcommand{\sig}{\succ}

\newcommand{\QLS}{\mathrm{QLS}}
\newcommand{\SLS}{\mathbb{B}^{\frac{\infty}{2}}}

\newcommand{\SB}{\mathrm{BG}^{\si}(W_{\af})}
\newcommand{\QB}{\mathrm{QBG}(W)}

\newcommand{\tb}[1]{\le_{#1}}

\newcommand{\tbmin}[3]{\min(#1W_{#2},\le_{#3}\nobreak)}

\newcommand{\tiv}[2]{\ti{\bv}(#1,#2)}
\newcommand{\tixi}[2]{\ti{\bm{\xi}}(#1,#2)}
\newcommand{\bxi}[1]{\bm{\xi}(#1)}

\newcommand{\turn}[1]{[0,1]_{#1}}

\newcommand{\J}{S}
\newcommand{\Jc}{I \setminus \J}
\newcommand{\QJ}{Q_{\J}}
\newcommand{\QJv}{Q_{\J}^{\vee}}
\newcommand{\QJvp}{Q_{\J}^{\vee,+}}
\newcommand{\DeJ}{\Delta_{\J}}
\newcommand{\PJ}{\Pi^{\J}}
\newcommand{\WJ}{W_{\J}}
\newcommand{\WJu}{W^{\J}}
\newcommand{\WJa}{(W^{\J})_{\af}}
\newcommand{\SBJ}{\mathrm{BG}^{\si}((\WJu)_{\af})}
\newcommand{\SBa}{\mathrm{BG}^{\si}_{\sigma\lambda}((\WJu)_{\af})}
\newcommand{\SBb}[1]{\mathrm{BG}^{\si}_{#1\lambda}((\WJu)_{\af})}
\newcommand{\QBJ}{\mathrm{QBG}(\WJu)}
\newcommand{\QBa}{\mathrm{QBG}_{\sigma\lambda}(\WJu)}
\newcommand{\QBb}[1]{\mathrm{QBG}_{#1\lambda}(\WJu)}

\newcommand{\PS}[1]{\Pi^{\J_{#1}}}
\newcommand{\WSu}[1]{W^{\J_{#1}}}
\newcommand{\WS}[1]{W_{\J_{#1}}}
\newcommand{\QSv}[1]{Q_{\J_{#1}}^{\vee}}

\newcommand{\DeS}[1]{\Delta_{\J_{#1}}}

\newcommand{\io}[2]{\iota(#1,#2)}
\newcommand{\ze}[2]{\zeta(#1,#2)}

\newcommand{\edge}[1]{ \xrightarrow{\hspace{2pt}#1\hspace{2pt}} }
\newcommand{\mcr}[1]{\lfloor #1 \rfloor}

\newcommand{\pair}[2]{\langle #1,\,#2 \rangle}

\newcommand{\ol}[1]{\overline{#1}}
\newcommand{\ti}[1]{\widetilde{#1}}

\numberwithin{equation}{section}

\theoremstyle{plain}
\newtheorem{lem}{Lemma}[section]
\newtheorem{prop}[lem]{Proposition}
\newtheorem{thm}[lem]{Theorem}
\newtheorem{cor}[lem]{Corollary}

\theoremstyle{definition}
\newtheorem{dfn}[lem]{Definition}

\theoremstyle{remark}

\newtheorem{rem}[lem]{Remark}
\newtheorem{claim}{Claim}[lem]

\newenvironment{enu}{%
 \begin{enumerate}%
}{\end{enumerate}}

\newcommand{\bqed}{\quad \hbox{\rule[-0.5pt]{3pt}{8pt}}}

\newcommand{\vsp}{\vspace{3mm}}

\begin{document}

\setlength{\baselineskip}{18pt}

\title{\Large\bf 
Tensor product decomposition theorem \\ 
for quantum Lakshmibai-Seshadri paths and \\ 
standard monomial theory \\ for semi-infinite Lakshmibai-Seshadri paths%
}
\author{%
Satoshi Naito \\ 
 \small Department of Mathematics, Tokyo Institute of Technology, \\
 \small 2-12-1 Oh-okayama, Meguro-ku, Tokyo 152-8551, Japan \\
 \small (e-mail: {\tt naito@math.titech.ac.jp}) \\[5mm]
Fumihiko Nomoto \\ 
 \small Department of Mathematics, Tokyo Institute of Technology, \\
 \small 2-12-1 Oh-okayama, Meguro-ku, Tokyo 152-8551, Japan \\
 \small (e-mail: {\tt nomoto.f.aa@m.titech.ac.jp}) \\[3mm]
and \\[3mm]
Daisuke Sagaki \\ 
 \small Institute of Mathematics, University of Tsukuba, \\
 \small 1-1-1 Tennodai, Tsukuba, Ibaraki 305-8571, Japan \\
 \small (e-mail: {\tt sagaki@math.tsukuba.ac.jp})
}
\date{}
\maketitle

%
\begin{abstract} \setlength{\baselineskip}{15pt}
Let $\lambda$ be a (level-zero) dominant integral weight 
for an untwisted affine Lie algebra, and let $\QLS(\lambda)$ 
denote the quantum Lakshmibai-Seshadri (QLS) paths of shape $\lambda$.
For an element $w$ of a finite Weyl group $W$, 
the specializations at $t = 0$ and $t = \infty$ of 
the nonsymmetric Macdonald polynomial $E_{w \lambda}(q, t)$ are 
explicitly described in terms of QLS paths of shape $\lambda$ 
and the degree function defined on them.
Also, for (level-zero) dominant integral weights $\lambda$, $\mu$, 
we have an isomorphism $\Theta : \QLS(\lambda + \mu) \rightarrow \QLS(\lambda) \otimes \QLS(\mu)$ of crystals.
In this paper, we study the behavior of the degree function 
under the isomorphism $\Theta$ of crystals through 
the relationship between semi-infinite Lakshmibai-Seshadri (LS) paths and QLS paths.
As an application, we give a crystal-theoretic proof of 
a recursion formula for the graded characters of generalized Weyl modules.
\end{abstract}
%
%
\section{Introduction.} 
\label{sec:intro}

In our previous paper \cite{KNS}, we established (a combinatorial version of) 
standard monomial theory for semi-infinite Lakshmibai-Seshadri (LS for short) paths.
To be more precise, let $\lambda \in P^{+}$ be a (level-zero) dominant integral weight 
for an untwisted affine Lie algebra $\Fg_{\af}$, and 
let $\SLS(\lambda)$ denote the set of semi-infinite LS paths of shape $\lambda$; 
note that the set $\SLS(\lambda)$ provides a realization of the crystal basis 
of the extremal weight module $V(\lambda)$ 
over the quantum affine algebra $U_{q}(\Fg_{\af})$ (see \cite{INS}).
In \cite{KNS}, we proved that for 
(level-zero) dominant integral weights $\lambda, \mu \in P^{+}$, 
there exists an embedding $\Xi : \SLS(\lambda + \mu) \rightarrow \SLS(\lambda) \otimes \SLS(\mu)$ 
of crystals that sends the straight-line path $\pi_{\lambda + \mu}$ to 
the tensor product $\pi_{\lambda} \otimes \pi_{\mu}$ of 
the straight-line paths $\pi_{\lambda}$ and $\pi_{\mu}$.
In particular, the restriction of $\Xi$ to 
the connected component $\SLS_{0}(\lambda + \mu)$ of $\SLS(\lambda + \mu)$ 
containing $\pi_{\lambda + \mu}$ gives an isomorphism of crystals 
from $\SLS_{0}(\lambda + \mu)$ to the connected component $(\SLS(\lambda) \otimes \SLS(\mu))_{0}$ 
of $\SLS(\lambda) \otimes \SLS(\mu)$ containing $\pi_{\lambda} \otimes \pi_{\mu}$.
Moreover, in \cite{KNS}, we gave an explicit description of the image 
$\Xi(\SLS(\lambda + \mu)) \subset \SLS(\lambda) \otimes \SLS(\mu)$ 
in terms of the semi-infinite Bruhat order on the affine Weyl group $W_{\af}$ 
in a way similar to the one for the ordinary standard monomial theory due to 
Littelmann (\cite{Lit96}). 

Also, in \cite{NS05}, we proved the tensor product decomposition theorem 
for quantum Lakshmibai-Seshadri (QLS for short) paths.
To be more precise, for a (level-zero) dominant integral weight 
$\lambda \in P^{+} = \sum_{i \in I} m_{i} \vpi_{i}$, 
where the $\vpi_{i}$, $i \in I$, are the level-zero fundamental weights for $\Fg_{\af}$, 
let $\QLS(\lambda)$ denote the set of QLS paths of shape $\lambda$; 
note that the set $\QLS(\lambda)$ provides a realization of the crystal basis of 
the tensor product $\bigotimes_{i \in I} W(\vpi_{i})^{\otimes m_{i}}$ of 
the level-zero fundamental representations $W(\vpi_{i})$, $i \in I$, 
over the quantum affine algebra $U_{q}^{\prime}(\Fg_{\af})$ 
without the degree operator (see \cite{NS03}, \cite{NS06}).
In \cite{NS05}, we proved that for (level-zero) 
dominant integral weights $\lambda, \mu \in P^{+}$, 
there exists an isomorphism $\Theta : \QLS(\lambda + \mu) \rightarrow \QLS(\lambda) \otimes \QLS(\mu)$ 
of crystals that sends the straight-line path $\eta_{\lambda + \mu}$ to 
the tensor product $\eta_{\lambda} \otimes \eta_{\mu}$ of 
the straight-line paths $\eta_{\lambda}$ and $\eta_{\mu}$.

Based on the fact that the affine Weyl group $W_{\af}$ is 
the semi-direct product of the finite Weyl group $W$ and 
the coroot lattice $Q^{\vee} = \sum_{i\in I} \BZ \alpha_{i}^{\vee}$ of 
the underlying simple Lie algebra $\Fg \subset \Fg_{\af}$,
we can define a surjective morphism $\cl : \SLS(\lambda) \rightarrow \QLS(\lambda)$ 
of crystals for $\lambda \in P^{+}$;
note that for $\lambda, \mu \in P^{+}$, we have the following commutative diagram:
\begin{equation*}
\begin{CD}
\SLS_{0}(\lambda + \mu) @>{\Xi}>> \left(\SLS(\lambda) \otimes \SLS(\mu)\right)_{0} \\ 
@V{\cl}VV @VV{\cl \otimes \cl}V \\
\QLS(\lambda + \mu) @>>{\Theta}> \QLS(\lambda) \otimes \QLS(\mu).
\end{CD}
\end{equation*}
Here we note that $\cl(\SLS_{0}(\lambda)) = \QLS(\lambda)$; 
in fact, for each $\eta \in \QLS(\lambda)$, 
there exists a unique element $\pi_{\eta} \in \SLS_{0}(\lambda)$ such that $\cl(\pi_{\eta}) = \eta$ 
and such that the final direction of $\pi_{\eta}$ is identical to that of $\eta \in W$.
Using this element $\pi_{\eta}$, we define the (tail) degree function 
$\deg_{\lambda} : \QLS(\lambda) \rightarrow \BZ_{\le 0}$ by:
$\wt(\pi_{\eta}) = \lambda - \gamma + \deg_{\lambda}(\eta) \delta$,
where $\gamma \in Q$ and $\delta$ is the null root of the affine Lie algebra $\Fg_{\af}$.
Also, for an arbitrary $w \in W$, we can define the degree function 
$\deg_{w \lambda} : \QLS(\lambda) \rightarrow \BZ_{\le 0}$ at $w \lambda$ 
by ``twisting'' $\pi_{\eta}$ by a certain element in $Q^{\vee} \subset W_{\af}$ 
corresponding to $w$; note that $\deg_{\lambda} = \deg_{e \lambda}$, 
where $e$ is the identity element of $W$.
In \cite{LNSSS2} and \cite{LNSSS3}, 
we gave an explicit description of the specialization at $t = 0$ 
of the nonsymmetric Macdonald polynomial $E_{w \lambda}(q, t)$ 
in terms of (a specific subset of) the set $\QLS(\lambda)$ 
equipped with the degree function $\deg_{\lambda}$. 
In addition, in \cite{NNS1} (see also \cite{NS18}), 
we gave an explicit description of the specialization at $t = \infty$ of 
the nonsymmetric Macdonald polynomial $E_{w \lambda}(q, t)$ 
in terms of (a specific subset of) the set $\QLS(\lambda)$ 
equipped with the degree function $\deg_{w \lambda}$.

In this paper, we study the behavior of the degree function 
under the isomorphism $\Theta : \QLS(\lambda + \mu) \rightarrow \QLS(\lambda) \otimes \QLS(\mu)$ 
of crystals for $\lambda, \mu \in P^{+}$.
To be more precise, let $\eta \in \QLS(\lambda + \mu)$, 
and write $\Theta(\eta) \in \QLS(\lambda) \otimes \QLS(\mu)$ as: 
$\Theta(\eta) = \eta_{1} \otimes \eta_{2}$, with $\eta_{1} \in \QLS(\lambda)$ and $\eta_{2} \in \QLS(\mu)$.
Then our main result (Theorem~\ref{thm:main})) states that for an arbitrary $w \in W$,
\begin{equation*}
\deg_{w(\lambda + \mu)}(\eta) = 
\deg_{\io{\eta_{2}}{w} \lambda}(\eta_{1}) + \deg_{w \mu}(\eta_{2}) - \pair{\lambda}{\ze{\eta_{2}}{w}}. 
\end{equation*}
Here, $\io{\eta_{2}}{w}$ is an element of $W$, 
called the initial direction of $\eta_{2}$ with respect to $w$, 
defined in terms of the quantum version of Deodhar lifts introduced in \cite{LNSSS1}; 
note that if $\mu \in P^{+}$ is regular, then the $\io{\eta_{2}}{w}$ is 
just the initial direction $\iota(\eta_{2})$ of $\eta_{2}$.
Also, $\ze{\eta_{2}}{w}$ is an element of 
$Q^{\vee, +} := \sum_{i \in I} \BZ_{\ge 0} \alpha_{i}^{\vee}$ 
defined in terms of the quantum Bruhat graph (see Section~\ref{subsec:main} for details);
note that if $\mu \in P^{+}$ is regular and $\eta_{2} \in \QLS(\mu)$ is 
of the form $\eta_{2} = (v_{1}, \ldots, v_{s}; \sigma_{0} = 0, \ldots, \sigma_{s} = 1)$, 
then the $\ze{\eta_{2}}{w}$ is just the element
$\sum_{k = 1}^{s} \wt(v_{k+1} \Rightarrow v_{k})$, where $v_{s+1} := w$.

As an application of this result, 
we obtain the following equation (Corollary~\ref{cor:main}) 
between the generating functions 
$\gch_{w \lambda} \QLS(\lambda) := 
\sum_{\eta \in \QLS(\lambda)} q^{\deg_{w \lambda}(\eta)} e^{\wt(\eta)}$ 
for $\lambda \in P^{+}$ and $w \in W$ (called graded characters):
\begin{equation*}
\gch_{w(\lambda + \mu)} \QLS(\lambda + \mu) = 
\sum_{\eta \in \QLS(\mu)} 
 e^{\wt(\eta)} q^{\deg_{w \mu}(\eta) - \pair{\lambda}{\ze{\eta}{w}}} 
 \gch_{\iota(\eta,w) \lambda} \QLS(\lambda).
\end{equation*}
We know from \cite[Sect.~5.1]{No} that 
the graded character $\lng(\gch_{w \lambda} \QLS(\lambda))$ 
twisted by the longest element $\lng$ of $W$ is identical 
to the graded character of the generalized Weyl module $W_{\lng w \lambda}$ introduced in \cite{FM}. 
Therefore, we have thus given a crystal-theoretic proof of \cite[Theorem~1.17]{FM}. 

This paper is organized as follows. In Section~\ref{sec:review}, 
we first fix our notation for affine Lie algebras. 
Next, we recall some basic facts about the (parabolic) semi-infinite Bruhat graph, 
and then briefly review fundamental results on semi-infinite LS paths. 
Also, we recall some basic facts about the (parabolic) quantum Bruhat graph, 
and then briefly review fundamental results on QLS paths, 
which includes the definition and some of the properties of the degree function.
In Section~\ref{sec:main}, we first state our main result (Theorem~\ref{thm:main}).
Next, we show a technical lemma about the quantum version of Deodhar lifts, 
which plays an important role in our proof of the main result.
Finally, by using similarity maps for semi-infinite LS paths and QLS paths, 
we prove Theorem~\ref{thm:main}.
%
%
\subsection*{Acknowledgments.}
S.N. was partially supported by 
JSPS Grant-in-Aid for Scientific Research (B) 16H03920. 
D.S. was partially supported by 
JSPS Grant-in-Aid for Scientific Research (C) 15K04803. 

%
\section{Semi-infinite Lakshmibai-Seshadri paths and quantum Lakshmibai-Seshadri paths.}
\label{sec:review}
%
%
\subsection{Affine Lie algebras.}
\label{subsec:liealg}

Let $\Fg$ be a finite-dimensional simple Lie algebra over $\BC$
with Cartan subalgebra $\Fh$. 
Denote by $\{ \alpha_{i}^{\vee} \}_{i \in I}$ and 
$\{ \alpha_{i} \}_{i \in I}$ the set of simple coroots and 
simple roots of $\Fg$, respectively, and set
%
$Q^{\vee} := \bigoplus_{i \in I} \BZ \alpha_i^{\vee}$ and 
$Q^{\vee,+} := \sum_{i \in I} \BZ_{\ge 0} \alpha_i^{\vee}$; 
for $\xi,\,\zeta \in Q^{\vee}$, we write $\xi \ge \zeta$ if $\xi-\zeta \in Q^{\vee,+}$. 
Let $\Delta$, $\Delta^{+}$, and $\Delta^{-}$ be 
the set of roots, positive roots, and negative roots of $\Fg$, respectively, 
with $\theta \in \Delta^{+}$ the highest root of $\Fg$. 
For a root $\alpha \in \Delta$, we denote by $\alpha^{\vee}$ its dual root. 
We set $\rho:=(1/2) \sum_{\alpha \in \Delta^{+}} \alpha$. 
Also, let $\vpi_{i}$, $i \in I$, denote the fundamental weights for $\Fg$, and set
%
%
\begin{equation} \label{eq:P-fin}
P:=\bigoplus_{i \in I} \BZ \vpi_{i}, \qquad 
P^{+} := \sum_{i \in I} \BZ_{\ge 0} \vpi_{i}. 
\end{equation} 

Let $\Fg_{\af} = \bigl(\BC[z,z^{-1}] \otimes \Fg\bigr) \oplus \BC c \oplus \BC d$ be 
the untwisted affine Lie algebra over $\BC$ associated to $\Fg$, 
where $c$ is the canonical central element, and $d$ is 
the scaling element (or the degree operator), 
with Cartan subalgebra $\Fh_{\af} = \Fh \oplus \BC c \oplus \BC d$. 
We regard an element $\mu \in \Fh^{\ast}:=\Hom_{\BC}(\Fh,\,\BC)$ as an element of 
$\Fh_{\af}^{\ast}$ by setting $\pair{\mu}{c}=\pair{\mu}{d}:=0$, where 
$\pair{\cdot\,}{\cdot}:\Fh_{\af}^{\ast} \times \Fh_{\af} \rightarrow \BC$ denotes
the canonical pairing of $\Fh_{\af}^{\ast}:=\Hom_{\BC}(\Fh_{\af},\,\BC)$ and $\Fh_{\af}$. 
Let $\{ \alpha_{i}^{\vee} \}_{i \in I_{\af}} \subset \Fh_{\af}$ and 
$\{ \alpha_{i} \}_{i \in I_{\af}} \subset \Fh_{\af}^{\ast}$ be the set of 
simple coroots and simple roots of $\Fg_{\af}$, respectively, 
where $I_{\af}:=I \sqcup \{0\}$; note that 
$\pair{\alpha_{i}}{c}=0$ and $\pair{\alpha_{i}}{d}=\delta_{i0}$ 
for $i \in I_{\af}$. 
Denote by $\delta \in \Fh_{\af}^{\ast}$ the null root of $\Fg_{\af}$; 
recall that $\alpha_{0}=\delta-\theta$. 
Also, let $\Lambda_{i} \in \Fh_{\af}^{\ast}$, $i \in I_{\af}$, 
denote the fundamental weights for $\Fg_{\af}$ such that $\pair{\Lambda_{i}}{d}=0$, 
and set 
%
%
\begin{equation} \label{eq:P}
P_{\af} := 
  \left(\bigoplus_{i \in I_{\af}} \BZ \Lambda_{i}\right) \oplus 
   \BZ \delta \subset \Fh^{\ast}, \qquad 
P_{\af}^{0}:=\bigl\{\mu \in P_{\af} \mid \pair{\mu}{c}=0\bigr\};
\end{equation}
notice that $P_{\af}^{0}=P \oplus \BZ \delta$, and that
$\pair{\mu}{\alpha_{0}^{\vee}} = - \pair{\mu}{\theta^{\vee}}$ 
for $\mu \in P_{\af}^{0}$. We remark that for each $i \in I$, 
$\vpi_{i}$ is equal to $\Lambda_{i}-\pair{\Lambda_{i}}{c}\Lambda_{0}$, 
which is called the level-zero fundamental weight in \cite{Kas02}. 

Let $W := \langle s_{i} \mid i \in I \rangle$ and 
$W_{\af} := \langle s_{i} \mid i \in I_{\af} \rangle$ be the (finite) Weyl group of $\Fg$ and 
the (affine) Weyl group of $\Fg_{\af}$, respectively, 
where $s_{i}$ is the simple reflection with respect to $\alpha_{i}$ 
for each $i \in I_{\af}$. We denote by $\ell:W_{\af} \rightarrow \BZ_{\ge 0}$ 
the length function on $W_{\af}$, whose restriction to $W$ agrees with 
the one on $W$, by $e \in W \subset W_{\af}$ the identity element, 
and by $\lng \in W$ the longest element. 
We set
%
%
\begin{equation} \label{eq:tis}
\ti{s}_{i}:=
 \begin{cases}
 s_{i} & \text{if $i \in I$}, \\[1mm]
 s_{\theta} & \text{if $i=0$},
 \end{cases}
\qquad
\ti{\alpha}_{i}:=
 \begin{cases}
 \alpha_{i} & \text{if $i \in I$}, \\[1mm]
 -\theta & \text{if $i=0$}. 
 \end{cases}
\end{equation}
For each $\xi \in Q^{\vee}$, let $t_{\xi} \in W_{\af}$ denote 
the translation in $\Fh_{\af}^{\ast}$ by $\xi$ (see \cite[Sect.~6.5]{Kac}); 
for $\xi \in Q^{\vee}$, we have 
%
%
\begin{equation}\label{eq:wtmu}
t_{\xi} \mu = \mu - \pair{\mu}{\xi}\delta \quad 
\text{if $\mu \in \Fh_{\af}^{\ast}$ satisfies $\pair{\mu}{c}=0$}.
\end{equation}
Then, $\bigl\{ t_{\xi} \mid \xi \in Q^{\vee} \bigr\}$ forms 
an abelian normal subgroup of $W_{\af}$, in which $t_{\xi} t_{\zeta} = t_{\xi + \zeta}$ 
holds for $\xi,\,\zeta \in Q^{\vee}$. Moreover, we know from \cite[Proposition 6.5]{Kac} that
\begin{equation*}
W_{\af} \cong 
 W \ltimes \bigl\{ t_{\xi} \mid \xi \in Q^{\vee} \bigr\} \cong W \ltimes Q^{\vee}. 
\end{equation*}

Denote by $\rr$ the set of real roots of $\Fg_{\af}$, and 
by $\prr \subset \rr$ the set of positive real roots; 
we know from \cite[Proposition 6.3]{Kac} that
$\rr = 
\bigl\{ \alpha + n \delta \mid \alpha \in \Delta,\, n \in \BZ \bigr\}$, 
and 
$\prr = 
\Delta^{+} \sqcup 
\bigl\{ \alpha + n \delta \mid \alpha \in \Delta,\, n \in \BZ_{> 0}\bigr\}$. 
For $\beta \in \rr$, we denote by $\beta^{\vee} \in \Fh_{\af}$ 
its dual root, and $s_{\beta} \in W_{\af}$ the corresponding reflection; 
if $\beta \in \rr$ is of the form $\beta = \alpha + n \delta$ 
with $\alpha \in \Delta$ and $n \in \BZ$, then 
$s_{\beta} =s_{\alpha} t_{n\alpha^{\vee}} \in W \ltimes Q^{\vee}$.
%
%
\subsection{Parabolic semi-infinite Bruhat graph.}
\label{subsec:SiBG}

In this subsection, we take and fix an arbitrary subset $\J \subset I$. We set 
$\QJ := \bigoplus_{i \in \J} \BZ \alpha_i$, 
$\QJv := \bigoplus_{i \in \J} \BZ \alpha_i^{\vee}$,  
$\QJvp := \sum_{i \in \J} \BZ_{\ge 0} \alpha_i^{\vee}$, 
$\DeJ := \Delta \cap \QJ$, 
$\DeJ^{+} := \Delta^{+} \cap \QJ$, 
$\WJ := \langle s_{i} \mid i \in \J \rangle$, and 
$\rho_{\J}:=(1/2) \sum_{\alpha \in \DeJ^{+}} \alpha$; 
we denote by
$[\,\cdot\,]^{\J} : Q^{\vee} \twoheadrightarrow Q_{\Jc}^{\vee}$
the projection from $Q^{\vee}=Q_{\Jc}^{\vee} \oplus \QJv$
onto $Q_{\Jc}^{\vee}$ with kernel $\QJv$. 
Let $\WJu$ denote the set of minimal(-length) coset representatives 
for the cosets in $W/\WJ$; we know from \cite[Sect.~2.4]{BB} that 
%
%
\begin{equation} \label{eq:mcr}
\WJu = \bigl\{ w \in W \mid 
\text{$w \alpha \in \Delta^{+}$ for all $\alpha \in \DeJ^{+}$}\bigr\}.
\end{equation}
For $w \in W$, we denote by $\mcr{w}=\mcr{w}^{\J} \in \WJu$ 
the minimal coset representative for the coset $w \WJ$ in $W/\WJ$.
Also, following \cite{Pet97} (see also \cite[Sect.~10]{LS10}), we set
\begin{align}
(\DeJ)_{\af} 
  & := \bigl\{ \alpha + n \delta \mid 
  \alpha \in \DeJ,\,n \in \BZ \bigr\} \subset \Delta_{\af}, \\
(\DeJ)_{\af}^{+}
  &:= (\DeJ)_{\af} \cap \prr = 
  \DeJ^{+} \sqcup \bigl\{ \alpha + n \delta \mid 
  \alpha \in \DeJ,\, n \in \BZ_{> 0} \bigr\}, \\
\label{eq:stabilizer}
(\WJ)_{\af} 
 & := \WJ \ltimes \bigl\{ t_{\xi} \mid \xi \in \QJv \bigr\}
   = \bigl\langle s_{\beta} \mid \beta \in (\DeJ)_{\af}^{+} \bigr\rangle, \\
\label{eq:Pet}
(\WJu)_{\af}
 &:= \bigl\{ x \in W_{\af} \mid 
 \text{$x\beta \in \prr$ for all $\beta \in (\DeJ)_{\af}^{+}$} \bigr\};
\end{align}
if $\J = \emptyset$, then 
$(W^{\emptyset})_{\af}=W_{\af}$ and $(W_{\emptyset})_{\af}=\bigl\{e\bigr\}$. 
We know from \cite{Pet97} (see also \cite[Lemma~10.6]{LS10}) that 
for each $x \in W_{\af}$, there exist a unique 
$x_1 \in (\WJu)_{\af}$ and a unique $x_2 \in (\WJ)_{\af}$ 
such that $x = x_1 x_2$; let 
%
%
\begin{equation} \label{eq:PiJ}
\PJ : W_{\af} \twoheadrightarrow (\WJu)_{\af}, \quad x \mapsto x_{1}, 
\end{equation}
denote the projection, 
where $x= x_1 x_2$ with $x_1 \in (\WJu)_{\af}$ and $x_2 \in (\WJ)_{\af}$. 
%
%
\begin{lem} \label{lem:PiJ}
\mbox{}
\begin{enu}
\item It holds that 
%
%
\begin{equation} \label{eq:PiJ2}
\begin{cases}
\PJ(w)=\mcr{w} 
  & \text{\rm for all $w \in W$}; \\[1mm]
\PJ(xt_{\xi})=\PJ(x)\PJ(t_{\xi}) 
  & \text{\rm for all $x \in W_{\af}$ and $\xi \in Q^{\vee}$};
\end{cases}
\end{equation}
in particular, $(\WJu)_{\af} 
  = \bigl\{ w \PJ(t_{\xi}) \mid w \in \WJu,\,\xi \in Q^{\vee} \bigr\}$.

\item For each $\xi \in Q^{\vee}$, 
the element $\PJ(t_{\xi}) \in (\WJu)_{\af}$ is 
of the form{\rm:} $\PJ(t_{\xi})=ut_{\xi+\xi_{1}}$ 
for some $u \in \WJ$ and $\xi_{1} \in \QJv$. 

\item For $\xi,\,\zeta \in Q^{\vee}$, 
$\PJ(t_{\xi}) = \PJ(t_{\zeta})$ if and only if $\xi-\zeta \in \QJv$.

\end{enu}
\end{lem}
\begin{proof}
Part (1) follows from \cite[Proposition~10.10]{LS10}, and 
part (2) follows from \cite[(3.7)]{LNSSS1}. 
The ``if'' part of part (3) is obvious by part (1) and 
the fact that $t_{\xi-\zeta} \in (\WJ)_{\af}$. 
The ``only if'' part of part (3) is obvious by part (2). 
\end{proof}
%
%
\begin{dfn} \label{dfn:sell}
Let $x \in W_{\af}$, and 
write it as $x = w t_{\xi}$ with $w \in W$ and $\xi \in Q^{\vee}$. 
We define the semi-infinite length $\sell(x)$ of $x$ by:
$\sell (x) = \ell (w) + 2 \pair{\rho}{\xi}$. 
\end{dfn}
%
%
\begin{dfn}[\cite{Lu80}, \cite{Lu97}; see also \cite{Pet97}] \label{dfn:SiB}
\mbox{}
\begin{enu}
\item The (parabolic) semi-infinite Bruhat graph $\SBJ$ 
is the $\prr$-labeled directed graph whose 
vertices are the elements of $(\WJu)_{\af}$, and 
whose directed edges are of the form: 
$x \edge{\beta} y$ for $x,y \in (\WJu)_{\af}$ and $\beta \in \prr$ 
such that $y=s_{\beta}x$ and $\sell (y) = \sell (x) + 1$. 
When $\J=\emptyset$, we write $\SB$ for 
$\mathrm{BG}^{\si}\bigl((W^{\emptyset})_{\af}\bigr)$. 

\item 
The (parabolic) semi-infinite Bruhat order is a partial order 
$\sile$ on $(\WJu)_{\af}$ defined as follows: 
for $x,\,y \in (\WJu)_{\af}$, we write $x \sile y$ 
if there exists a directed path in $\SBJ$ from $x$ to $y$; 
we write $x \sil y$ if $x \sile y$ and $x \ne y$. 
\end{enu}
\end{dfn}

\begin{rem}
In the case $\J = \emptyset$, the semi-infinite Bruhat order on $W_{\af}$ is
essentially the same as the generic Bruhat order introduced in \cite{Lu80}; 
see \cite[Appendix~A.3]{INS} for details. Also, for a general $\J$, 
the parabolic semi-infinite Bruhat order on $(\WJu)_{\af}$
is essentially the same as the partial order on $\J$-alcoves introduced in
\cite{Lu97} when we take a special point to be the origin.
\end{rem}
%
%
\begin{prop}[{\cite[Corollary~4.2.2]{INS}}] \label{prop:beta}
Let $x,y \in (\WJu)_{\af}$, and $\beta \in \prr$. 
If $x \edge{\beta} y$ is an edge of $\SBJ$, then 
$\beta$ is either of the following forms{\rm:}
$\beta = \alpha$ with $\alpha \in \Delta^{+}$, or 
$\beta = \alpha + \delta$ with $\alpha \in \Delta^{-}$. 
Moreover, if $x = w \PJ(t_{\xi})$ with $w \in \WJu$ and $\xi \in Q^{\vee}$, 
then $w^{-1}\alpha \in \Delta^{+} \setminus \DeJ$. 
\end{prop}
%
%
\begin{lem}[{\cite[Remark~4.1.3]{INS}}] \label{lem:si}
Let $x \in (\WJu)_{\af}$, and $i \in I_{\af}$. Then, 
\begin{equation} \label{eq:si1}
s_{i}x \in (\WJu)_{\af} \iff 
\pair{x\lambda}{\alpha_{i}^{\vee}} \ne 0 \iff 
x^{-1}\alpha_{i} \in (\Delta \setminus \DeJ)+\BZ\delta.
\end{equation}
Moreover, in this case, 
%
%
\begin{equation} \label{eq:simple}
\begin{cases}
x \edge{\alpha_{i}} s_{i}x \iff
\pair{x\lambda}{\alpha_{i}^{\vee}} > 0 \iff 
x^{-1}\alpha_{i}^{\vee} \in (\Delta^{+} \setminus \DeJ^{+})+\BZ\delta, & \\[1.5mm]
s_{i}x \edge{\alpha_{i}} x  \iff 
\pair{x\lambda}{\alpha_{i}^{\vee}} < 0 \iff 
x^{-1}\alpha_{i}^{\vee} \in (\Delta^{-} \setminus \DeJ^{-})+\BZ\delta. & 
\end{cases}
\end{equation}
\end{lem}
%
%
%
%
\subsection{Semi-infinite Lakshmibai-Seshadri paths.}
\label{subsec:SLS}

In this subsection, we fix $\lambda \in P^{+} \subset P_{\af}^{0} \subset P_{\af}$ 
(see \eqref{eq:P-fin} and \eqref{eq:P}), and set 
%
%
\begin{equation} \label{eq:J}
\J=\J_{\lambda}:= 
\bigl\{ i \in I \mid \pair{\lambda}{\alpha_i^{\vee}}=0 \bigr\} \subset I.
\end{equation}
%
%
\begin{dfn} \label{dfn:SBa}
For a rational number $0 < \sigma < 1$, 
we define $\SBa$ to be the subgraph of $\SBJ$ 
with the same vertex set but having only 
those directed edges of the form
$x \edge{\beta} y$ for which 
$\sigma\pair{x\lambda}{\beta^{\vee}} \in \BZ$ holds.
\end{dfn}
%
%
\begin{dfn}\label{dfn:SLS}
A semi-infinite Lakshmibai-Seshadri (LS for short) path of 
shape $\lambda $ is a pair 
%
%
\begin{equation} \label{eq:SLS}
\pi = (x_{1},\,\dots,\,x_{s} \,;\, 
       \sigma_{0},\,\sigma_{1},\,\dots,\,\sigma_{s}), \quad s \ge 1, 
\end{equation}
of a strictly decreasing sequence $x_1 \sig \cdots \sig x_s$ 
of elements in $(\WJu)_{\af}$ and an increasing sequence 
$0 = \sigma_0 < \sigma_1 < \cdots  < \sigma_s =1$ of rational numbers 
satisfying the condition that there exists a directed path 
in $\SBb{\sigma_{u}}$ from $x_{u+1}$ to  $x_{u}$ 
for each $u = 1,\,2,\,\dots,\,s-1$. 
\end{dfn}
%
%
\begin{rem} \label{rem:SLS}
We set
$\turn{\lambda}:=\bigl\{\sigma \in [0,1] \mid 
\sigma \pair{\lambda}{\alpha^{\vee}} \in \BZ 
\text{ for some $\alpha \in \Delta^{+} \setminus \DeJ^{+}$} \bigr\}$; 
note that $\turn{\lambda}$ is a finite set contained in $\BQ$. 
Let $0 < \sigma < 1$ be a rational number, and 
assume that there exists an edge $x \edge{\beta} y$ in $\SBa$; 
write $x \in \WJa$ and $\beta \in \prr$ 
as in Proposition~\ref{prop:beta}. 
Then we see that $\sigma\pair{\lambda}{w^{-1}\alpha^{\vee}} = 
\sigma\pair{w\lambda}{\alpha^{\vee}} = 
\sigma\pair{x\lambda}{\beta^{\vee}} \in \BZ$, 
which implies that $\sigma \in \turn{\lambda}$. 
Therefore, if $\pi \in \SLS(\lambda)$ is of the form \eqref{eq:SLS}, 
then $\sigma_{0},\,\sigma_{1},\,\dots,\,\sigma_{s} \in \turn{\lambda}$. 
\end{rem}

We denote by $\SLS(\lambda)$ 
the set of all semi-infinite LS paths of shape $\lambda$.
If $\pi \in \SLS(\lambda)$ is of the form \eqref{eq:SLS}, 
then we set $\iota(\pi):=x_{1} \in \WJa$ and $\kappa(\pi):=x_{s} \in \WJa$, 
and call them the initial and final directions of $\pi$, 
respectively. 

Following \cite[Sect.~3.1]{INS}, we endow the set $\SLS(\lambda)$ 
with a crystal structure with weights in $P_{\af}$ as follows. 
Let $\pi \in \SLS(\lambda)$ be of the form \eqref{eq:SLS}. 
We define $\ol{\pi}:[0,1] \rightarrow \BR \otimes_{\BZ} P_{\af}$ 
to be the piecewise-linear, continuous map 
whose ``direction vector'' for the interval 
$[\sigma_{u-1},\,\sigma_{u}]$ is $x_{u}\lambda \in P_{\af}$ 
for each $1 \le u \le s$, that is, 
%
%
\begin{equation} \label{eq:olpi}
\ol{\pi} (t) := 
\sum_{k = 1}^{u-1}(\sigma_{k} - \sigma_{k-1}) x_{k}\lambda + (t - \sigma_{u-1}) x_{u}\lambda
\quad
\text{for $t \in [\sigma_{u-1},\,\sigma_u]$, $1 \le u \le s$}. 
\end{equation}
We know from \cite[Proposition~3.1.3]{INS} that $\ol{\pi}$ 
is an (affine) LS path of shape $\lambda \in P_{\af}$, 
introduced in \cite[Sect.~4]{Lit95}. We set
%
%
\begin{equation} \label{eq:wt}
\wt (\pi):= \ol{\pi}(1) = \sum_{u = 1}^{s} (\sigma_{u}-\sigma_{u-1})x_{u}\lambda \in P_{\af}.
\end{equation}
We define root operators $e_{i}$, $f_{i}$, $i \in I_{\af}$, 
in the same manner as in \cite[Sect.~2]{Lit95}. Set 
%
%
\begin{equation} \label{eq:H}
\begin{cases}
H^{\pi}_{i}(t) := \pair{\ol{\pi}(t)}{\alpha_{i}^{\vee}} \quad 
\text{for $t \in [0,1]$}, \\[1.5mm]
m^{\pi}_{i} := 
 \min \bigl\{ H^{\pi}_{i} (t) \mid t \in [0,1] \bigr\}. 
\end{cases}
\end{equation}
As explained in \cite[Remark~2.4.3]{NS16}, 
all local minima of the function $H^{\pi}_{i}(t)$, $t \in [0,1]$, 
are integers; in particular, 
the minimum value $m^{\pi}_{i}$ is a nonpositive integer 
(recall that $\ol{\pi}(0)=0$, and hence $H^{\pi}_{i}(0)=0$).
We define $e_{i}\pi$ as follows. 
If $m^{\pi}_{i}=0$, then we set $e_{i} \pi := \bzero$, 
where $\bzero$ is an additional element not 
contained in any crystal. 
If $m^{\pi}_{i} \le -1$, then we set
%
%
\begin{equation} \label{eq:t-e}
\begin{cases}
t_{1} := 
  \min \bigl\{ t \in [0,\,1] \mid 
    H^{\pi}_{i}(t) = m^{\pi}_{i} \bigr\}, \\[1.5mm]
t_{0} := 
  \max \bigl\{ t \in [0,\,t_{1}] \mid 
    H^{\pi}_{i}(t) = m^{\pi}_{i} + 1 \bigr\}; 
\end{cases}
\end{equation}
note that $H^{\pi}_{i}(t)$ is 
strictly decreasing on the interval $[t_{0},\,t_{1}]$. 
Let $1 \le p \le q \le s$ be such that 
$\sigma_{p-1} \le t_{0} < \sigma_p$ and $t_{1} = \sigma_{q}$. 
Then we define $e_{i}\pi$ by
%
%
\begin{equation} \label{eq:epi}
\begin{split}
& e_{i} \pi := ( 
  x_{1},\,\ldots,\,x_{p},\,s_{i}x_{p},\,s_{i}x_{p+1},\,\ldots,\,
  s_{i}x_{q},\,x_{q+1},\,\ldots,\,x_{s} ; \\
& \hspace*{40mm}
  \sigma_{0},\,\ldots,\,\sigma_{p-1},\,t_{0},\,\sigma_{p},\,\ldots,\,\sigma_{q}=t_{1},\,
\ldots,\,\sigma_{s});
\end{split}
\end{equation}
if $t_{0} = \sigma_{p-1}$, then we drop $x_{p}$ and $\sigma_{p-1}$, and 
if $s_{i} x_{q} = x_{q+1}$, then we drop $x_{q+1}$ and $\sigma_{q}=t_{1}$.
Similarly, we define $f_{i}\pi$ as follows. 
Observe that $H^{\pi}_{i}(1) - m^{\pi}_{i}$ is a nonnegative integer. 
If $H^{\pi}_{i}(1) - m^{\pi}_{i} = 0$, then we set $f_{i} \pi := \bzero$. 
If $H^{\pi}_{i}(1) - m^{\pi}_{i}  \ge 1$, then we set
%
%
\begin{equation} \label{eq:t-f}
\begin{cases}
t_{0} := 
 \max \bigl\{ t \in [0,1] \mid H^{\pi}_{i}(t) = m^{\pi}_{i} \bigr\}, \\[1.5mm]
t_{1} := 
 \min \bigl\{ t \in [t_{0},\,1] \mid H^{\pi}_{i}(t) = m^{\pi}_{i} + 1 \bigr\};
\end{cases}
\end{equation}
note that $H^{\pi}_{i}(t)$ is 
strictly increasing on the interval $[t_{0},\,t_{1}]$. 
Let $0 \le p \le q \le s-1$ be such that $t_{0} = \sigma_{p}$ and 
$\sigma_{q} < t_{1} \le \sigma_{q+1}$. Then we define $f_{i}\pi$ by
%
%
\begin{equation} \label{eq:fpi}
\begin{split}
& f_{i} \pi := ( x_{1},\,\ldots,\,x_{p},\,s_{i}x_{p+1},\,\dots,\,
  s_{i} x_{q},\,s_{i} x_{q+1},\,x_{q+1},\,\ldots,\,x_{s} ; \\
& \hspace{40mm} 
  \sigma_{0},\,\ldots,\,\sigma_{p}=t_{0},\,\ldots,\,\sigma_{q},\,t_{1},\,
  \sigma_{q+1},\,\ldots,\,\sigma_{s});
\end{split}
\end{equation}
if $t_{1} = \sigma_{q+1}$, then we drop $x_{q+1}$ and $\sigma_{q+1}$, and 
if $x_{p} = s_{i} x_{p+1}$, then we drop $x_{p}$ and $\sigma_{p}=t_{0}$.
In addition, we set $e_{i} \bzero = f_{i} \bzero := \bzero$ 
for all $i \in I_{\af}$.
%
%
\begin{thm}[{see \cite[Theorem~3.1.5]{INS}}] \label{thm:SLS}
\mbox{}
\begin{enu}
\item The set $\SLS(\lambda) \sqcup \{ \bzero \}$ is 
stable under the action of the root operators 
$e_{i}$ and $f_{i}$, $i \in I_{\af}$.

\item For each $\pi \in \SLS(\lambda)$ 
and $i \in I_{\af}$, we set 
\begin{equation*}
\begin{cases}
\ve_{i} (\pi) := 
 \max \bigl\{ n \ge 0 \mid e_{i}^{n} \pi \neq \bzero \bigr\}, \\[1.5mm]
\vp_{i} (\pi) := 
 \max \bigl\{ n \ge 0 \mid f_{i}^{n} \pi \neq \bzero \bigr\}.
\end{cases}
\end{equation*}
Then, the set $\SLS(\lambda)$, 
equipped with the maps $\wt$, $e_{i}$, $f_{i}$, $i \in I_{\af}$, 
and $\ve_{i}$, $\vp_{i}$, $i \in I_{\af}$, 
defined above, is a crystal with weights in $P_{\af}$.
\end{enu}
\end{thm}

We denote by $\SLS_{0}(\lambda)$ the connected component of 
$\SLS(\lambda)$ containing $\pi_{\lambda}:=(e;0,1)$. 
Also, for $x \in (\WJu)_{\af}$, 
we set $\pi_{\lambda}^{x}:=(x;0,1) \in \SLS(\lambda)$.

Recall from \cite[Theorem~3.2.1]{INS} that 
$\SLS(\lambda)$ is isomorphic as a crystal 
(with weights in $P_{\af}$) to the crystal basis of 
the extremal weight module of extremal weight $\lambda$. 
Hence we deduce from \cite[Sect.~7]{Kas94} that 
the affine Weyl group $W_{\af}$ acts on $\SLS(\lambda)$ by
%
%
\begin{equation} \label{eq:W1}
s_{i} \cdot \pi:=
\begin{cases}
f_{i}^{n}\pi & \text{if $n:=\pair{\wt(\pi)}{\alpha_{i}^{\vee}} \ge 0$}, \\[1.5mm]
e_{i}^{-n}\pi & \text{if $n:=\pair{\wt(\pi)}{\alpha_{i}^{\vee}} \le 0$}, 
\end{cases}
\end{equation}
for $\pi \in \SLS(\lambda)$ and $i \in I_{\af}$; 
by convention, we set $w \cdot \bzero:=\bzero$ 
for all $w \in W_{\af}$.
%
%
\begin{lem}[{see \cite[(5.1.6)]{INS}}] \label{lem:pix}
For every $x \in W_{\af}$, it holds that 
$x \cdot \pi_{\lambda} = (\PJ(x);0,1) = \pi_{\lambda}^{\PJ(x)}$. 
In particular, $\pi_{\lambda}^{x} \in \SLS_{0}(\lambda)$ 
for all $x \in (\WJu)_{\af}$. 
\end{lem}

Let $\pi=(x_{1},\,\dots,\,x_{s} \,;\, 
\sigma_{0},\,\sigma_{1},\,\dots,\,\sigma_{s}) \in \SLS(\lambda)$. 
For $\xi \in Q^{\vee}$, we set
%
%
\begin{equation} \label{eq:Txi}
\pi \cdot T_{\xi} := (x_{1}\PJ(t_{\xi}),\,\dots,\,x_{s}\PJ(t_{\xi}) \,;\, 
        \sigma_{0},\,\sigma_{1},\,\dots,\,\sigma_{s});
\end{equation}
by convention, we set $\bzero \cdot T_{\xi}:=\bzero$. 
Then we see from the definition of semi-infinite LS paths that 
$\pi \cdot T_{\xi} \in \SLS(\lambda)$; note that $x_{u}\PJ(t_{\xi})=
\PJ(x_{u})\PJ(t_{\xi}) = \PJ(x_{u}t_{\xi}) \in \WJa$ by \eqref{eq:PiJ2}. 
The following lemma can be easily shown from the definitions. 
%
%
\begin{lem} \label{lem:Txi}
Let $\pi \in \SLS(\lambda)$, $\xi \in Q^{\vee}$, and $i \in I_{\af}$. 
Then, 
%
%
\begin{equation} \label{eq:Txi1}
e_{i}(\pi \cdot T_{\xi}) = (e_{i}\pi) \cdot T_{\xi}, \qquad
f_{i}(\pi \cdot T_{\xi}) = (f_{i}\pi) \cdot T_{\xi}, 
\end{equation}
%
%
\begin{equation} \label{eq:Txi1a}
\begin{cases}
\wt (\pi \cdot T_{\xi}) = \wt(\pi)-\pair{\lambda}{\xi}\delta, 
\ \text{\rm and hence} \ 
\pair{\wt (\pi \cdot T_{\xi})}{\alpha_{i}^{\vee}} = 
\pair{\wt(\pi)}{\alpha_{i}^{\vee}}, \\
\ve_{i}(\pi \cdot T_{\xi}) = \ve_{i}(\pi), \quad
\vp_{i}(\pi \cdot T_{\xi}) = \vp_{i}(\pi).
\end{cases}
\end{equation}
\end{lem}
%
%
\begin{rem} \label{rem:Txi}
Let $\pi \in \SLS_{0}(\lambda)$, and write it as
$\pi=X\pi_{\lambda}$ for some monomial $X$ in root operators. Then, 
$\pi \cdot T_{\xi} = (X\pi_{\lambda}) \cdot T_{\xi} 
\stackrel{\eqref{eq:Txi1}}{=} X (\pi_{\lambda} \cdot T_{\xi}) = 
X \pi_{\lambda}^{\PJ(t_{\xi})} \in \SLS_{0}(\lambda)$. 
\end{rem}

Now, for $\lambda_{1},\dots,\,\lambda_{n} \in P^{+}$, 
we denote by $(\SLS(\lambda_{1}) \otimes \cdots \otimes \SLS(\lambda_{n}))_{0}$
the connected component of $\SLS(\lambda_{1}) \otimes \cdots \otimes \SLS(\lambda_{n})$
containing $\pi_{\lambda_{1}} \otimes \cdots \otimes \pi_{\lambda_{n}}$. 
The next lemma follows from \cite[Theorem~3.1]{KNS}. 
%
%
\begin{thm} \label{thm:SMT}
Let $\lambda_{1},\dots,\,\lambda_{n} \in P^{+}$, 
and set $\lambda:=\lambda_{1}+\cdots+\lambda_{n} \in P^{+}$. 
There exists an embedding 
$\Xi=\Xi_{\lambda_{1},\dots,\lambda_{n}}:
\SLS(\lambda) \hookrightarrow 
\SLS(\lambda_{1}) \otimes \cdots \otimes \SLS(\lambda_{n})$
of crystals that sends $\pi_{\lambda}$ to 
$\pi_{\lambda_{1}} \otimes \cdots \otimes \pi_{\lambda_{n}}$. 
In particular, the restriction of $\Xi=\Xi_{\lambda_{1},\dots,\lambda_{n}}$
to the connected component $\SLS_{0}(\lambda)$ is 
a {\rm(}unique{\rm)} isomorphism 
\begin{equation} \label{eq:Xi0}
\Xi=\Xi_{\lambda_{1},\dots,\lambda_{n}}:
\SLS_{0}(\lambda) \stackrel{\sim}{\rightarrow}
(\SLS(\lambda_{1}) \otimes \cdots \otimes \SLS(\lambda_{n}))_{0}
\end{equation}
of crystals that sends $\pi_{\lambda}$ to 
$\pi_{\lambda_{1}} \otimes \cdots \otimes \pi_{\lambda_{n}}$. 
\end{thm}
%
%
\begin{rem} \label{rem:ass1}
Let $\lambda,\mu,\nu \in P^{+}$. We see that 
the following diagram is commutative: 
%
%
\begin{equation} \label{eq:ass1a}
\begin{split}
\xymatrix{
\SLS_{0}(\lambda+\mu+\nu) 
 \ar[rr]^-{\Xi_{\lambda+\mu,\nu}}
 \ar[d]_{\Xi_{\lambda,\mu+\nu}} & &
(\SLS(\lambda+\mu) \otimes \SLS(\nu))_{0} \ar[d]^{\Xi_{\lambda\mu} \otimes \id} \\
(\SLS(\lambda) \otimes \SLS(\mu+\nu))_{0}
 \ar[rr]^-{\id \otimes \Xi_{\mu\nu}} & &
(\SLS(\lambda) \otimes \SLS(\mu) \otimes \SLS(\nu))_{0}.
}
\end{split}
\end{equation}
Indeed, both of the maps $(\id \otimes \Xi_{\mu\nu}) \circ 
\Xi_{\lambda,\mu+\nu}$ and 
$(\Xi_{\lambda\mu} \otimes \id) \circ
\Xi_{\lambda+\mu,\nu}$ are isomorphisms of crystals that 
send the element $\pi_{\lambda+\mu+\nu} \in \SLS_{0}(\lambda+\mu+\nu)$ 
to $\pi_{\lambda} \otimes \pi_{\mu} \otimes \pi_{\nu} \in 
(\SLS(\lambda) \otimes \SLS(\mu) \otimes \SLS(\nu))_{0}$.
Because $\SLS_{0}(\lambda+\mu+\nu) \subset \SLS(\lambda+\mu+\nu)$ 
is connected and contains $\pi_{\lambda+\mu+\nu}$, 
we conclude that the diagram above is commutative. 
\end{rem}

Keep the setting of Theorem~\ref{thm:SMT}.
Take $\J_{k}=\J_{\lambda_{k}}$, $1 \le k \le n$ and 
$\J=\J_{\lambda}$ as in \eqref{eq:J}.
%
%
\begin{lem}[{see \cite[Remark~3.5.2]{NS16}}] \label{lem:ext}
It holds that $\Xi(x \cdot \pi_{\lambda})=
(x \cdot \pi_{\lambda_{1}}) \otimes \cdots \otimes (x \cdot \pi_{\lambda_{n}})$
for $x \in W_{\af}$. 
\end{lem}

For $\pi_{1} \otimes \cdots \otimes \pi_{n} 
\in \SLS(\lambda_{1}) \otimes \cdots \otimes \SLS(\lambda_{n})$ and $\xi \in Q^{\vee}$, 
we set
%
%
\begin{equation} \label{eq:Txiten}
(\pi_{1} \otimes \cdots \otimes \pi_{n}) \cdot T_{\xi}:=
(\pi_{1} \cdot T_{\xi}) \otimes \cdots \otimes 
(\pi_{n} \cdot T_{\xi}).
\end{equation}
%
%
\begin{lem} \label{lem:Txi2}
It holds that $\Xi(\pi \cdot T_{\xi}) = \Xi(\pi) \cdot T_{\xi}$ 
for all $\pi \in \SLS_{0}(\lambda)$ and $\xi \in Q^{\vee}$. 
\end{lem}


\begin{proof}
Let $\pi \in \SLS_{0}(\lambda)$, and write it as $\pi=X\pi_{\lambda}$ 
for some monomial $X$ in root operators. Assume that 
$\Xi(\pi)=\Xi(X\pi_{\lambda}) = X\Xi(\pi_{\lambda}) = 
X (\pi_{\lambda_{1}} \otimes \cdots \otimes \pi_{\lambda_{n}}) = 
X_{1}\pi_{\lambda_{1}} \otimes \cdots \otimes X_{n}\pi_{\lambda_{n}}$
for some monomials $X_{k}$, $1 \le k \le n$, in root operators 
(which are ``submonomials'' of $X$ by the tensor product rule for crystals). 
Then we have 
\begin{align} 
\Xi(\pi) \cdot T_{\xi} & = 
(X_{1}\pi_{\lambda_{1}} \cdot T_{\xi}) \otimes \cdots \otimes 
(X_{n}\pi_{\lambda_{n}} \cdot T_{\xi}) \nonumber \\ 
& = 
(X_{1}\pi_{\lambda_{1}}^{\PS{1}(t_{\xi})}) \otimes \cdots \otimes 
(X_{n}\pi_{\lambda_{n}}^{\PS{n}(t_{\xi})}) \quad 
 \text{by Remark~\ref{rem:Txi}}.  \label{eq:txi2a}
\end{align}
Here, by Remark~\ref{rem:Txi}, we see that 
$\pi \cdot T_{\xi} = X \pi_{\lambda}^{\PJ(t_{\xi})}$. 
Therefore, it follows from Lemma~\ref{lem:ext} that 
$\Xi(\pi \cdot T_{\xi})=\Xi(X \pi_{\lambda}^{\PJ(t_{\xi})}) = 
X \Xi(\pi_{\lambda}^{\PJ(t_{\xi})}) = 
X(\pi_{\lambda_{1}}^{\PS{1}(t_{\xi})} 
\otimes \cdots \otimes 
\pi_{\lambda_{n}}^{\PS{n}(t_{\xi})})$. 
Hence, the tensor product rule for crystals, 
along with \eqref{eq:Txi1a}, shows that 
\begin{equation*}
\Xi(\pi \cdot T_{\xi})= 
X(\pi_{\lambda_{1}}^{\PS{1}(t_{\xi})} 
\otimes \cdots \otimes 
\pi_{\lambda_{n}}^{\PS{n}(t_{\xi})}) = 
(X_{1}\pi_{\lambda_{1}}^{\PS{1}(t_{\xi})}) \otimes \cdots \otimes 
(X_{n}\pi_{\lambda_{n}}^{\PS{n}(t_{\xi})}).
\end{equation*}
By this equality and \eqref{eq:txi2a}, 
we obtain $\Xi(\pi \cdot T_{\xi}) = \Xi(\pi) \cdot T_{\xi}$, as desired. 
\end{proof}

%
\subsection{Parabolic quantum Bruhat graph and the tilted Bruhat order.}
\label{subsec:QBG}

In this subsection, we take and fix a subset $\J$ of $I$. 

\begin{dfn}
The (parabolic) quantum Bruhat graph $\QBJ$ is 
the ($\Delta^{+} \setminus \DeJ^{+})$-labeled
directed graph whose vertices are the elements of $\WJu$, and 
whose directed edges are of the form: $w \edge{\beta} v$ 
for $w,v \in \WJu$ and $\beta \in \Delta^{+} \setminus \DeJ^{+}$ 
such that $v= \mcr{ws_{\beta}}$, and such that either of 
the following holds: 
(i) $\ell(v) = \ell (w) + 1$; 
(ii) $\ell(v) = \ell (w) + 1 - 2 \pair{\rho-\rho_{\J}}{\beta^{\vee}}$.
An edge satisfying (i) (resp., (ii)) is called a Bruhat (resp., quantum) edge. 
When $\J=\emptyset$, we write $\QB$ for $\mathrm{QBG}(W^{\emptyset})$. 
\end{dfn}
%
%
\begin{rem} \label{rem:PQBG}
We know from \cite[Remark~6.13]{LNSSS1} that for each $w,\,v \in \WJu$, 
there exists a directed path in $\QBJ$ from $w$ to $v$.
\end{rem}

Let $w,\,v \in \WJu$, and let 
$\bp:w=
 v_{0} \edge{\beta_{1}}
 v_{1} \edge{\beta_{2}} \cdots 
       \edge{\beta_{s}}
 v_{s}=v$
be a directed path in $\QBJ$ from $w$ to $v$. 
Then we define the weight of $\bp$ by
%
%
\begin{equation} \label{eq:wtdp}
\wt^{\J}(\bp) := \sum_{
 \begin{subarray}{c}
 1 \le r \le s\,; \\[1mm]
 \text{$v_{r-1} \edge{\beta_{r}} v_{r}$ is} \\[1mm]
 \text{a quantum edge}
 \end{subarray}}
\beta_{r}^{\vee} \in Q^{\vee,+}; 
\end{equation}
when $\J=\emptyset$, we write $\wt(\bp)$ for $\wt^{\emptyset}(\bp)$. 
We know the following proposition from 
\cite[Proposition~8.1 and its proof]{LNSSS1}.
%
%
\begin{prop} \label{prop:81}
Let $w,\,v \in \WJu$. 
Let $\bp$ be a shortest directed path in $\QBJ$ from $w$ to $v$, 
and $\bq$ an arbitrary directed path in $\QBJ$ from $w$ to $v$. 
Then, $[\wt^{\J}(\bq)-\wt^{\J}(\bp)]^{\J} \in Q_{\Jc}^{\vee,+}$, 
where $[\,\cdot\,]^{\J} : Q^{\vee} \twoheadrightarrow Q_{\Jc}^{\vee}$ 
is as defined in Section~\ref{subsec:SiBG}.
Moreover, $\bq$ is shortest if and only if 
$[\wt^{\J}(\bq)]^{\J}=[\wt^{\J}(\bp)]^{\J}$.
\end{prop}

For $w,\,v \in \WJu$, we take a shortest directed path $\bp$ in 
$\QBJ$ from $w$ to $v$, and set 
$\wt^{\J}(w \Rightarrow v):=[\wt^{\J}(\bp)]^{\J} \in Q_{\Jc}^{\vee,+}$. 
When $\J=\emptyset$, we write $\wt(u \Rightarrow v)$ 
for $\wt^{\emptyset}(u \Rightarrow v)$. 
%
%
\begin{lem}[{\cite[Lemma~7.2]{LNSSS2}}] \label{lem:wtS} \mbox{}
Let $w,\,v \in \WJu$, and let $w_{1} \in w\WJ$, $v_{1} \in v\WJ$. 
Then we have $\wt^{\J}(w \Rightarrow v) = [\wt(w_{1} \Rightarrow v_{1})]^{S}$. 
\end{lem}

For $w,\,v \in W$, we denote by $\ell^{\J}(w \Rightarrow v)$ 
the length of a shortest directed path from $w$ to $v$ in $\QBJ$; 
when $\J=\emptyset$, we write $\ell(w \Rightarrow v)$ 
for $\ell^{\emptyset}(w \Rightarrow v)$. 
%
%
\begin{dfn}[tilted Bruhat order; see \cite{BFP}] \label{dfn:tilted}
For each $w \in W$, we define the $w$-tilted Bruhat order $\tb{w}$ on $W$ as follows:
for $v_{1},v_{2} \in W$, 
%
%
\begin{equation} \label{eq:tilted}
v_{1} \le_{w} v_{2} \iff \ell(w \Rightarrow v_{2}) = 
 \ell(w \Rightarrow v_{1}) + \ell(v_{1} \Rightarrow v_{2}).
\end{equation}
Namely, $v_{1} \le_{w} v_{2}$ if and only if 
there exists a shortest directed path in $\QB$ 
from $w$ to $v_{2}$ passing through $v_{1}$; 
or equivalently, the concatenation of a shortest directed path 
from $w$ to $v_{1}$ and one from $v_{1}$ to $v_{2}$ 
is one from $w$ to $v_{2}$. 
\end{dfn}
%
%
\begin{prop}[{\cite[Theorem~7.1]{LNSSS1}}] \label{prop:tbmin}
Let $w \in W$, and let $\J$ be a subset of $I$. 
Then each coset $v\WJ$, $v \in W$, has a unique minimal element
with respect to $\tb{w}$; 
we denote it by $\tbmin{v}{\J}{w}$. 
\end{prop}
%
%
\subsection{Quantum Lakshmibai-Seshadri paths and the degree function.}
\label{subsec:QLS}
We fix $\lambda \in P^{+}$, and take $\J=\J_{\lambda}$ as in \eqref{eq:J}. 
%
%
\begin{dfn} \label{dfn:QBa}
For a rational number $0 < \sigma < 1$, 
we define $\QBa$ to be the subgraph of $\QBJ$ 
with the same vertex set but having only those directed edges of the form
$w \edge{\beta} v$ for which 
$\sigma\pair{\lambda}{\beta^{\vee}} \in \BZ$ holds.
\end{dfn}
%
%
\begin{dfn}\label{dfn:QLS}
A quantum LS path of shape $\lambda$ is a pair 
%
%
\begin{equation} \label{eq:QLS}
\eta = (v_{1},\,\dots,\,v_{s} \,;\, \sigma_{0},\,\sigma_{1},\,\dots,\,\sigma_{s}), \quad s \ge 1, 
\end{equation}
of a sequence $v_{1},\,\dots,\,v_{s}$ 
of elements in $\WJu$ with $v_{u} \ne v_{u+1}$ 
for any $1 \le u \le s-1$ and an increasing sequence 
$0 = \sigma_0 < \sigma_1 < \cdots  < \sigma_s =1$ of rational numbers 
satisfying the condition that there exists a directed path 
in $\QBb{\sigma_{u}}$ from $v_{u+1}$ to  $v_{u}$ 
for each $u = 1,\,2,\,\dots,\,s-1$. 
\end{dfn}
%
%
\begin{rem} \label{rem:QLS}
It is obvious by the definition that if $\eta \in \QLS(\lambda)$ 
is of the form \eqref{eq:QLS}, then 
$\sigma_{0},\,\sigma_{1},\,\dots,\,\sigma_{s} \in \turn{\lambda}$, 
where $\turn{\lambda}$ is as defined in Remark~\ref{rem:SLS}. 
\end{rem}

We denote by $\QLS(\lambda)$ 
the set of all quantum LS paths of shape $\lambda$.
We set $\eta_{\lambda}:=(e;0,1) \in \QLS(\lambda)$, and 
set $\eta_{\lambda}^{v}:=(v;0,1) \in \QLS(\lambda)$ 
for $v \in \WJu$. 
Also, if $\eta \in \QLS(\lambda)$ is of the form \eqref{eq:QLS}, 
then we set $\iota(\eta):=v_{1} \in \WJu$ and 
$\kappa(\eta):=v_{s} \in \WJu$, 
and call them the initial and final directions of $\eta$, 
respectively. 
We identify $\eta \in \QLS(\lambda)$ of the form \eqref{eq:QLS} 
with the piecewise-linear, continuous map 
$\eta:[0,1] \rightarrow \BR \otimes_{\BZ} P$ 
whose ``direction vector'' for the interval 
$[\sigma_{u-1},\,\sigma_{u}]$ is $v_{u}\lambda \in P$ 
for each $1 \le u \le s$, that is, 
%
%
\begin{equation} \label{eq:eta}
\eta (t) := 
\sum_{k = 1}^{u-1}(\sigma_{k} - \sigma_{k-1}) v_{k}\lambda + (t - \sigma_{u-1}) v_{u}\lambda
\quad
\text{for $t \in [\sigma_{u-1},\,\sigma_u]$, $1 \le u \le s$};
\end{equation}
note that $\BR \otimes_{\BZ} P \cong 
(\BR \otimes_{\BZ} P_{\af}^{0})/\BR\delta \subset 
(\BR \otimes_{\BZ} P_{\af})/\BR\delta$. 
%
%
\begin{rem} \label{rem:LScl}
We know from \cite[Theorem~3.3]{LNSSS2} that 
the set $\QLS(\lambda)$ of quantum LS paths of shape $\lambda$ is identical
(as a set of piecewise-linear, continuous map 
from $[0,1]$ to $(\BR \otimes_{\BZ} P_{\af})/\BR\delta$)
to the set $\BB(\lambda)_{\cl}$ of 
(affine) LS paths of shape $\lambda$ projected by 
$\cl:\BR \otimes_{\BZ} P_{\af} \twoheadrightarrow 
(\BR \otimes_{\BZ} P_{\af})/\BR\delta$, 
which we studied in \cite{NS03}, \cite{NS05}, \cite{NS06}. 
\end{rem}

We endow the set $\QLS(\lambda)$ with a crystal structure 
with weights in $P \cong P_{\af}^{0}/\BZ\delta \subset P_{\af}/\BZ\delta$ as follows
(see \cite[\S4.2]{LNSSS16}). Let $\eta \in \QLS(\lambda)$ 
be of the form \eqref{eq:QLS}. We set
%
%
\begin{equation} \label{eq:wteta}
\wt (\eta):= \eta(1) = \sum_{u = 1}^{s} (\sigma_{u}-\sigma_{u-1})v_{u}\lambda \in P. 
\end{equation}
We define root operators $e_{i}$, $f_{i}$, $i \in I_{\af}$, 
in the same manner as in \cite[Sect.~2]{Lit95}. Set
%
%
\begin{equation} \label{eq:H2}
\begin{cases}
H^{\eta}_{i}(t) := \pair{\eta(t)}{\alpha_{i}^{\vee}} \quad 
\text{for $t \in [0,1]$}, \\[1.5mm]
m^{\eta}_{i} := 
 \min \bigl\{ H^{\eta}_{i} (t) \mid t \in [0,1] \bigr\}. 
\end{cases}
\end{equation}
Since $\QLS(\lambda) = \BB(\lambda)_{\cl}$, 
it follows from \cite[Lemma~4.5\,d)]{Lit95} 
(see also \cite[Proposition~4.1.12]{LNSSS16}) that 
all local minima of the function $H^{\eta}_{i}(t)$, $t \in [0,1]$, 
are integers; in particular, 
the minimum value $m^{\eta}_{i}$ is a nonpositive integer 
(recall that $\eta(0)=0$, and hence $H^{\eta}_{i}(0)=0$).
We define $e_{i}\eta$ as follows. 
If $m^{\eta}_{i}=0$, then we set $e_{i} \eta := \bzero$. 
If $m^{\eta}_{i} \le -1$, then we set
%
%
\begin{equation} \label{eq:t-e2}
\begin{cases}
t_{1} := 
  \min \bigl\{ t \in [0,\,1] \mid 
    H^{\eta}_{i}(t) = m^{\eta}_{i} \bigr\}, \\[1.5mm]
t_{0} := 
  \max \bigl\{ t \in [0,\,t_{1}] \mid 
    H^{\eta}_{i}(t) = m^{\eta}_{i} + 1 \bigr\}; 
\end{cases}
\end{equation}
note that $H^{\eta}_{i}(t)$ is 
strictly decreasing on the interval $[t_{0},\,t_{1}]$. 
Let $1 \le p \le q \le s$ be such that 
$\sigma_{p-1} \le t_{0} < \sigma_p$ and $t_{1} = \sigma_{q}$. 
Then we define $e_{i}\eta$ by
%
%
\begin{equation} \label{eq:eeta}
\begin{split}
& e_{i} \eta := ( 
  v_{1},\,\ldots,\,v_{p},\,\mcr{\ti{s}_{i}v_{p}},\,\mcr{\ti{s}_{i}v_{p+1}},\,\ldots,\,
  \mcr{\ti{s}_{i}v_{q}},\,v_{q+1},\,\ldots,\,v_{s} ; \\
& \hspace*{40mm}
  \sigma_{0},\,\ldots,\,\sigma_{p-1},\,t_{0},\,\sigma_{p},\,\ldots,\,\sigma_{q}=t_{1},\,
\ldots,\,\sigma_{s}), 
\end{split}
\end{equation}
where $\ti{s}_{i}$ is as in \eqref{eq:tis}; 
if $t_{0} = \sigma_{p-1}$, then we drop $v_{p}$ and $\sigma_{p-1}$, and 
if $\ti{s}_{i} v_{q} = v_{q+1}$, then we drop $v_{q+1}$ and $\sigma_{q}=t_{1}$.
Similarly, we define $f_{i}\eta$ as follows. 
Observe that $H^{\eta}_{i}(1) - m^{\eta}_{i}$ is a nonnegative integer. 
If $H^{\eta}_{i}(1) - m^{\eta}_{i} = 0$, then we set $f_{i} \eta := \bzero$. 
If $H^{\eta}_{i}(1) - m^{\eta}_{i}  \ge 1$, then we set
%
%
\begin{equation} \label{eq:t-f2}
\begin{cases}
t_{0} := 
 \max \bigl\{ t \in [0,1] \mid H^{\eta}_{i}(t) = m^{\eta}_{i} \bigr\}, \\[1.5mm]
t_{1} := 
 \min \bigl\{ t \in [t_{0},\,1] \mid H^{\eta}_{i}(t) = m^{\eta}_{i} + 1 \bigr\};
\end{cases}
\end{equation}
note that $H^{\eta}_{i}(t)$ is 
strictly increasing on the interval $[t_{0},\,t_{1}]$. 
Let $0 \le p \le q \le s-1$ be such that $t_{0} = \sigma_{p}$ and 
$\sigma_{q} < t_{1} \le \sigma_{q+1}$. Then we define $f_{i}\eta$ by
%
%
\begin{equation} \label{eq:feta}
\begin{split}
& f_{i} \eta := ( v_{1},\,\ldots,\,v_{p},\,\mcr{\ti{s}_{i}v_{p+1}},\,\dots,\,
  \mcr{\ti{s}_{i} v_{q}},\,\mcr{\ti{s}_{i} v_{q+1}},\,v_{q+1},\,\ldots,\,v_{s} ; \\
& \hspace{40mm} 
  \sigma_{0},\,\ldots,\,\sigma_{p}=t_{0},\,\ldots,\,\sigma_{q},\,t_{1},\,
  \sigma_{q+1},\,\ldots,\,\sigma_{s});
\end{split}
\end{equation}
if $t_{1} = \sigma_{q+1}$, then we drop $v_{q+1}$ and $\sigma_{q+1}$, and 
if $v_{p} = \ti{s}_{i} v_{p+1}$, then we drop $v_{p}$ and $\sigma_{p}=t_{0}$.
In addition, we set $e_{i} \bzero = f_{i} \bzero := \bzero$ for all $i \in I_{\af}$.
%
%
\begin{thm} \label{thm:QLS}
The set $\QLS(\lambda) \sqcup \{ \bzero \}$ is 
stable under the action of the root operators 
$e_{i}$ and $f_{i}$, $i \in I_{\af}$.
Moreover, if we set 
\begin{equation*}
\begin{cases}
\ve_{i} (\eta) := 
 \max \bigl\{ n \ge 0 \mid e_{i}^{n} \eta \neq \bzero \bigr\}, \\[1.5mm]
\vp_{i} (\eta) := 
 \max \bigl\{ n \ge 0 \mid f_{i}^{n} \eta \neq \bzero \bigr\}
\end{cases}
\end{equation*}
for $\eta \in \QLS(\lambda)$ and $i \in I_{\af}$, then 
the set $\QLS(\lambda)$, 
equipped with the maps $\wt$, $e_{i}$, $f_{i}$, $i \in I_{\af}$, 
and $\ve_{i}$, $\vp_{i}$, $i \in I_{\af}$, 
defined above, is a crystal with weights in 
$P \cong P_{\af}^{0}/\BZ\delta \subset P_{\af}/\BZ\delta$.
\end{thm}

The next lemma follows from 
\cite[Theorem~3.2 and Proposition~3.23]{NS05}. 
%
%
\begin{thm} \label{thm:NS05} \mbox{}
\begin{enu}
\item For every $\lambda \in P^{+}$, 
the crystal $\QLS(\lambda)$ is connected. 

\item Let $\lambda_{1},\dots,\,\lambda_{n} \in P^{+}$, 
and set $\lambda:=\lambda_{1}+\cdots+\lambda_{n}$. 
There exists a {\rm(}unique{\rm)} isomorphism
%
%
\begin{equation} \label{eq:Theta}
\Theta=\Theta_{\lambda_{1},\dots,\lambda_{n}}:
\QLS(\lambda) \stackrel{\sim}{\rightarrow}
\QLS(\lambda_{1}) \otimes \cdots \otimes \QLS(\lambda_{n})
\end{equation}
of crystals that sends $\eta_{\lambda}$ to 
$\eta_{\lambda_{1}} \otimes \cdots \otimes \eta_{\lambda_{n}}$. 
\end{enu}
\end{thm}
%
%
\begin{rem} \label{rem:ass2}
Let $\lambda,\mu,\nu \in P^{+}$. 
By the same reasoning as in Remark~\ref{rem:ass1}, 
we see that the following diagram is commutative: 
%
%
\begin{equation} \label{eq:ass2a}
\begin{split}
\xymatrix{
\QLS(\lambda+\mu+\nu) 
 \ar[rr]^-{\Theta_{\lambda+\mu,\nu}}
 \ar[d]_{\Theta_{\lambda,\mu+\nu}} & &
\QLS(\lambda+\mu) \otimes \QLS(\nu) \ar[d]^{\Theta_{\lambda\mu} \otimes \id} \\
\QLS(\lambda) \otimes \QLS(\mu+\nu)
 \ar[rr]^-{\id \otimes \Theta_{\mu\nu}} & &
\QLS(\lambda) \otimes \QLS(\mu) \otimes \QLS(\nu). 
}
\end{split}
\end{equation}
\end{rem}

Now, we define a projection $\cl : (\WJu)_{\af} \twoheadrightarrow \WJu$ by
$\cl (x) := w$ for $x \in (\WJu)_{\af}$ of the form 
$x = w\PJ(t_{\xi})$ with $w \in \WJu$ and $\xi \in Q^{\vee}$.
For $\pi = (x_{1},\,\dots,\,x_{s}\,;\,\sigma_{0},\,\sigma_{1},\,\dots,\,\sigma_{s}) 
\in \SLS(\lambda)$, we define 
\begin{equation} \label{eq:clpi}
\cl(\pi):=(\cl(x_{1}),\,\dots,\,\cl(x_{s})\,;\,\sigma_{0},\,\sigma_{1},\,\dots,\,\sigma_{s});
\end{equation}
here, for each $1 \le p < q \le s$ such that $\cl(x_{p})= \cdots = \cl(x_{q})$, 
we drop $\cl(x_{p}),\,\dots,\,\cl(x_{q-1})$ and $\sigma_{p},\,\dots,\,\sigma_{q-1}$. 
We set $\cl(\bzero):=\bzero$ by convention. 
We know from \cite[Sect.~6.2]{NS16} that $\cl(\pi) \in \QLS(\lambda)$ 
for all $\pi \in \SLS(\lambda)$. Also, we see from the definitions that
%
%
\begin{equation} \label{eq:cl}
\begin{cases}
\wt(\cl(\pi)) = \cl(\wt(\pi)), \\
\cl(e_{i}\pi) = e_{i} \cl(\pi), \ \cl(f_{i}\pi) = f_{i} \cl(\pi), \\
\ve_{i}(\cl(\pi))=\ve_{i}(\pi), \ \vp_{i}(\cl(\pi))=\vp_{i}(\pi)
\end{cases}
\end{equation}
for all $\pi \in \SLS(\lambda)$ and $i \in I$. 
We know the following lemma from \cite[Lemma~6.2.3]{NS16}; 
recall that $\SLS_{0}(\lambda)$ denotes the connected component of $\SLS(\lambda)$ 
containing $\pi_{\lambda}=(e\,;\,0,1)$. 
%
%
\begin{lem} \label{lem:deg}
For each $\eta \in \QLS(\lambda)$, there exists a unique 
$\pi_{\eta} \in \SLS_{0}(\lambda)$ such that 
$\cl(\pi_{\eta})=\eta$ and $\kappa(\pi_{\eta}) = \kappa(\eta) \in \WJu$. 
\end{lem}

We define the (tail) degree function $\deg_{\lambda}:
\QLS(\lambda) \rightarrow \BZ_{\le 0}$ as follows. 
Let $\eta \in \QLS(\lambda)$, and 
take $\pi_{\eta} \in \SLS_{0}(\lambda)$ as in Lemma~\ref{lem:deg};
we see from the argument in \cite[Sect.~6.2]{NS16} that 
$\wt(\pi_{\eta}) = \lambda - \gamma + k\delta$ 
for some $\gamma \in Q^{+}$ and $k \in \BZ_{\le 0}$. 
Then we set $\deg_{\lambda}(\eta):=k$. 

Now, for $\eta = (v_{1},\,\dots,\,v_{s} \,;\, 
\sigma_{0},\,\sigma_{1},\,\dots,\,\sigma_{s}) \in \QLS(\lambda)$, 
we set 
%
%
\begin{equation} \label{eq:bxi}
\begin{cases}
\bxi{\eta}:=(\xi_1,\,\dots,\,\xi_{s-1},\,\xi_{s}), \quad \text{where} \\[2mm]
\xi_{s}:=0, \quad 
\xi_{u}:=\xi_{u+1} + \wt^{\J}(v_{u+1} \Rightarrow v_{u})
\quad \text{for $1 \le u \le s-1$}; 
\end{cases}
\end{equation}
for the definition of $\wt^{\J}(w \Rightarrow v)$, see Section~\ref{subsec:QBG}. 
%
%
\begin{prop} \label{prop:deg}
Keep the notation and setting above. It holds that 
%
%
\begin{equation} \label{eq:pieta}
\pi_{\eta} = 
(v_{1}\PJ(t_{\xi_1}),\,\dots,\,v_{s-1}\PJ(t_{\xi_{s-1}}),\,v_{s} \,;\, 
\sigma_{0},\,\sigma_{1},\,\dots,\,\sigma_{s}), 
\end{equation}
%
%
\begin{equation} \label{eq:deg0}
\deg_{\lambda}(\eta) = - \sum_{u=1}^{s-1} 
\sigma_{u} \pair{\lambda}{\wt^{\J}(v_{u+1} \Rightarrow v_{u})}.
\end{equation}
\end{prop}

\begin{proof}
Equality \eqref{eq:pieta} can be shown in exactly the same way as 
\cite[Theorem~4.1.1]{LNSSS15}. 
Equality \eqref{eq:deg0} follows from 
\cite[Theorem~4.1.1]{LNSSS15}; or, we can show this equality by 
direct computation, using with \eqref{eq:pieta} and \eqref{eq:wt}. 
\end{proof}

Also, for an arbitrary $w \in W$ and $\eta = 
(v_{1},\,\dots,\,v_{s} \,;\, \sigma_{0},\,\sigma_{1},\,\dots,\,\sigma_{s}) \in \QLS(\lambda)$, 
we define the degree of $\eta$ at $w\lambda$ 
(see \cite[Sect.~3.2]{NNS1} and \cite[Sect.~2.3]{NNS2}) by 
%
%
\begin{equation} \label{eq:degw}
\deg_{w\lambda}(\eta):=
- \sum_{u=1}^{s} 
\sigma_{u} \pair{\lambda}{\wt^{\J}(v_{u+1} \Rightarrow v_{u})}, \quad 
\text{with $v_{s+1}:=\mcr{w}$}.
\end{equation}
We set
\begin{equation}
\gch_{w\lambda} \QLS(\lambda) := 
 \sum_{\eta \in \QLS(\lambda)} q^{\deg_{w\lambda}(\eta)}e^{\wt(\eta)}. 
\end{equation}

\begin{rem} \label{rem:degw}
Let $\eta \in \QLS(\lambda)$, and $w \in W$. 
We see by \eqref{eq:Txi1a} that 
$\wt (\pi_{\eta} \cdot T_{\wt^{\J}(\mcr{w} \Rightarrow \kappa(\eta))}) = 
\wt (\eta) + \deg_{w\lambda}(\eta)\delta$. 
\end{rem}

Let $\lambda_{1},\,\dots,\,\lambda_{n} \in P^{+}$, 
and set $\lambda:=\lambda_{1}+\cdots+\lambda_{n} \in P^{+}$. 
Then the following diagram is commutative: 
%
%
\begin{equation} \label{eq:CD}
\begin{split}
\xymatrix{
\SLS_{0}(\lambda) 
 \ar[rr]^-{\Xi_{\lambda_{1},\dots,\lambda_{n}}}
 \ar[d]_{\cl}
 & & (\SLS(\lambda_{1}) \otimes \cdots \otimes \SLS(\lambda_{n}))_{0} 
 \ar[d]^{\cl \otimes \cdots \otimes \cl} \\
\QLS(\lambda) 
 \ar[rr]^-{\Theta_{\lambda_{1},\dots,\lambda_{n}}}
 & & \QLS(\lambda_{1}) \otimes \cdots \otimes \QLS(\lambda_{n}),}
\end{split}
\end{equation}
where $\Xi_{\lambda_{1},\dots,\lambda_{n}}$ and 
$\Theta_{\lambda_{1},\dots,\lambda_{n}}$ are the isomorphisms of crystals 
in Theorems~\ref{thm:SMT} and \ref{thm:NS05}\,(2), respectively. 
Indeed, observe that both of the maps 
$(\cl \otimes \cdots \otimes \cl) \circ 
\Xi_{\lambda_{1},\dots,\lambda_{n}}$ and 
$\Theta_{\lambda_{1},\dots,\lambda_{n}} \circ \cl$ send 
$\pi_{\lambda}$ 
to $\eta_{\lambda_{1}} \otimes \cdots \otimes \eta_{\lambda_{n}}$, 
and that these two maps commute with the root operators 
(see \eqref{eq:cl}). Because $\SLS_{0}(\lambda) \subset \SLS(\lambda)$ 
is connected and contains $\pi_{\lambda}$, 
we conclude that the diagram above is commutative. 
%
%
\section{Main result and its proof.}
\label{sec:main}
%
%
\subsection{Main result.}
\label{subsec:main}
In this subsection, we fix $\lambda,\mu \in P^{+}$, 
and take $\J_{\lambda}$, $\J_{\mu}$, $\J_{\lambda+\mu}$ as in \eqref{eq:J}; 
note that $\J_{\lambda+\mu} = \J_{\lambda} \cap \J_{\mu}$. 
Recall from \eqref{eq:CD} that the following diagram commutes: 
%
%
\begin{equation} \label{eq:CD2}
\begin{split}
\xymatrix{
\SLS_{0}(\lambda+\mu) 
 \ar[rr]^-{\Xi_{\lambda\mu}}
 \ar[d]_{\cl}
 & & (\SLS(\lambda) \otimes \SLS(\mu))_{0} 
 \ar[d]^{\cl \otimes \cl} \\
\QLS(\lambda+\mu) 
 \ar[rr]^-{\Theta_{\lambda\mu}}
 & & \QLS(\lambda) \otimes \QLS(\mu).}
\end{split}
\end{equation}
For $\eta=(v_1,\,\dots,\,v_{s};
\sigma_{0},\,\sigma_1,\,\dots,\,\sigma_s) \in \QLS(\mu)$ and $w \in W$, 
we define
%
%
\begin{equation} \label{eq:ti1}
\begin{cases}
\tiv{\eta}{w}:=
(\ti{v}_1,\,\dots,\,\ti{v}_{s},\,\ti{v}_{s+1}), \quad \text{where} \\[2mm]
\ti{v}_{s+1}:=w, \quad \ti{v}_{u}:=\tbmin{v_{u}}{\J_{\mu}}{\ti{v}_{u+1}} 
\quad \text{for $1 \le u \le s$}; 
\end{cases}
\end{equation}
note that $\ti{v}_{u} \in v_{u}\WS{\mu}$ for $1 \le u \le s$. 
In this notation, we set $\io{\eta}{w}:=\ti{v}_{1}$, and call it 
the initial direction of $\eta$ with respect to $w$. Also, we define
%
%
\begin{equation} \label{eq:ti2}
\begin{cases}
\tixi{\eta}{w}:=
(\ti{\xi}_1,\,\dots,\,\ti{\xi}_{s}), \quad \text{where} \\[2mm]
\ti{\xi}_{s}:=\wt(\ti{v}_{s+1} \Rightarrow \ti{v}_{s}), \quad
\ti{\xi}_{u}:=\ti{\xi}_{u+1} + \wt(\ti{v}_{u+1} \Rightarrow \ti{v}_{u})
\quad \text{for $1 \le u \le s-1$}.
\end{cases}
\end{equation}
In this notation, we set $\ze{\eta}{w}:=\ti{\xi}_{1}$. 

%
\begin{rem} \label{rem:equiv}
Keep the notation and setting above. 
We see from Lemma~\ref{lem:wtS} that 
\begin{equation*}
\begin{cases}
\wt(\ti{v}_{u+1} \Rightarrow \ti{v}_{u}) \equiv 
 \wt^{\J_{\mu}}(v_{u+1} \Rightarrow v_{u}) \mod \QSv{\mu} \quad 
 \text{for every $1 \le u \le s-1$}, \\
\wt(\ti{v}_{s+1} \Rightarrow \ti{v}_{s}) \equiv 
\wt^{\J_{\mu}}(\mcr{w}^{\J_{\mu}} \Rightarrow v_{s}) = 
\wt^{\J_{\mu}}(\mcr{w}^{\J_{\mu}} \Rightarrow \kappa(\eta)) \mod \QSv{\mu}.
\end{cases}
\end{equation*}
If $\bxi{\eta}=(\xi_{1},\,\dots,\,\xi_{s-1},\,\xi_{s})$ (see \eqref{eq:bxi}), 
then $\ti{\xi}_{u} \equiv 
\xi_{u} + \wt^{\J_{\mu}}(\mcr{w}^{\J_{\mu}} \Rightarrow \kappa(\eta))$ 
mod $\QSv{\mu}$ for all $1 \le u \le s$. 
\end{rem}

The following theorem is the main result of this paper. 
%
%
\begin{thm} \label{thm:main}
Keep the notation and setting above. 
Let $\eta \in \QLS(\lambda+\mu)$, and write 
$\Theta_{\lambda\mu}(\eta) \in \QLS(\lambda) \otimes \QLS(\mu)$ as
$\Theta_{\lambda\mu}(\eta) = \eta_{1} \otimes \eta_{2}$ 
with $\eta_{1} \in \QLS(\lambda)$ and 
$\eta_{2} \in \QLS(\mu)$. Let $w \in W$. 
Then the following equality holds{\rm:}
%
%
\begin{equation} \label{eq:main}
\Xi_{\lambda\mu}(\pi_{\eta} \cdot T_{\wt(w \Rightarrow \kappa(\eta))}) = 
\pi_{\eta_{1}} \cdot T_{\wt(\io{\eta_2}{w} \Rightarrow \kappa(\eta_1))+ \ze{\eta_2}{w}} \otimes 
\pi_{\eta_{2}} \cdot T_{\wt(w \Rightarrow \kappa(\eta_2))}. 
\end{equation}
\end{thm}

We will give a proof of Theorem~\ref{thm:main} 
in Section~\ref{subsec:prf-main1}. 

As an application of Theorem~\ref{thm:main}, 
using Remarks~\ref{rem:degw} and \ref{rem:equiv}, 
we obtain the following.  
%
%
\begin{cor} \label{cor:main}
Keep the notation and setting in Theorem~\ref{thm:main}. 
It holds that
\begin{equation}
\gch_{w(\lambda+\mu)} \QLS(\lambda+\mu) = 
\sum_{\eta \in \QLS(\mu)} e^{\wt(\eta)} 
 q^{\deg_{w\mu}(\eta)-\pair{\lambda}{\ze{\eta}{w}} }
 \gch_{\io{\eta}{w}\lambda} \QLS(\lambda). 
\end{equation}
\end{cor}

%
\subsection{Some technical lemmas concerning the quantum Bruhat graph.}
\label{subsec:lem1}
%
%
\begin{lem}[{see \cite[Lemma~7.7]{LNSSS1}}] \label{lem:dia}
Let $w,v \in W$, and $i \in I_{\af}$. 
\begin{enu}
\item If $w^{-1}\ti{\alpha}_{i} \in \Delta^{+}$ and 
$v^{-1}\ti{\alpha}_{i} \in \Delta^{-}$, then 
\begin{equation*}
\begin{cases}
\ell(w \Rightarrow v) = \ell(\ti{s}_{i}w \Rightarrow v) + 1 = \ell(w \Rightarrow \ti{s}_{i}v) + 1, \\
\wt(w \Rightarrow v) 
 = \wt(\ti{s}_{i}w \Rightarrow v) + \delta_{i0}w^{-1}\ti{\alpha}_{0}^{\vee} 
 = \wt(w \Rightarrow \ti{s}_{i}v) - \delta_{i0}v^{-1}\ti{\alpha}_{0}^{\vee}.
\end{cases}
\end{equation*}

\item If $w^{-1}\ti{\alpha}_{i},\,v^{-1}\ti{\alpha}_{i} \in \Delta^{+}$, 
or if $w^{-1}\ti{\alpha}_{i},\,v^{-1}\ti{\alpha}_{i} \in \Delta^{-}$, then 
\begin{equation*}
\begin{cases}
\ell(w \Rightarrow v) = \ell(s_{i}w \Rightarrow s_{i}v), \\
\wt(w \Rightarrow v) 
 = \wt(\ti{s}_{i}w \Rightarrow \ti{s}_{i}v) + \delta_{i0}w^{-1}\ti{\alpha}_{0}^{\vee} 
   - \delta_{i0}v^{-1}\ti{\alpha}_{0}^{\vee}.
\end{cases}
\end{equation*}
\end{enu}
\end{lem}
%
%
\begin{lem}[{see \cite[Propositions~5.10 and 5.11]{LNSSS1}}] \label{lem:edge}
Let $w \in W$, and $i \in I_{\af}$. 
If $w^{-1}\ti{\alpha}_{i} \in \Delta^{+}$, then 
$w \edge{w^{-1}\ti{\alpha}_{i}} \ti{s}_{i}w$ is a directed edge of $\QB${\rm;}
this edge is a quantum edge if and only if $i = 0$. 
\end{lem}

%
\begin{lem} \label{lem:tb}
Let $\J$ be a subset of $I$.
Let $w \in W$ and $i \in I_{\af}$ 
be such that $w^{-1}\ti{\alpha}_{i} \in \Delta^{+}$. 
Let $v \in W$, and set $\ti{v}:=\tbmin{v}{\J}{w}$.
\begin{enu}
\item 
If $v^{-1}\ti{\alpha}_{i} \in \Delta^{+} \setminus \DeJ^{+}$, then 
$\tbmin{\ti{s}_{i}v}{\J}{\ti{s}_{i}w} = \ti{s}_{i}\ti{v}$. 

\item 
If $v^{-1}\ti{\alpha}_{i} \in \Delta^{-} \setminus \DeJ^{-}$, 
then $\tbmin{v}{\J}{\ti{s}_{i}w} = \ti{v}$. 

\item 
If $v^{-1}\ti{\alpha}_{i} \in \DeJ$, then 
$\ti{v}^{-1} \ti{\alpha}_{i} \in \Delta^{+}$. 
Moreover, if we set 
$\ti{v}':=\tbmin{v}{\J}{\ti{s}_{i}w}$, then
\begin{equation} \label{eq:tb0}
\ti{v}'=
 \begin{cases}
 \ti{v} & \text{\rm if $(\ti{v}')^{-1} \ti{\alpha}_{i} \in \Delta^{+}$}, \\[1.5mm]
 \ti{s}_{i}\ti{v} & \text{\rm if $(\ti{v}')^{-1} \ti{\alpha}_{i} \in \Delta^{-}$}. 
 \end{cases}
\end{equation}
\end{enu}
\end{lem}
\begin{proof}
(1) First we show that 
$u^{-1}\ti{\alpha}_{i} \in \Delta^{-}$ for all $u \in \ti{s}_{i}v\WJ$. 
Let us write $u \in \ti{s}_{i}v\WJ$ as 
$u = \ti{s}_{i}vz$ for some $z \in \WJ$. 
We see by the assumption of part (1) that 
$u^{-1}\ti{\alpha}_{i} = - z^{-1}v^{-1}\ti{\alpha}_{i} \subset z^{-1}(\Delta^{-} \setminus \DeJ^{-})$. 
Since  $\Delta^{-} \setminus \DeJ^{-}$ is stable under the action of $\WJ$, 
we deduce that 
$u^{-1}\ti{\alpha}_{i} \in \Delta^{-} \setminus \DeJ^{-} \subset \Delta^{-}$, as desired. 

We set $\ti{v}':=\tbmin{\ti{s}_{i}v}{\J}{\ti{s}_{i}w}$. 
Since $\ti{s}_{i} \ti{v} \in \ti{s}_{i}v\WJ$, 
we have $\ti{v}' \tb{\ti{s}_{i}w} \ti{s}_{i} \ti{v}$, and hence 
\begin{equation} \label{eq:tb1a}
\ell(\ti{s}_{i}w \Rightarrow \ti{s}_{i} \ti{v}) = 
\ell(\ti{s}_{i}w \Rightarrow \ti{v}') + \ell(\ti{v}' \Rightarrow \ti{s}_{i} \ti{v})
\end{equation}
by the definition of $\tb{\ti{s}_{i}w}$. 
Since $(\ti{s}_{i} \ti{v})^{-1}\ti{\alpha}_{i} \in \Delta^{-}$, as seen above, 
and $(\ti{s}_{i}w)^{-1}\ti{\alpha}_{i} \in \Delta^{-}$ by the assumption, 
we deduce from Lemma~\ref{lem:dia}\,(2) that 
$\ell(\ti{s}_{i}w \Rightarrow \ti{s}_{i} \ti{v}) = 
 \ell(w \Rightarrow \ti{v})$.
Similarly, since $(\ti{s}_{i}w)^{-1}\ti{\alpha}_{i},\,
(\ti{v}')^{-1}\ti{\alpha}_{i},\,
(\ti{s}_{i} \ti{v})^{-1}\ti{\alpha}_{i} \in \Delta^{-}$, 
we deduce that
$\ell(\ti{s}_{i}w \Rightarrow \ti{v}') = 
 \ell(w \Rightarrow \ti{s}_{i}\ti{v}')$ and 
$\ell(\ti{v}' \Rightarrow \ti{s}_{i} \ti{v}) = 
 \ell(\ti{s}_{i}\ti{v}' \Rightarrow \ti{v})$ 
from Lemma~\ref{lem:dia}\,(2). 
Substituting these equalities into \eqref{eq:tb1a}, 
we obtain 
$\ell(w \Rightarrow \ti{v}) = 
\ell(w \Rightarrow \ti{s}_{i}\ti{v}') + 
\ell(\ti{s}_{i}\ti{v}' \Rightarrow \ti{v})$, 
which implies that $\ti{s}_{i}\ti{v}' \tb{w} \ti{v}$. 
Since $\ti{v}=\tbmin{v}{\J}{w}$ and $\ti{s}_{i}\ti{v}' \in v\WJ$, 
it follows that $\ti{s}_{i}\ti{v}' = \ti{v}$, and hence 
$\ti{v}'=\ti{s}_{i}\ti{v}$. 

(2) We set $\ti{v}':=\tbmin{v}{\J}{\ti{s}_{i}w}$. 
Since $\ti{v} \in v\WJ$, we have $\ti{v}' \tb{\ti{s}_{i}w} \ti{v}$, and hence 
\begin{equation} \label{eq:tb2a}
\ell(\ti{s}_{i}w \Rightarrow \ti{v}) = 
\ell(\ti{s}_{i}w \Rightarrow \ti{v}') + \ell(\ti{v}' \Rightarrow \ti{v})
\end{equation}
by the definition of $\tb{\ti{s}_{i}w}$. 
Note that $w^{-1}\ti{\alpha}_{i} \in \Delta^{+}$ by the assumption. 
If we write $\ti{v} \in v\WJ$ as $\ti{v}=vz$ for some $z \in \WJ$, 
then we see by the assumption of part (2) 
that $\ti{v}^{-1}\ti{\alpha}_{i} = z^{-1}v^{-1}\ti{\alpha}_{i} 
\in z^{-1}(\Delta^{-} \setminus \DeJ^{-}) \subset 
\Delta^{-} \setminus \DeJ^{-} \subset \Delta^{-}$. 
Similarly, we see that $(\ti{v}')^{-1}\ti{\alpha}_{i} \in \Delta^{-}$ 
since $\ti{v}' \in v\WJ$.
Therefore, we see from Lemma~\ref{lem:dia}\,(1) that 
$\ell(w \Rightarrow \ti{v}) = \ell(\ti{s}_{i}w \Rightarrow \ti{v}) + 1$ and 
$\ell(w \Rightarrow \ti{v}') = \ell(\ti{s}_{i}w \Rightarrow \ti{v}') +1$. 
Substituting these equalities into \eqref{eq:tb2a}, we obtain 
$\ell(w \Rightarrow \ti{v}) = 
\ell(w \Rightarrow \ti{v}') + 
\ell(\ti{v}' \Rightarrow \ti{v})$, 
which implies that $\ti{v}' \tb{w} \ti{v}$. 
Since $\ti{v}=\tbmin{v}{\J}{w}$ and $\ti{v}' \in v\WJ$, 
we get $\ti{v}' = \ti{v}$. 

(3) Note that $\ti{s}_{i}v\WJ = v\WJ$ since $v^{-1}\ti{\alpha}_{i} \in \DeJ$. 
Suppose, for a contradiction, that 
$\ti{v}^{-1}\ti{\alpha}_{i} \in \Delta^{-}$. 
Since $\ti{s}_{i}\ti{v} \in \ti{s}_{i}v\WJ = v\WJ$ and 
$\ti{v}=\tbmin{v}{\J}{w}$, 
we have $\ti{v} \tb{w} \ti{s}_{i}\ti{v}$, and hence 
\begin{equation} \label{eq:tb3a}
\ell(w \Rightarrow \ti{s}_{i}\ti{v}) = 
\ell(w \Rightarrow \ti{v}) + \ell(\ti{v} \Rightarrow \ti{s}_{i}\ti{v}) \ge 
\ell(w \Rightarrow \ti{v}). 
\end{equation}
Also, since $w^{-1}\ti{\alpha}_{i} \in \Delta^{+}$ and 
$\ti{v}^{-1}\ti{\alpha}_{i} \in \Delta^{-}$ by our assumptions, 
we see from Lemma~\ref{lem:dia}\,(1) that 
$\ell(w \Rightarrow \ti{s}_{i}\ti{v}) = \ell(w \Rightarrow \ti{v})-1 < 
\ell(w \Rightarrow \ti{v})$, which contradicts \eqref{eq:tb3a}. 
Thus we obtain $\ti{v}^{-1}\ti{\alpha}_{i} \in \Delta^{+}$. 
%
\paragraph{Case 1.}
%
Assume that $(\ti{v}')^{-1}\ti{\alpha}_{i} \in \Delta^{+}$. 
Since $\ti{v}' \in v\WJ$, we have $\ti{v} \tb{w} \ti{v}'$, and hence 
\begin{equation} \label{eq:tb3b}
\ell(w \Rightarrow \ti{v}') = 
\ell(w \Rightarrow \ti{v}) + \ell(\ti{v} \Rightarrow \ti{v}')
\end{equation}
by the definition of $\tb{w}$. 
Recall that $w^{-1}\ti{\alpha}_{i} \in \Delta^{+}$ by the assumption. 
Also, recall that $\ti{v}^{-1}\ti{\alpha}_{i} \in \Delta^{+}$ as seen above, 
and that $(\ti{v}')^{-1}\ti{\alpha}_{i} \in \Delta^{+}$ by the assumption of Case 1. 
Therefore, by Lemma~\ref{lem:dia}\,(2), we deduce that 
$\ell(\ti{s}_{i}w \Rightarrow \ti{s}_{i}\ti{v}') = \ell(w \Rightarrow \ti{v}')$, 
$\ell(\ti{s}_{i}w \Rightarrow \ti{s}_{i}\ti{v}) = \ell(w \Rightarrow \ti{v})$, and 
$\ell(\ti{s}_{i}\ti{v} \Rightarrow \ti{s}_{i}\ti{v}') = \ell(\ti{v} \Rightarrow \ti{v}')$. 
Substituting these equalities into \eqref{eq:tb3b}, we obtain
\begin{equation} \label{eq:tb3c}
\ell(\ti{s}_{i}w \Rightarrow \ti{s}_{i}\ti{v}') = 
\ell(\ti{s}_{i}w \Rightarrow \ti{s}_{i}\ti{v}) + 
\ell(\ti{s}_{i}\ti{v} \Rightarrow \ti{s}_{i}\ti{v}').
\end{equation}
Since $\ti{s}_{i}\ti{v} \in \ti{s}_{i}v\WJ=v\WJ$, 
and since $\ti{v}'=\tbmin{v}{\J}{\ti{s}_{i}w}$, 
we have $\ti{v}' \tb{\ti{s}_{i}w} \ti{s}_{i}\ti{v}$, and hence 
$\ell(\ti{s}_{i}w \Rightarrow \ti{s}_{i}\ti{v}) = 
\ell(\ti{s}_{i}w \Rightarrow \ti{v}') + \ell(\ti{v}' \Rightarrow \ti{s}_{i}\ti{v})$. 
Substituting this equality into \eqref{eq:tb3c}, we obtain
\begin{equation} \label{eq:tb3d}
\ell(\ti{s}_{i}w \Rightarrow \ti{s}_{i}\ti{v}') = 
\ell(\ti{s}_{i}w \Rightarrow \ti{v}') + \ell(\ti{v}' \Rightarrow \ti{s}_{i}\ti{v}) + 
\ell(\ti{s}_{i}\ti{v} \Rightarrow \ti{s}_{i}\ti{v}').
\end{equation}
Since $\ti{s}_{i}\ti{v}' \in \ti{s}_{i}v\WJ=v\WJ$, we have 
$\ti{v}' \tb{\ti{s}_{i}w} \ti{s}_{i}\ti{v}'$, 
and hence 
$\ell(\ti{s}_{i}w \Rightarrow \ti{s}_{i}\ti{v}') = 
 \ell(\ti{s}_{i}w \Rightarrow \ti{v}') +
 \ell(\ti{v}' \Rightarrow \ti{s}_{i}\ti{v}')$. 
Combining this equality and \eqref{eq:tb3d}, 
we obtain 
$\ell(\ti{v}' \Rightarrow \ti{s}_{i}\ti{v}) + 
 \ell(\ti{s}_{i}\ti{v} \Rightarrow \ti{s}_{i}\ti{v}') = 
 \ell(\ti{v}' \Rightarrow \ti{s}_{i}\ti{v}')$. 
Since $(\ti{v}')^{-1}\ti{\alpha}_{i} \in \Delta^{+}$ by the assumption of Case 1,  
it follows from Lemma~\ref{lem:edge} that 
$\ti{v}' \edge{(\ti{v}')^{-1}\ti{\alpha}_{i}} \ti{s}_{i}\ti{v}'$ 
is a directed edge of $\QB$, 
which implies that $\ell(\ti{v}' \Rightarrow \ti{s}_{i}\ti{v}') = 1$. 
Hence either $\ell(\ti{v}' \Rightarrow \ti{s}_{i}\ti{v})=0$ or 
$\ell(\ti{s}_{i}\ti{v} \Rightarrow \ti{s}_{i}\ti{v}')=0$. 
If $\ell(\ti{v}' \Rightarrow \ti{s}_{i}\ti{v}) = 0$, then 
$\ti{v}' = \ti{s}_{i}\ti{v}$. However, this contradicts our assumption that
$(\ti{v}')^{-1}\ti{\alpha}_{i} \in \Delta^{+}$ since 
$(\ti{s}_{i}\ti{v})^{-1}\ti{\alpha}_{i} = -\ti{v}^{-1}\ti{\alpha}_{i} \in \Delta^{-}$. 
Therefore, we obtain $\ell(\ti{s}_{i}\ti{v} \Rightarrow \ti{s}_{i}\ti{v}')=0$, 
from which we conclude that $\ti{s}_{i}\ti{v} = \ti{s}_{i}\ti{v}'$, 
and hence $\ti{v} = \ti{v}'$. 
%
\paragraph{Case 2.}
%
Assume that $(\ti{v}')^{-1}\ti{\alpha}_{i} \in \Delta^{-}$. 
Since $\ti{s}_{i} \ti{v} \in \ti{s}_{i}v\WJ = v\WJ$, 
we have $\ti{v}' \tb{\ti{s}_{i}w} \ti{s}_{i} \ti{v}$, and hence 
\begin{equation} \label{eq:tb3e}
\ell(\ti{s}_{i}w \Rightarrow \ti{s}_{i} \ti{v}) = 
\ell(\ti{s}_{i}w \Rightarrow \ti{v}') + \ell(\ti{v}' \Rightarrow \ti{s}_{i} \ti{v}). 
\end{equation}
Since $\ti{v}^{-1}\ti{\alpha}_{i},\,
w^{-1}\ti{\alpha}_{i} \in \Delta^{+}$, 
we deduce from Lemma~\ref{lem:dia}\,(2) that 
$\ell(\ti{s}_{i}w \Rightarrow \ti{s}_{i} \ti{v}) = 
 \ell(w \Rightarrow \ti{v})$.
Similarly, since $(\ti{s}_{i}w)^{-1}\ti{\alpha}_{i},\,
(\ti{v}')^{-1}\ti{\alpha}_{i},\,
(\ti{s}_{i} \ti{v})^{-1}\ti{\alpha}_{i} \in \Delta^{-}$, 
we deduce that 
$\ell(\ti{s}_{i}w \Rightarrow \ti{v}') = 
 \ell(w \Rightarrow \ti{s}_{i}\ti{v}')$ and 
$\ell(\ti{v}' \Rightarrow \ti{s}_{i} \ti{v}) = 
 \ell(\ti{s}_{i}\ti{v}' \Rightarrow \ti{v})$ 
from Lemma~\ref{lem:dia}\,(2). 
Substituting these equalities into \eqref{eq:tb3e}, 
we obtain 
$\ell(w \Rightarrow \ti{v}) = 
\ell(w \Rightarrow \ti{s}_{i}\ti{v}') + 
\ell(\ti{s}_{i}\ti{v}' \Rightarrow \ti{v})$, 
which implies that $\ti{s}_{i}\ti{v}' \tb{w} \ti{v}$. 
Since $\ti{v}=\tbmin{v}{\J}{w}$ and $\ti{s}_{i}\ti{v}' \in \ti{s}_{i}v\WJ = v\WJ$, 
it follows that $\ti{s}_{i}\ti{v}' = \ti{v}$, and hence 
$\ti{v}'=\ti{s}_{i}\ti{v}$. 

This completes the proof of Lemma~\ref{lem:tb}. 
\end{proof}
%
%
\subsection{Similarity maps for $\SLS(\lambda)$ and $\QLS(\lambda)$.}
\label{subsec:sim}

Let $\lambda \in P^{+}$, and take $\J=\J_{\lambda}$ as in \eqref{eq:J}; 
recall the definition of $\turn{\lambda}$ from Remark~\ref{rem:SLS}. 
Take $N=N_{\lambda} \in \BZ_{\ge 1}$ such that 
%
%
\begin{equation} \label{eq:N}
N\sigma=N_{\lambda}\sigma \in \BZ \quad \text{for all $\sigma \in \turn{\lambda}$}.
\end{equation} 
We define $\Sigma_{N}:\SLS(\lambda) \rightarrow 
\SLS(\lambda)^{\otimes N}$ as follows. 
Let $\pi = (x_{1},\,\dots,\,x_{s};\sigma_{0},\sigma_{1},\dots,\sigma_{s}) 
\in \SLS(\lambda)$; note that $N\sigma_{u} \in \BZ$ for all $0 \le u \le s$ 
(see Remark~\ref{rem:SLS}). We set
\begin{equation}
\Sigma_{N}(\pi) : = 
(\pi_{\lambda}^{x_{1}})^{\otimes N(\sigma_{1}-\sigma_{0})} \otimes 
(\pi_{\lambda}^{x_{2}})^{\otimes N(\sigma_{2}-\sigma_{1})} \otimes \cdots \otimes 
(\pi_{\lambda}^{x_{s}})^{\otimes N(\sigma_{s}-\sigma_{s-1})} \in \SLS(\lambda)^{\otimes N}; 
\end{equation}
recall that $\pi_{\lambda}^{x}=(x;0,1) \in \SLS(\lambda)$ for $x \in (\WJu)_{\af}$. 
We set $\Sigma_N (\bzero) := \bzero$ by convention. 
We know the following proposition from \cite[Proposition~5.24]{INS}. 
%
%
\begin{prop} \label{prop:sim1}
The map $\Sigma_{N}:\SLS(\lambda) \rightarrow 
\SLS(\lambda)^{\otimes N}$ is an injective map such that 
$\Sigma_{N}(\pi_{\lambda})=\pi_{\lambda}^{\otimes N}$, and 
for all $\pi \in \SLS(\lambda)$ and $i \in I_{\af}$, 
\begin{align}
& \wt(\Sigma_N (\pi)) = N \wt(\pi), &
& \ve_i (\Sigma_N (\pi)) = N \ve_i (\pi), &
& \vp_i (\Sigma_N (\pi)) = N \vp_i (\pi), \label{eq:SN1} \\
& \Sigma_N (e_i \pi) = e_i^N \Sigma_N (\pi), &
& \Sigma_N (f_i \pi) = f_i^N \Sigma_N (\pi). \label{eq:SN2}
\end{align}
\end{prop}
%
%
\begin{rem} \label{rem:sim1}
By \eqref{eq:SN2}, we see that 
$\Sigma_{N}(\SLS_{0}(\lambda)) \subset 
(\SLS(\lambda)^{\otimes N})_{0}$. 
\end{rem}

For each $\pi = (x_{1},\,\dots,\,x_{s};\sigma_{0},\sigma_{1},\dots,\sigma_{s}) 
\in \SLS(\lambda)$, we see from the definition of semi-infinite LS paths that 
$\Sigma_{N}'(\pi):=
(x_{1},\,\dots,\,x_{s};\sigma_{0},\sigma_{1},\dots,\sigma_{s})$ is contained in $\SLS(N\lambda)$.
Hence the map $\Sigma_{N}':\SLS(\lambda) \rightarrow \SLS(N\lambda)$, 
$\pi \mapsto \Sigma_{N}'(\pi)$, is an analogue of the ``$N$-multiple'' map 
in \cite[page~504]{Lit95}; it is verified 
in the same way as \cite[Lemma~2.4]{Lit95} that $\Sigma_{N}'$ has 
the same properties as \eqref{eq:SN1} and \eqref{eq:SN2}. 
In particular, we see that $\Sigma_{N}'(\SLS_{0}(\lambda)) \subset 
\SLS_{0}(N\lambda)$. Also, the following diagram is commutative:
%
%
\begin{equation} \label{eq:CDs1}
\begin{split}
\xymatrix{%
\SLS_{0}(\lambda) 
 \ar[r]^{\Sigma'_{N}}
 \ar[d]_{\Sigma_{N}} & 
\SLS_{0}(N\lambda) \ar[ld]^{\Xi_{\lambda}^{(N)}} \\
(\SLS(\lambda)^{\otimes N})_{0},
 & }
\end{split}
\end{equation}
where $\Xi_{\lambda}^{(N)}$ is the isomorphism of crystals 
given by Theorem~\ref{thm:SMT}, applied to 
the decomposition $N\lambda=\lambda+\cdots+\lambda$ ($N$ times). 
Indeed, both of the maps $\Sigma_{N}$ and $\Xi_{\lambda}^{(N)} \circ \Sigma_{N}'$ 
send the element $\pi_{\lambda}$ to 
$\pi_{\lambda}^{\otimes N}$, 
and have property \eqref{eq:SN2}. 
Since $\SLS_{0}(\lambda) \subset \SLS(\lambda)$ 
is connected and contains $\pi_{\lambda}$, we conclude that 
the diagram above is commutative. 

We define $\Sigma_{N}:\QLS(\lambda) \rightarrow 
\QLS(\lambda)^{\otimes N}$ and 
$\Sigma_{N}':\QLS(\lambda) \rightarrow 
\QLS(N\lambda)$ in exactly the same way as above; 
we deduce that these maps $\Sigma_{N}$ and $\Sigma_{N}'$ have 
the same properties as \eqref{eq:SN1} and \eqref{eq:SN2}, and that
the following diagram is commutative:  
%
%
\begin{equation} \label{eq:CDs2}
\begin{split}
\xymatrix{%
\QLS(\lambda) 
 \ar[r]^{\Sigma'_{N}}
 \ar[d]_{\Sigma_{N}} & 
\QLS(N\lambda) \ar[ld]^{\Theta_{\lambda}^{(N)}} \\
\QLS(\lambda)^{\otimes N},
 & }
\end{split}
\end{equation}
where $\Theta_{\lambda}^{(N)}$ is the isomorphism of crystals 
given by Theorem~\ref{thm:NS05}\,(2), applied to 
the decomposition $N\lambda=\lambda+\cdots+\lambda$ ($N$ times). 
Moreover, the same argument as above shows that 
the following diagram is commutative:
%
%
\begin{equation} \label{eq:CD3}
\begin{split}
\xymatrix{%
 \SLS(\lambda) \ar[r]^-{\Sigma_{N}} \ar[d]_-{\cl} & 
 \SLS(\lambda)^{\otimes N} \ar[d]^-{\cl^{\otimes N}} \\
 \QLS(\lambda) \ar[r]^-{\Sigma_{N}} & 
 \QLS(\lambda)^{\otimes N}. 
}
\end{split}
\end{equation}

Now, let $\lambda_{1},\dots,\lambda_{n} \in P^{+}$, and set 
$\lambda := \lambda_{1}+\cdots +\lambda_{n}$. 
Take $N_{\lambda_{k}} \in \BZ_{\ge 1}$, $1 \le k \le n$, and 
$N_{\lambda} \in \BZ_{\ge 1}$ as in \eqref{eq:N}, 
and let $N \in \BZ_{\ge 1}$ be a common multiple of $N_{\lambda}$, 
$N_{\lambda_1},\,\dots,\,N_{\lambda_n}$. 
In exactly the same way as above, we see that 
the following diagram is commutative: 
%
%
\begin{equation} \label{eq:CD4}
\begin{split}
\xymatrix{%
 \SLS_{0}(\lambda) 
  \ar[rr]^-{\Xi_{\lambda_1,\dots,\lambda_{n}}} 
  \ar[d]_-{\Sigma_{N}}
  \ar[ddr]^-{\Sigma_{N}'} & &
 (\SLS(\lambda_{1}) \otimes \cdots \otimes \SLS(\lambda_{n}))_{0} 
 \ar[d]^-{\Sigma_{N} \otimes \cdots \otimes \Sigma_{N}} \\
 (\SLS(\lambda)^{\otimes N})_{0} & & 
 (\SLS(\lambda_{1})^{\otimes N} \otimes \cdots \otimes \SLS(\lambda_{n})^{\otimes N})_{0} \\
 & \SLS_{0}(N\lambda), \ar[ru]_{\Xi_{\lambda_1,\cdots,\lambda_n}^{(N)}} 
 \ar[lu]^-{\Xi_{\lambda}^{(N)}} & 
}
\end{split}
\end{equation}
where the map $\Xi_{\lambda_1,\cdots,\lambda_n}^{(N)}$ is 
the isomorphism of crystals given by Theorem~\ref{thm:SMT}, 
applied to the decomposition
\begin{equation} \label{eq:decN}
N\lambda = 
(\underbrace{\lambda_1+\cdots+\lambda_1}_{\text{$N$ times}}) + \cdots +
(\underbrace{\lambda_n+\cdots+\lambda_n}_{\text{$N$ times}}); 
\end{equation}
note that $\Sigma_{N} \otimes \cdots \otimes \Sigma_{N}$ has 
the same properties as \eqref{eq:SN1} and \eqref{eq:SN2}. 
Similarly, we obtain the following commutative diagram:
%
%
\begin{equation} \label{eq:CD5}
\begin{split}
\xymatrix{%
 \QLS(\lambda) 
  \ar[rr]^-{\Theta_{\lambda_1,\dots,\lambda_{n}}} 
  \ar[d]_-{\Sigma_{N}}
  \ar[ddr]^-{\Sigma_{N}'} & &
 \QLS(\lambda_{1}) \otimes \cdots \otimes \QLS(\lambda_{n}))
 \ar[d]^-{\Sigma_{N} \otimes \cdots \otimes \Sigma_{N}} \\
 \QLS(\lambda)^{\otimes N} & & 
 \QLS(\lambda_{1})^{\otimes N} \otimes \cdots \otimes \QLS(\lambda_{n})^{\otimes N} \\
 & \QLS(N\lambda), \ar[ru]_{\Theta_{\lambda_1,\cdots,\lambda_n}^{(N)}} 
 \ar[lu]^-{\Theta_{\lambda}^{(N)}} & 
}
\end{split}
\end{equation}
where the map $\Theta_{\lambda_1,\cdots,\lambda_n}^{(N)}$ is 
the isomorphism of crystals given by Theorem~\ref{thm:NS05}\,(2), 
applied to the decomposition \eqref{eq:decN}. 
%
%
\subsection{Some lemmas concerning final directions.}
\label{subsec:lem2}
Let $\lambda \in P^{+}$, and take $\J=\J_{\lambda}$ as in \eqref{eq:J}. 
For $\pi \in \SLS(\lambda)$ (resp., $\eta \in \QLS(\lambda)$) and $i \in I_{\af}$, 
we set $e_{i}^{\max}\pi:=e_{i}^{\ve_{i}(\pi)}\pi$ 
(resp., $e_{i}^{\max}\eta:=e_{i}^{\ve_{i}(\eta)}\eta$) and 
$f_{i}^{\max}\pi:=f_{i}^{\vp_{i}(\pi)}\pi$ 
(resp., $f_{i}^{\max}\eta:=f_{i}^{\vp_{i}(\eta)}\eta$). 
The next lemma follows from the definition of the root operator $f_{i}$ 
(see also \cite[Remark~41]{NNS2}). 
%
%
\begin{lem} \label{lem:kap} \mbox{}
\begin{enu}
\item Let $\pi \in \SLS(\lambda)$, and $i \in I_{\af}$. 
If $0 \le m < \vp_{i}(\pi)$, then $\kappa(f_{i}^{m}\pi) = \kappa(\pi)$. 
Also, 
\begin{equation*}
\kappa(f_{i}^{\max}\pi) = 
 \begin{cases}
   s_{i}\kappa(\pi) 
   & \text{\rm if $\pair{\kappa(\pi)\lambda}{\alpha_{i}^{\vee}} > 0$}, \\[1mm]
   \kappa(\pi) 
   & \text{\rm if $\pair{\kappa(\pi)\lambda}{\alpha_{i}^{\vee}} \le 0$}.
 \end{cases}
\end{equation*}

\item Let $\eta \in \QLS(\lambda)$, and $i \in I_{\af}$. 
If $0 \le m < \vp_{i}(\eta)$, then $\kappa(f_{i}^{m}\eta) = \kappa(\eta)$. 
Also, 
\begin{equation*}
\kappa(f_{i}^{\max}\eta) = 
 \begin{cases}
   \mcr{\ti{s}_{i}\kappa(\eta)} (\ne \kappa(\eta))
   & \text{\rm if $\pair{\kappa(\eta)\lambda}{\alpha_{i}^{\vee}} > 0$}, \\[1mm]
   \kappa(\eta) 
   & \text{\rm if $\pair{\kappa(\eta)\lambda}{\alpha_{i}^{\vee}} \le 0$}.
 \end{cases}
\end{equation*}
\end{enu}
\end{lem}

We can show the following lemma 
by Lemma~\ref{lem:kap}, \eqref{eq:cl}, and 
the definition of $\pi_{\eta}$ (cf. \cite[Remark~4.4]{LNSSS1}). 
%
%
\begin{lem} \label{lem:rec}
Let $\eta \in \QLS(\lambda)$, and $i \in I_{\af}$. 
Then
\begin{align*}
& \pi_{f_{i}^{m}\eta}=f_{i}^{m}\pi_{\eta} \quad 
  \text{\rm for all $0 \le m < \vp_{i}(\eta)$}, \\[2mm]
& \pi_{f_{i}^{\max}\eta} = 
\begin{cases}
f_{i}^{\max}\pi_{\eta} & \text{\rm if $i \ne 0$, or if 
$i = 0$ and $\pair{\kappa(\eta)\lambda}{\alpha_{0}^{\vee}} \le 0$}, \\
f_{0}^{\max}\pi_{\eta} \cdot T_{-\kappa(\eta)^{-1}\ti{\alpha}_{0}^{\vee}}
 & \text{\rm if $i=0$ and $\pair{\kappa(\eta)\lambda}{\alpha_{0}^{\vee}} > 0$}. 
\end{cases}
\end{align*}
\end{lem}
%
%
\begin{lem} \label{lem:conn}
Let $w \in W$. For each $\eta \in \QLS(\lambda)$, 
there exist $i_{1},\,i_{2},\,\dots,\,i_{n} \in I_{\af}$ 
such that 
$f_{i_n}^{\max} \cdots f_{i_2}^{\max}f_{i_1}^{\max}\eta = \eta_{\lambda}^{\mcr{\lng}}$, and 
$(\ti{s}_{i_{k-1}} \cdots \ti{s}_{i_2}\ti{s}_{i_1}w)^{-1}\ti{\alpha}_{i_{k}} \in \Delta^{+}$ 
for all $1 \le k \le n$.
\end{lem}

\begin{proof}
It follows from \cite[Lemma~1.4]{AK} that 
there exist $i_{1},i_{2},\dots,i_{a} \in I_{\af}$ such that 
$(\ti{s}_{i_{k-1}} \cdots \ti{s}_{i_2}\ti{s}_{i_1}w)^{-1}\ti{\alpha}_{i_{k}} 
\in \Delta^{+}$ for all $1 \le k \le a$, and 
$\ti{s}_{i_{a}} \cdots \ti{s}_{i_2}\ti{s}_{i_1}w = e$; 
we set $\eta':=
f_{i_a}^{\max} \cdots f_{i_2}^{\max}f_{i_1}^{\max}\eta$. 
Take $i_{a+1},\,i_{a+2},\,\dots,\,i_{b} \in I$ 
in such a way that 
$\lng = s_{i_{b}} \cdots s_{i_{a+2}}s_{i_{a+1}}$
is a reduced expression for $\lng$. Then, for $a+1 \le k \le b$,
\begin{equation*}
(\ti{s}_{i_{k-1}} \cdots \ti{s}_{i_{a+1}}
 \underbrace{\ti{s}_{i_{a}} \cdots \ti{s}_{i_1}w}_{=e})^{-1}\ti{\alpha}_{i_{k}} = 
(\ti{s}_{i_{k-1}} \cdots \ti{s}_{i_{a+1}})^{-1}\ti{\alpha}_{i_{k}} = 
s_{i_{a+1}} \cdots s_{i_{k-1}}\alpha_{i_{k}} \in \Delta^{+}, 
\end{equation*}
and $\ti{s}_{i_{b}} \cdots \ti{s}_{i_{a+1}} 
\ti{s}_{i_{a}} \cdots \ti{s}_{i_1}w = \lng$. 
Here we recall that the crystal $\QLS(\lambda) = \BB(\lambda)_{\cl}$ is regular 
in the sense that for every proper subset $J \subsetneqq I$, 
it is isomorphic, as a crystal for $U_{q}(\Fg_{J})$, to 
the crystal basis of a finite-dimensional $U_{q}(\Fg_{J})$-module, 
where $\Fg_{J}$ is the (finite-dimensional) Levi subalgebra of $\Fg_{\af}$ corresponding to $J$
(see \cite[Proposition~3.1.3]{NS05}). 
Therefore, we deduce from \cite[Corollaire~9.1.4\,(2)]{KasF} that
$\eta'':=
f_{i_b}^{\max} \cdots f_{i_{a+1}}^{\max}
f_{i_{a}}^{\max} \cdots f_{i_1}^{\max}\eta = 
f_{i_b}^{\max} \cdots f_{i_{a+1}}^{\max}\eta'$
satisfies the condition that $f_{i}\eta'' = \bzero$ for all $i \in I$. 

We know from \cite[Sect.~4.5]{LNSSS2} that 
there exists a (unique) involution $\Lus:\QLS(\lambda) \rightarrow \QLS(\lambda)$ 
(called the Lusztig involution) such that 
$\wt (\Lus(\psi)) = \lng \wt (\psi)$ for $\psi \in \QLS(\lambda)$, 
and $\Lus(e_{i}\psi)=f_{i}\Lus (\psi)$, $\Lus(f_{i}\psi)=e_{i}\Lus (\psi)$ 
for $\psi \in \QLS(\lambda)$ and $i \in I_{\af}$; 
by convention, we set $\Lus(\bzero):=\bzero$. 
Observe that $\Lus(\eta_{\lambda}) = \eta_{\lambda}^{\mcr{\lng}}$. 
We see that $e_{i}\Lus(\eta'')=\bzero$ for all $i \in I$. 
It follows from \cite[Proposition~4.3.1]{NS08} that
there exist $i_{b+1},\,i_{b+2},\,\dots,\,i_{n} \in I_{\af}$ such that 
$e_{i_n}^{\max} \cdots e_{i_{b+2}}^{\max}e_{i_{b+1}}^{\max}\Lus(\eta'') = \eta_{\lambda}$, and 
$(\ti{s}_{i_{k-1}} \cdots \ti{s}_{i_{b+2}}\ti{s}_{i_{b+1}})^{-1}\ti{\alpha}_{i_{k}} \in \Delta^{-}$ 
for all $b+1 \le k \le n$. 
Since $\Lus(\eta_{\lambda})=\eta_{\lambda}^{\mcr{\lng}}$ and 
$e_{i}\Lus(\eta'')=\bzero$ for all $i \in I$, 
we have $f_{i_n}^{\max} \cdots f_{i_{b+2}}^{\max}f_{i_{b+1}}^{\max}\eta'' = \eta_{\lambda}^{\mcr{\lng}}$, 
and $(\ti{s}_{i_{k-1}} \cdots \ti{s}_{i_{b+2}}\ti{s}_{i_{b+1}}\lng)^{-1}\ti{\alpha}_{i_{k}} \in \Delta^{+}$ 
for all $b+1 \le k \le n$. Thus, the sequence 
$i_{1},\,\dots,\,i_{a},\,i_{a+1},\,\dots,\,i_{b},\,i_{b+1},\,\dots,\,i_{n} \in I_{\af}$ 
satisfies the condition of the assertion. This proves the lemma. 
\end{proof}
%
%
\begin{lem} \label{lem:kappa}
Let $\lambda_{1},\,\dots,\,\lambda_{n} \in P^{+}$, and set 
$\lambda:=\lambda_{1} + \cdots + \lambda_{n}$. 
Take $\J_{n}:=\J_{\lambda_{n}}$ and $\J=\J_{\lambda}$ as in \eqref{eq:J}{\rm;}
note that $\J \subset \J_{n}$. 
Let $\eta \in \QLS(\lambda)$, and write $\Theta_{\lambda_{1},\dots,\lambda_{n}}(\eta)$ as 
$\Theta_{\lambda_{1},\dots,\lambda_{n}}(\eta) = \eta_{1} \otimes \cdots \otimes \eta_{n}$ 
for some $\eta_{k} \in \QLS(\lambda_{k})$, $1 \le k \le n$. Then, 
$\mcr{\kappa(\eta)}^{\J_{n}}=\kappa(\eta_{n})$. 
\end{lem}

\begin{proof}
We deduce from \cite[Proposition~4.3.1]{NS08} (along with Remark~\ref{rem:LScl}) 
that there exists a monomial 
$X=f_{i_m}f_{i_{m-1}} \cdots f_{i_1}$ in root operators $f_{i}$, $i \in I_{\af}$, 
such that $\eta = X \eta_{\lambda}$. 
We prove the assertion of the lemma 
by induction on $m$. If $m=0$, then the assertion is obvious. 
Assume that $m > 0$, and set $X':=f_{i_{m-1}} \cdots f_{i_1}$ and $i:=i_{m}$; 
we have $X=f_{i}X'$. We set $\eta':=X' \eta_{\lambda}$, and write 
$X' (\eta_{\lambda_1} \otimes \cdots \otimes \eta_{\lambda_{n}}) = 
\eta_{1}' \otimes \cdots \otimes \eta_{n}'$ 
for some $\eta_{k}' \in \QLS(\lambda_{k})$, $1 \le k \le n$. 
Note that $\eta=f_{i}\eta'$ and $\eta_{1} \otimes \cdots \otimes \eta_{n} = 
f_{i}(\eta_{1}' \otimes \cdots \otimes \eta_{n}')$. Hence
\begin{enu}
\item[(a)] $\eta_{n}$ is identical to either $\eta_{n}'$ or $f_{i}\eta_{n}'$ 
by the tensor product rule for crystals; 
\item[(b)] if $\eta_{n} = \eta_{n}'$ (resp., $\eta_{n}=f_{i}\eta_{n}'$), then 
$\kappa(\eta_{n})$ is identical to $\kappa(\eta_{n}')$ (resp., 
either $\kappa(\eta_{n}')$ or $\mcr{\ti{s}_{i}\kappa(\eta_{n}')}^{S_n}$). 
\end{enu}
Also, we claim that
\begin{enu}
\item[(c)] $\kappa(\eta)=
\mcr{\ti{s}_{i}\kappa(\eta')}^{\J} \ne \kappa(\eta')$ if and only if
$\vp_{i}(\eta')=1$ and $\pair{\kappa(\eta')\lambda}{\alpha_{i}^{\vee}} > 0$;

\item[(d)] 
$\kappa(\eta_{n})=
\mcr{\ti{s}_{i}\kappa(\eta_{n}')}^{\J_{n}} \ne \kappa(\eta_{n}')$ if and only if
$\vp_{i}(\eta_{1}' \otimes \cdots \otimes \eta_{n}') = 1$ and 
$\pair{\kappa(\eta_{n}')\lambda_{n}}{\alpha_{i}^{\vee}} > 0$.
\end{enu}
Indeed, (c) is obvious by Lemma~\ref{lem:kap}\,(2). 
Let us show the ``if'' part of (d). 
Since $\pair{\kappa(\eta_{n}')\lambda_{n}}{\alpha_{i}^{\vee}} > 0$, 
we see from the definition of $f_{i}$ that $\vp_{i}(\eta_n') \ge 1$. 
By the tensor product rule for crystals, we have
$\vp_{i}(\eta_n') \le \vp_{i}(\eta_{1}' \otimes \cdots \otimes \eta_{n}') = 1$. 
Hence we obtain $\vp_{i}(\eta_n') = 1$. 
Suppose that $\eta_{n} = \eta_{n}'$ (see (a)). 
Then we have $1 = \vp_{i}(\eta_n') = \vp_{i}(\eta_n) \le 
\vp_{i}(\eta_{1} \otimes \cdots \otimes \eta_{n})$, 
which contradicts that
$\vp_{i}(\eta_{1} \otimes \cdots \otimes \eta_{n}) = 
\vp_{i}(\eta_{1}' \otimes \cdots \otimes \eta_{n}') - 1 = 0$. 
Thus we have $\eta_{n} = f_{i} \eta_{n}'$. 
Therefore, by Lemma~\ref{lem:kap}\,(2), we conclude that 
$\kappa(\eta_{n})=
\mcr{\ti{s}_{i}\kappa(\eta_{n}')}^{\J_{n}} \ne \kappa(\eta_{n}')$. 
Let us show the ``only if'' part of (d). 
If $\kappa(\eta_{n})=
\mcr{\ti{s}_{i}\kappa(\eta_{n}')}^{\J_{n}} \ne \kappa(\eta_{n}')$, 
then we have $\eta_{n}=f_{i}\eta_{n}'$ (see (b)), 
$\vp_{i}(\eta_{n}') = 1$, and 
$\pair{\kappa(\eta_{n}')\lambda_{n}}{\alpha_{i}^{\vee}} > 0$ 
by Lemma~\ref{lem:kap}\,(2). 
Since $f_{i}((\eta_{1}' \otimes \cdots \otimes \eta_{n-1}') \otimes \eta_{n}') = 
(\eta_{1}' \otimes \cdots \otimes \eta_{n-1}') \otimes f_{i}\eta_{n}'$, 
we see by the tensor product rule for crystals that
$\vp_{i}(\eta_{1}' \otimes \cdots \otimes \eta_{n-1}') 
\le \ve_{i}(\eta_{n}) = \vp_{i}(\eta_{n}') - 
\pair{\wt \eta_{n}'}{\alpha_{i}^{\vee}}$, 
and hence $\vp_{i}(\eta_{1}' \otimes \cdots \otimes \eta_{n-1}' \otimes \eta_{n}') = 
\max \bigl\{ \vp_{i}(\eta_{n}'),\,
\vp_{i}(\eta_{1}' \otimes \cdots \otimes \eta_{n-1}') + 
\pair{\wt \eta_{n}'}{\alpha_{i}^{\vee}} \bigr\} 
= \vp_{i}(\eta_{n}') =1$. This shows (d). 
Finally, by the induction hypothesis, we have 
%
%
\begin{equation} \label{eq:kappa1}
\mcr{\kappa(\eta')}^{\J_n}=\kappa(\eta_{n}').
\end{equation}
%
\paragraph{Case 1.}
%
Assume that $\kappa(\eta)=\kappa(\eta')$; 
it suffices to show that $\kappa(\eta_{n}') = \kappa(\eta_{n})$. 
We see by (c) that 
$\pair{\kappa(\eta')\lambda}{\alpha_{i}^{\vee}} \le 0$ or 
$\vp_{i}(\eta') \ge 2$. If $\pair{\kappa(\eta')\lambda}{\alpha_{i}^{\vee}} \le 0$, 
then it follows from \eqref{eq:kappa1} that 
$\pair{\kappa(\eta_{n}')\lambda_{n}}{\alpha_{i}^{\vee}} \le 0$, which implies that 
$\kappa(\eta_{n}) = \kappa(\eta_{n}')$ by (d). 
If $\vp_{i}(\eta') \ge 2$, then we have 
$\vp_{i}(\eta_{1}' \otimes \cdots \otimes \eta_{n}') \ge 2$, 
and hence $\kappa(\eta_{n}') = \kappa(\eta_{n})$ by (d). 

\paragraph{Case 2.}
%
Assume that $\kappa(\eta)=\mcr{\ti{s}_{i}\kappa(\eta')}^{\J} \ne \kappa(\eta')$; 
we see that 
\begin{equation} \label{eq:kc2}
\mcr{\kappa(\eta)}^{\J_n} = 
\mcr{\mcr{\ti{s}_{i}\kappa(\eta')}^{\J}}^{\J_n} = 
\mcr{\ti{s}_{i}\kappa(\eta')}^{\J_n} = 
\mcr{ \ti{s}_{i}\mcr{\kappa(\eta')}^{\J_n} }^{\J_n} 
\stackrel{\eqref{eq:kappa1}}{=} \mcr{ \ti{s}_{i}\kappa(\eta_{n}') }^{\J_n}.
\end{equation}
Also, by (c), we have $\vp_{i}(\eta')=1$ and 
$\pair{\kappa(\eta')\lambda}{\alpha_{i}^{\vee}} > 0$. 
Hence, $\vp_{i}(\eta_{1}' \otimes \cdots \otimes \eta_{n}')=\vp_{i}(\eta')=1$. 
We see that 
$\pair{\kappa(\eta_{n}')\lambda_{n}}{\alpha_{i}^{\vee}} 
\stackrel{\eqref{eq:kappa1}}{=}
\pair{\kappa(\eta')\lambda_{n}}{\alpha_{i}^{\vee}} \ge 0$ 
since $\pair{\kappa(\eta')\lambda}{\alpha_{i}^{\vee}} > 0$ implies that 
$\kappa(\eta')^{-1}\ti{\alpha}_{i} \in \Delta^{+}$. 
If $\pair{\kappa(\eta_{n}')\lambda_n}{\alpha_{i}^{\vee}} > 0$, then 
we have $\kappa(\eta_{n})=\mcr{\ti{s}_{i}\kappa(\eta_{n}')}^{\J_n}$ 
by (d), and hence $\mcr{\kappa(\eta)}^{\J_n} = 
\kappa(\eta_{n})$ by \eqref{eq:kc2}. 
If $\pair{\kappa(\eta_{n}')\lambda_n}{\alpha_{i}^{\vee}} = 0$, then 
we have $\mcr{\ti{s}_{i}\kappa(\eta_{n}')}^{\J_n} = \kappa(\eta_{n}')$, 
and also $\kappa(\eta_{n}) = \kappa(\eta_{n}')$ by (d). 
Therefore, by \eqref{eq:kc2}, 
$\mcr{\kappa(\eta)}^{\J_n} = \kappa(\eta_{n}') = \kappa(\eta_{n})$. 
This proves the lemma.
\end{proof}
%
%
\subsection{Proof of Theorem~\ref{thm:main}.}
\label{subsec:prf-main1}
%
%
\begin{prop} \label{prop:s}
Let $\lambda,\,\mu \in P^{+}$, and take $\J_{\lambda}$, $\J_{\mu}$, 
$\J_{\lambda+\mu} \subset I$ as in \eqref{eq:J}{\rm;} note that $\J_{\lambda+\mu} = 
\J_{\lambda} \cap \J_{\mu}$. 
Let $\eta_{1} \in \QLS(\lambda)$, and let $\eta_{2}=\eta_{\mu}^{v}=(v;0,1) \in \QLS(\mu)$
with $v \in \WSu{\mu}$. We set $\eta:=\Theta_{\lambda\mu}^{-1}(\eta_{1} \otimes \eta_{2}) 
\in \QLS(\lambda+\mu)$. Then, for every $w \in W$, 
\begin{equation} \label{eq:ps}
\Xi_{\lambda\mu}(\pi_{\eta} \cdot T_{\wt(w \Rightarrow \kappa(\eta))}) = 
\pi_{\eta_{1}} \cdot T_{\wt(\ti{v}_{w} \Rightarrow \kappa(\eta_1)) + 
 \wt(w \Rightarrow \ti{v}_{w})} \otimes \pi_{\eta_{2}} \cdot 
 T_{\wt(w \Rightarrow \ti{v}_{w})}, 
\end{equation}
where $\ti{v}_{w}:=\tbmin{v}{\J_{\mu}}{w}${\rm;} 
note that $\pi_{\eta_{2}} = \pi_{\mu}^{v} = (v;0,1)$. 
\end{prop}

\begin{proof}
By Lemma~\ref{lem:conn}, 
there exist $i_{1},\,i_{2},\,\dots,\,i_{n} \in I_{\af}$ 
such that 
$f_{i_n}^{\max} \cdots f_{i_2}^{\max}f_{i_1}^{\max}\eta = 
\eta_{\lambda+\mu}^{\mcr{\lng}^{\J_{\lambda+\mu}}}$, and 
$(\ti{s}_{i_{k-1}} \cdots \ti{s}_{i_2}\ti{s}_{i_1}w)^{-1}
\ti{\alpha}_{i_{k}} \in \Delta^{+}$ for all $1 \le k \le n$.
We prove the assertion of the proposition by induction on $n$. 
Assume that $n=0$; we have $\eta=\eta_{\lambda+\mu}^{\mcr{\lng}^{\J_{\lambda+\mu}}}$. 
By \cite[Lemma~3.19\,(3)]{NS05}, we deduce that 
$\Theta_{\lambda\mu}(\eta_{\lambda+\mu}^{\mcr{\lng}^{\J_{\lambda+\mu}}}) = 
\eta_{\lambda}^{\mcr{\lng}^{\J_{\lambda}}} \otimes 
 \eta_{\mu}^{\mcr{\lng}^{\J_{\mu}}}$, and hence 
$\eta_{1}=\eta_{\lambda}^{\mcr{\lng}^{\J_{\lambda}}}$ and 
$\eta_{2}=\eta_{\mu}^{\mcr{\lng}^{\J_{\mu}}}$; 
in particular, $\kappa(\eta)=\mcr{\lng}^{\J_{\lambda+\mu}}$, 
$\kappa(\eta_1) = \mcr{\lng}^{\J_{\lambda}}$, and $v=\mcr{\lng}^{\J_{\mu}}$. 
Also, we have $\pi_{\eta}=\pi_{\lambda+\mu}^{\mcr{\lng}^{\J_{\lambda+\mu}}}$, 
$\pi_{\eta_{1}}=\pi_{\lambda}^{\mcr{\lng}^{\J_{\lambda}}}$, 
and $\pi_{\eta_{2}}=\pi_{\mu}^{\mcr{\lng}^{\J_{\mu}}}$. 
By Lemma~\ref{lem:wtS}, we have
$\wt(w \Rightarrow \kappa(\eta)) = 
\wt(w \Rightarrow \mcr{\lng}^{\J_{\lambda+\mu}}) \equiv 
\wt(w \Rightarrow \lng)$ mod $\QSv{\lambda+\mu}$. 
Since $\lng$ is greater than or equal to $w$ in the (ordinary) Bruhat order on $W$, 
there exists a (shortest) directed path in $\QB$ from 
$w$ to $\lng$ whose edges are all Bruhat edges. 
Hence we obtain $\wt(w \Rightarrow \lng) = 0$, so that
$\wt(w \Rightarrow \kappa(\eta)) \in \QSv{\lambda+\mu}$. 
Therefore, by Lemma~\ref{lem:PiJ}\,(3), 
the left-hand side of \eqref{eq:ps} is equal to 
$\Xi_{\lambda\mu}(\pi_{\lambda+\mu}^{\mcr{\lng}^{\J_{\lambda+\mu}}})$. 
Since $\ti{v}_{w} \in v\WS{\mu}=\lng\WS{\mu}$, 
it follows from Lemma~\ref{lem:wtS} that 
$\wt(w \Rightarrow \ti{v}_{w}) \equiv \wt(w \Rightarrow \lng)$ 
mod $\QSv{\mu}$. Since $\wt(w \Rightarrow \lng) = 0$ as seen above, 
we have $\wt(w \Rightarrow \ti{v}_{w}) \in \QSv{\mu}$. 
Therefore, by Lemma~\ref{lem:PiJ}\,(3), 
the second factor of the right-hand side of \eqref{eq:ps} is equal to 
$\pi_{\mu}^{\mcr{\lng}^{\J_{\mu}}}$. 
Similarly, since $\kappa(\eta_{1})=\mcr{\lng}^{\J_{\lambda}} \in \lng \WS{\lambda}$, 
we see by Lemmas~\ref{lem:wtS} and \ref{lem:PiJ}\,(3) 
that the first factor of the right-hand side of \eqref{eq:ps} is equal to 
\begin{equation*}
\pi_{\eta_{1}} \cdot T_{\wt(\ti{v}_{w} \Rightarrow \kappa(\eta_1)) + 
 \wt(w \Rightarrow \ti{v}_{w})} = 
\pi_{\eta_{1}} \cdot T_{\wt(\ti{v}_{w} \Rightarrow \lng) + 
 \wt(w \Rightarrow \ti{v}_{w})}.
\end{equation*}
Since $\ti{v}_{w} \tb{w} \lng$ by the definition of $\ti{v}_{w}$, we have 
$\wt(w \Rightarrow \lng) = 
\wt(w \Rightarrow \ti{v}_{w})+ \wt(\ti{v}_{w} \Rightarrow \lng)$. 
Since $\wt(w \Rightarrow \lng) = 0$ as seen above, 
we conclude that the first factor of the right-hand side of 
\eqref{eq:ps} is equal to 
$\pi_{\lambda}^{\mcr{\lng}^{\J_{\lambda}}}$. 
Thus, equation \eqref{eq:ps} reduces to:
\begin{equation} \label{eq:311a}
\Xi_{\lambda\mu}(\pi_{\lambda+\mu}^{\mcr{\lng}^{\J_{\lambda+\mu}}}) = 
\pi_{\lambda}^{\mcr{\lng}^{\J_{\lambda}}} \otimes 
\pi_{\mu}^{\mcr{\lng}^{\J_{\mu}}}. 
\end{equation}
This equality is verified by using 
Lemmas~\ref{lem:PiJ}\,(1), \ref{lem:pix}, and \ref{lem:ext} as follows:
\begin{equation*}
\Xi_{\lambda\mu}(\pi_{\lambda+\mu}^{\mcr{\lng}^{\J_{\lambda+\mu}}}) = 
\Xi_{\lambda\mu}(\lng \cdot \pi_{\lambda+\mu})=
(\lng \cdot \pi_{\lambda}) \otimes (\lng \cdot \pi_{\mu})=
\pi_{\lambda}^{\mcr{\lng}^{\J_{\lambda}}} \otimes 
\pi_{\mu}^{\mcr{\lng}^{\J_{\mu}}}.
\end{equation*}
This proves the assertion in the case $n=0$.

Assume that $n > 0$. For simplicity of notation, 
we set $i:=i_{1}$, $\eta':=f_{i_{1}}^{\max}\eta = f_{i}^{\max}\eta$, 
and $w':=\ti{s}_{i_1}w = \ti{s}_{i}w$. 
Note that $\vp_{i}(\eta_{2})=\max\{ \pair{v\mu}{\alpha_{i}^{\vee}},0 \}$, 
and 
\begin{equation*}
f_{i}^{\max}\eta_{2} = 
 \begin{cases}
 (\mcr{\ti{s}_{i}v}^{\J_{\mu}};0,1) & 
   \text{if $\pair{v\mu}{\alpha_{i}^{\vee}} > 0$}, \\
 (v;0,1) & 
   \text{if $\pair{v\mu}{\alpha_{i}^{\vee}} \le 0$}; 
 \end{cases}
\end{equation*}
see, e.g., \cite[Lemma~8.2.7]{KasF}.
We see by the tensor product rule for crystals, 
along with these equalities, that 
\begin{equation} \label{eq:vp}
\vp_{i}(\eta) = \vp_{i}(\eta_{1} \otimes \eta_{2}) = 
 \begin{cases}
 \vp_{i}(\eta_{1}) + \pair{v\mu}{\alpha_{i}^{\vee}} 
    & \text{if $\pair{v\mu}{\alpha_{i}^{\vee}} > 0$}, \\[1mm]
 \vp_{i}(\eta_{1}) & \text{if $\pair{v\mu}{\alpha_{i}^{\vee}} = 0$}, \\[1mm]
 \max\{ \vp_{i}(\eta_{1}) + \pair{v\mu}{\alpha_{i}^{\vee}},0\} 
    & \text{if $\pair{v\mu}{\alpha_{i}^{\vee}} < 0$}, \\[1mm]
 \end{cases}
\end{equation}
and
\begin{align}
\Theta_{\lambda\mu}(f_{i}^{\max}\eta) = 
f_{i}^{\max}(\eta_{1} \otimes \eta_{2}) & = 
  \begin{cases}
  f_{i}^{\max}\eta_{1} \otimes f_{i}^{\max}\eta_{2} 
     & \text{if $\pair{v\mu}{\alpha_{i}^{\vee}} > 0$}, \\[1.5mm]
  f_{i}^{\max}\eta_{1} \otimes \eta_{2} 
     & \text{if $\pair{v\mu}{\alpha_{i}^{\vee}} = 0$}, \\[1.5mm]
  f_{i}^{\vp_{i}(\eta)}\eta_{1} \otimes \eta_{2}
     & \text{if $\pair{v\mu}{\alpha_{i}^{\vee}} < 0$},
  \end{cases} \nonumber \\[2mm]
& = 
  \begin{cases}
  f_{i}^{\max}\eta_{1} \otimes (\mcr{\ti{s}_{i}v}^{\J_{\mu}};0,1)
     & \text{if $\pair{v\mu}{\alpha_{i}^{\vee}} > 0$}, \\[1.5mm]
  f_{i}^{\max}\eta_{1} \otimes (v;0,1)
     & \text{if $\pair{v\mu}{\alpha_{i}^{\vee}} = 0$}, \\[1.5mm]
  f_{i}^{\vp_{i}(\eta)}\eta_{1} \otimes (v;0,1)
     & \text{if $\pair{v\mu}{\alpha_{i}^{\vee}} < 0$}. 
  \end{cases} \label{eq:s1}
\end{align}
We define $\eta_{1}' \in \QLS(\lambda)$ and 
$\eta_{2}' \in \QLS(\mu)$ by: 
$\eta_{1}' \otimes \eta_{2}'= \Theta_{\lambda\mu}(\eta') = 
\Theta_{\lambda\mu}(f_{i}^{\max}\eta) = 
f_{i}^{\max}(\eta_{1} \otimes \eta_{2})$. 
By \eqref{eq:s1}, we can apply our induction hypothesis to 
$\Theta_{\lambda\mu}(\eta')=\eta_{1}' \otimes \eta_{2}'$ and $w'=\ti{s}_{i}w$ 
to obtain
%
%
\begin{equation} \label{eq:s2}
\Xi_{\lambda\mu}(\pi_{\eta'} \cdot T_{\wt(w' \Rightarrow \kappa(\eta'))}) = 
\pi_{\eta_{1}'} \cdot T_{\wt(\ti{v}'_{w'} \Rightarrow \kappa(\eta_1')) + 
 \wt(w' \Rightarrow \ti{v}'_{w'})} \otimes \pi_{\eta_{2}'} \cdot 
 T_{\wt(w' \Rightarrow \ti{v}'_{w'})}, 
\end{equation}
where  
\begin{equation*}
v':=
  \begin{cases}
  \mcr{\ti{s}_{i}v}^{\J_{\mu}}
  & \text{if $\pair{v\mu}{\alpha_{i}^{\vee}} > 0$}, \\[1.5mm]
  v & \text{if $\pair{v\mu}{\alpha_{i}^{\vee}} \le 0$},
  \end{cases}
\qquad \text{and} \quad
\ti{v}'_{w'}:=\tbmin{v'}{\J_{\mu}}{w'}. 
\end{equation*}
%
%
\begin{claim} \label{c:s1}
It holds that 
%
%
\begin{equation} \label{eq:cs1a}
\pi_{\eta'} \cdot T_{\wt(w' \Rightarrow \kappa(\eta'))} = 
f_{i}^{\max}\pi_{\eta} \cdot 
T_{\wt(w \Rightarrow \kappa(\eta))-\delta_{i0}w^{-1}\ti{\alpha}_{0}^{\vee}}. 
\end{equation}
\end{claim}

\noindent
{\it Proof of Claim~\ref{c:s1}.} 
Assume first that $\pair{\kappa(\eta)(\lambda+\mu)}{\alpha_{i}^{\vee}} > 0$; 
we see from Lemma~\ref{lem:kap}\,(2) that 
$\kappa(\eta') = \kappa(f_{i}^{\max}\eta) = 
 \mcr{\ti{s}_{i}\kappa(\eta)}^{\J_{\lambda+\mu}}$. 
It follows from Lemma~\ref{lem:wtS} that
$\wt(w' \Rightarrow \kappa(\eta')) = 
 \wt(\ti{s}_{i}w \Rightarrow \mcr{\ti{s}_{i}\kappa(\eta)}^{\J_{\lambda+\mu}})
 \equiv 
 \wt(\ti{s}_{i}w \Rightarrow \ti{s}_{i}\kappa(\eta))$ mod 
$\QSv{\lambda+\mu}$. Hence, 
$\pi_{\eta'} \cdot T_{\wt(w' \Rightarrow \kappa(\eta'))} = 
 \pi_{\eta'} \cdot T_{\wt(\ti{s}_{i}w \Rightarrow \ti{s}_{i}\kappa(\eta))}$ 
by Lemma~\ref{lem:PiJ}\,(3). 
Since $\kappa(\eta)^{-1}\ti{\alpha}_{i} \in \Delta^{+}$ 
and $w^{-1}\ti{\alpha}_{i} \in \Delta^{+}$, 
we see from Lemma~\ref{lem:dia}\,(2) that 
$\wt(\ti{s}_{i}w \Rightarrow \ti{s}_{i}\kappa(\eta)) = 
\wt(w \Rightarrow \kappa(\eta)) - \delta_{i0}w^{-1}\ti{\alpha}_{0}^{\vee} + 
\delta_{i0}\kappa(\eta)^{-1}\ti{\alpha}_{0}^{\vee}$. 
Also, we have 
$\pi_{\eta'} = \pi_{f_{i}^{\max}\eta} = 
f_{i}^{\max}\pi_{\eta} \cdot T_{-\delta_{i0}\kappa(\eta)^{-1}\ti{\alpha}_{0}^{\vee}}$ 
by Lemma~\ref{lem:rec} and the assumption that 
$\pair{\kappa(\eta)(\lambda+\mu)}{\alpha_{i}^{\vee}} > 0$. 
Combining these equalities, we obtain \eqref{eq:cs1a}. 

Assume next that $\pair{\kappa(\eta)(\lambda+\mu)}{\alpha_{i}^{\vee}} \le 0$; 
we see from Lemma~\ref{lem:kap}\,(2) that 
$\kappa(\eta') = \kappa(f_{i}^{\max}\eta) = \kappa(\eta)$. 
We claim that there exists $z \in \WS{\lambda+\mu}$ such that 
$(\kappa(\eta')z)^{-1}\ti{\alpha}_{i} = 
(\kappa(\eta)z)^{-1}\ti{\alpha}_{i} \in \Delta^{-}$. 
Indeed, since $\pair{\kappa(\eta)(\lambda+\mu)}{\alpha_{i}^{\vee}} \le 0$, 
we see that $\alpha:=\kappa(\eta')^{-1}\ti{\alpha}_{i} = 
\kappa(\eta)^{-1}\ti{\alpha}_{i} \in \Delta^{-} \cup \DeS{\lambda+\mu}$. 
If $\alpha \in \Delta^{-}$ (resp., $\alpha \in \DeS{\lambda+\mu}^{+}$), 
then $z=e$ (resp., $z=s_{\alpha}$) satisfies the condition that 
$z^{-1}\alpha \in \Delta^{-}$. By Lemma~\ref{lem:wtS}, we have
\begin{equation*}
\wt(w' \Rightarrow \kappa(\eta')) = 
\wt(\ti{s}_{i}w \Rightarrow \kappa(\eta)) \equiv 
\wt(\ti{s}_{i}w \Rightarrow \kappa(\eta)z) \mod \QSv{\lambda+\mu}. 
\end{equation*}
Since $w^{-1}\ti{\alpha}_{i} \in \Delta^{+}$ and 
$(\kappa(\eta)z)^{-1}\ti{\alpha}_{i} \in \Delta^{-}$, 
it follows from Lemma~\ref{lem:dia}\,(1) that
$\wt(\ti{s}_{i}w \Rightarrow \kappa(\eta)z) = 
\wt(w \Rightarrow \kappa(\eta)z)-\delta_{i0}w^{-1}\ti{\alpha}_{0}^{\vee}$. 
Also, it follows from Lemma~\ref{lem:wtS} that
$\wt(w \Rightarrow \kappa(\eta)z) \equiv 
 \wt(w \Rightarrow \kappa(\eta))$ mod $\QSv{\lambda+\mu}$.
Therefore, we have
\begin{equation*}
\wt(w' \Rightarrow \kappa(\eta')) \equiv 
\wt(w \Rightarrow \kappa(\eta))-\delta_{i0}w^{-1}\ti{\alpha}_{0}^{\vee} 
\mod \QSv{\lambda+\mu}, 
\end{equation*}
and hence $\pi_{\eta'} \cdot T_{\wt(w' \Rightarrow \kappa(\eta'))} = 
\pi_{\eta'} \cdot
T_{\wt(w \Rightarrow \kappa(\eta))-\delta_{i0}w^{-1}\ti{\alpha}_{0}^{\vee}}$ 
by Lemma~\ref{lem:PiJ}\,(3). 
Because $\pi_{\eta'} = \pi_{f_{i}^{\max}\eta} = f_{i}^{\max}\pi_{\eta}$ 
by Lemma~\ref{lem:rec} and the assumption that 
$\pair{\kappa(\eta)(\lambda+\mu)}{\alpha_{i}^{\vee}} \le 0$, 
we obtain \eqref{eq:cs1a}. This proves Claim~\ref{c:s1}. \bqed
%
%
\begin{claim} \label{c:s2}
It holds that 
%
%
\begin{equation} \label{eq:cs2a}
\pi_{\eta_{2}'} \cdot T_{\wt(w' \Rightarrow \ti{v}'_{w'})} = 
\begin{cases}
f_{i}^{\max}\pi_{\eta_{2}} \cdot 
 T_{\wt(w \Rightarrow \ti{v}_{w}) - \delta_{i0}w^{-1}\ti{\alpha}_{0}^{\vee}}
 & \text{if $\pair{v\mu}{\alpha_{i}^{\vee}} > 0$}, \\[1.5mm]
 \pi_{\eta_{2}} \cdot 
 T_{\wt(w \Rightarrow \ti{v}_{w}) - \delta_{i0}w^{-1}\ti{\alpha}_{0}^{\vee}}
 & \text{if $\pair{v\mu}{\alpha_{i}^{\vee}} \le 0$}. 
\end{cases}
\end{equation}
\end{claim}

\noindent
{\it Proof of Claim~\ref{c:s2}.} 
Assume first that $\pair{v\mu}{\alpha_{i}^{\vee}} > 0$; 
note that $v^{-1}\ti{\alpha}_{i} \in \Delta^{+} \setminus \DeS{\mu}^{+}$ and 
$v'=\mcr{\ti{s}_{i}v}^{\J_{\mu}}$, so that
$\ti{v}'_{w'}=\tbmin{\ti{s}_{i}v}{\J_{\mu}}{\ti{s}_{i}w}$. 
Since $w^{-1}\ti{\alpha}_{i} \in \Delta^{+}$, 
we deduce from Lemma~\ref{lem:tb}\,(1) that 
$\ti{v}'_{w'} = \ti{s}_{i}\ti{v}_{w}$. 
Since $w^{-1}\ti{\alpha}_{i} \in \Delta^{+}$ and 
$\ti{v}_{w}^{-1}\ti{\alpha}_{i} \in \Delta^{+}$, 
it follows from Lemma~\ref{lem:dia}\,(2) that 
%
%
\begin{equation} \label{eq:cs2d}
\wt(w' \Rightarrow \ti{v}'_{w'}) 
 = \wt(\ti{s}_{i}w \Rightarrow \ti{s}_{i}\ti{v}_{w})
 = \wt(w \Rightarrow \ti{v}_{w})-\delta_{i0}w^{-1}\ti{\alpha}_{0}^{\vee} + 
\delta_{i0}\ti{v}_{w}^{-1}\ti{\alpha}_{0}^{\vee}. 
\end{equation}
Since $\ti{v}_{w} \in v\WS{\mu}$, 
we have $\ti{v}_{w}^{-1}\ti{\alpha}_{0}^{\vee} \equiv 
v^{-1}\ti{\alpha}_{0}^{\vee}$ mod $\QSv{\mu}$. Hence
\begin{equation*}
\wt(w' \Rightarrow \ti{v}'_{w'}) 
\equiv \wt(w \Rightarrow \ti{v}_{w})-\delta_{i0}w^{-1}\ti{\alpha}_{0}^{\vee} + 
\delta_{i0}v^{-1}\ti{\alpha}_{0}^{\vee} \mod \QSv{\mu}.
\end{equation*}
Because $\eta_{2}=(v;0,1) \in \QLS(\mu)$ with $\pair{v\mu}{\alpha_{i}^{\vee}} > 0$, and 
$\eta_{2}'=f_{i}^{\max}\eta_{2}$ (see \eqref{eq:s1}), 
we have $\pi_{\eta_{2}'} = \pi_{f_{i}^{\max}\eta_{2}}=f_{i}^{\max}\pi_{\eta_{2}}
\cdot T_{-\delta_{i0}v^{-1}\ti{\alpha}_{0}^{\vee}}$ by Lemma~\ref{lem:rec}. 
Therefore, we see that
\begin{align*}
\pi_{\eta_{2}'} \cdot T_{\wt(w' \Rightarrow \ti{v}'_{w'})}
 & = f_{i}^{\max}\pi_{\eta_{2}}
\cdot T_{-\delta_{i0}v^{-1}\ti{\alpha}_{0}^{\vee} + 
\wt(w \Rightarrow \ti{v}_{w})-\delta_{i0}w^{-1}\ti{\alpha}_{0}^{\vee} + 
\delta_{i0}\ti{v}_{w}^{-1}\ti{\alpha}_{0}^{\vee}} \\
 & = f_{i}^{\max}\pi_{\eta_{2}}
\cdot T_{-\delta_{i0}v^{-1}\ti{\alpha}_{0}^{\vee} + 
\wt(w \Rightarrow \ti{v}_{w})-\delta_{i0}w^{-1}\ti{\alpha}_{0}^{\vee} + 
\delta_{i0}v^{-1}\ti{\alpha}_{0}^{\vee}} \quad \text{by Lemma~\ref{lem:PiJ}\,(3)} \\
 & = f_{i}^{\max}\pi_{\eta_{2}}
\cdot T_{\wt(w \Rightarrow \ti{v}_{w})-\delta_{i0}w^{-1}\ti{\alpha}_{0}^{\vee}}. 
\end{align*}

Assume next that $\pair{v\mu}{\alpha_{i}^{\vee}} \le 0$; 
note that $v^{-1}\ti{\alpha}_{i} \in \Delta^{-} \cup \DeS{\mu}$ and $v'=v$, 
so that $\ti{v}'_{w'}=\tbmin{v}{\J_{\mu}}{\ti{s}_{i}w}$. We claim that
%
%
\begin{equation} \label{eq:cs2c}
\begin{split}
& \wt(w' \Rightarrow \ti{v}'_{w'}) = \\
& 
\begin{cases}
\wt(w \Rightarrow \ti{v}_{w}) - \delta_{i0}w^{-1}\ti{\alpha}_{0}^{\vee}
  & \text{if $v^{-1}\ti{\alpha}_{i} \in \Delta^{-} \setminus \DeS{\mu}^{-}$}, \\[1mm]
\wt(w \Rightarrow \ti{v}_{w}) - \delta_{i0}w^{-1}\ti{\alpha}_{0}^{\vee}
  & \text{if $v^{-1}\ti{\alpha}_{i} \in \DeS{\mu}$ and 
    $(\ti{v}'_{w'})^{-1}\ti{\alpha}_{i} \in \Delta^{+}$}, \\[1mm]
\wt(w \Rightarrow \ti{v}_{w})-\delta_{i0}w^{-1}\ti{\alpha}_{0}^{\vee} + 
\delta_{i0}\ti{v}_{w}^{-1}\ti{\alpha}_{0}^{\vee}
  & \text{if $v^{-1}\ti{\alpha}_{i} \in \DeS{\mu}$ and 
    $(\ti{v}'_{w'})^{-1}\ti{\alpha}_{i} \in \Delta^{-}$}. 
\end{cases}%
\end{split}
\end{equation}
Indeed, since $w^{-1}\ti{\alpha}_{i} \in \Delta^{+}$, 
we deduce from Lemma~\ref{lem:tb}\,(2) and (3) that
%
%
\begin{equation} \label{eq:cs2b}
\ti{v}'_{w'} = 
\begin{cases}
\ti{v}_{w} & 
  \text{if $v^{-1}\ti{\alpha}_{i} \in \Delta^{-} \setminus \DeS{\mu}^{-}$}, \\[1mm]
\ti{v}_{w} & 
  \text{if $v^{-1}\ti{\alpha}_{i} \in \DeS{\mu}$ 
        and $(\ti{v}'_{w'})^{-1}\ti{\alpha}_{i} \in \Delta^{+}$}, \\[1mm]
\ti{s}_{i}\ti{v}_{w} & 
  \text{if $v^{-1}\ti{\alpha}_{i} \in \DeS{\mu}$ 
        and $(\ti{v}'_{w'})^{-1}\ti{\alpha}_{i} \in \Delta^{-}$}. 
\end{cases}
\end{equation}
In the first case, we have $\ti{v}_{w}^{-1}\ti{\alpha}_{i} \in \Delta^{-}$ 
since $\ti{v}_{w} \in v\WS{\mu}$. 
Also, recall that $w^{-1}\ti{\alpha}_{i} \in \Delta^{+}$. 
It follows from Lemma~\ref{lem:dia}\,(1) that
$\wt(w' \Rightarrow \ti{v}'_{w'}) = 
\wt(\ti{s}_{i}w \Rightarrow \ti{v}_{w}) = 
\wt(w \Rightarrow \ti{v}_{w}) - \delta_{i0}w^{-1}\ti{\alpha}_{0}^{\vee}$. 
In the third case, since $w^{-1}\ti{\alpha}_{i} \in \Delta^{+}$ and 
$\ti{v}_{w}^{-1}\ti{\alpha}_{i} \in \Delta^{+}$, 
by the same argument as for \eqref{eq:cs2d}, we deduce that 
$\wt(w' \Rightarrow \ti{v}'_{w'}) = \wt(w \Rightarrow \ti{v}_{w})
-\delta_{i0}w^{-1}\ti{\alpha}_{0}^{\vee} + 
\delta_{i0}\ti{v}_{w}^{-1}\ti{\alpha}_{0}^{\vee}$. 
Let us consider the second case. 
Since $\ti{v}_{w}^{-1}\ti{\alpha}_{i} \in \Delta^{+}$, 
we see by Lemma~\ref{lem:edge} that 
$\ti{v}_{w} \edge{\ti{v}_{w}^{-1}\ti{\alpha}_{i}} \ti{s}_{i}\ti{v}_{w}$ 
is a directed edge of $\QB$; 
in particular, $\wt (\ti{v}_{w} \Rightarrow \ti{s}_{i}\ti{v}_{w}) = 
\delta_{i0}\ti{v}_{w}^{-1}\ti{\alpha}_{0}^{\vee}$. 
Because $\ti{v}_{w} = \ti{v}'_{w'} = \tbmin{v'}{\J_{\mu}}{w'} = 
\tbmin{v}{\J_{\mu}}{\ti{s}_{i}w}$, and 
because $\ti{s}_{i}\ti{v}_{w} \in v\WS{\mu}$ in this case, 
we have $\ti{v}_{w} \tb{\ti{s}_{i}w} \ti{s}_{i}\ti{v}_{w}$, and hence
\begin{equation*}
\wt(\ti{s}_{i}w \Rightarrow \ti{s}_{i}\ti{v}_{w}) = 
\underbrace{\wt(\ti{s}_{i}w \Rightarrow \ti{v}_{w})}_{=\wt(w' \Rightarrow \ti{v}'_{w'})} + 
\underbrace{\wt(\ti{v}_{w} \Rightarrow \ti{s}_{i}\ti{v}_{w})}_{%
=\delta_{i0}\ti{v}_{w}^{-1}\ti{\alpha}_{0}^{\vee}}.
\end{equation*}
By the same argument as for \eqref{eq:cs2d}, we deduce that 
$\wt(\ti{s}_{i}w \Rightarrow \ti{s}_{i}\ti{v}_{w}) = \wt(w \Rightarrow \ti{v}_{w})
-\delta_{i0}w^{-1}\ti{\alpha}_{0}^{\vee} + 
\delta_{i0}\ti{v}_{w}^{-1}\ti{\alpha}_{0}^{\vee}$. 
Hence we obtain $\wt(w' \Rightarrow \ti{v}'_{w'}) = 
\wt(w \Rightarrow \ti{v}_{w}) - \delta_{i0}w^{-1}\ti{\alpha}_{0}^{\vee}$. 
This shows \eqref{eq:cs2c}. 

Now, we remark that if $v^{-1}\ti{\alpha}_{i} \in \DeS{\mu}$, then 
$\delta_{i0}\ti{v}_{w}^{-1}\ti{\alpha}_{0}^{\vee} \in \QSv{\mu}$. 
Hence we see by \eqref{eq:cs2c} that 
$\wt(w' \Rightarrow \ti{v}'_{w'}) 
\equiv 
\wt(w \Rightarrow \ti{v}_{w})-\delta_{i0}w^{-1}\ti{\alpha}_{0}^{\vee}$ mod $\QSv{\mu}$. 
Also, recall that $\eta_{2}' = \eta_{2}$ 
in the case that $\pair{v\mu}{\alpha_{i}^{\vee}} \le 0$. 
Therefore, we obtain
$\pi_{\eta_{2}'} \cdot T_{\wt(w' \Rightarrow \ti{v}'_{w'})} = 
 \pi_{\eta_{2}} \cdot 
 T_{\wt(w \Rightarrow \ti{v}_{w}) - \delta_{i0}w^{-1}\ti{\alpha}_{0}^{\vee}}$
by Lemma~\ref{lem:PiJ}\,(3). Thus we have shown Claim~\ref{c:s2}. \bqed
%
%
\begin{claim} \label{c:s3}
It holds that 
%
%
\begin{equation} \label{eq:cs3a}
\begin{split}
& \pi_{\eta_{1}'} \cdot 
  T_{\wt(\ti{v}'_{w'} \Rightarrow \kappa(\eta_1')) + \wt(w' \Rightarrow \ti{v}'_{w'})} \\
& =
\begin{cases}
f_{i}^{\max}\pi_{\eta_{1}} \cdot 
   T_{\wt(\ti{v}_{w} \Rightarrow \kappa(\eta_1)) + \wt(w \Rightarrow \ti{v}_{w}) -
      \delta_{i0}w^{-1}\ti{\alpha}_{0}^{\vee}}
   & \text{if $\pair{v\mu}{\alpha_{i}^{\vee}} \ge 0$}, \\[1mm]
f_{i}^{\vp_{i}(\eta)}\pi_{\eta_{1}} \cdot 
   T_{\wt(\ti{v}_{w} \Rightarrow \kappa(\eta_1)) + \wt(w \Rightarrow \ti{v}_{w}) -
      \delta_{i0}w^{-1}\ti{\alpha}_{0}^{\vee}}
   & \text{if $\pair{v\mu}{\alpha_{i}^{\vee}} < 0$}.
\end{cases}
\end{split}
\end{equation}
\end{claim}

\noindent
{\it Proof of Claim~\ref{c:s3}.} 
We know from \eqref{eq:cs2d} and \eqref{eq:cs2c} that
%
%
\begin{equation} \label{eq:cs3b}
\begin{split}
& \wt(w' \Rightarrow \ti{v}'_{w'}) = \\
& 
\begin{cases}
\wt(w \Rightarrow \ti{v}_{w})-\delta_{i0}w^{-1}\ti{\alpha}_{0}^{\vee} + 
\delta_{i0}\ti{v}_{w}^{-1}\ti{\alpha}_{0}^{\vee}
  & \text{if $v^{-1}\ti{\alpha}_{i} \in \Delta^{+} \setminus \DeS{\mu}^{+}$}, \\[1mm]
\wt(w \Rightarrow \ti{v}_{w}) - \delta_{i0}w^{-1}\ti{\alpha}_{0}^{\vee}
  & \text{if $v^{-1}\ti{\alpha}_{i} \in \Delta^{-} \setminus \DeS{\mu}^{-}$}, \\[1mm]
\wt(w \Rightarrow \ti{v}_{w}) - \delta_{i0}w^{-1}\ti{\alpha}_{0}^{\vee}
  & \text{if $v^{-1}\ti{\alpha}_{i} \in \DeS{\mu}$ and 
    $(\ti{v}'_{w'})^{-1}\ti{\alpha}_{i} \in \Delta^{+}$}, \\[1mm]
\wt(w \Rightarrow \ti{v}_{w})-\delta_{i0}w^{-1}\ti{\alpha}_{0}^{\vee} + 
\delta_{i0}\ti{v}_{w}^{-1}\ti{\alpha}_{0}^{\vee}
  & \text{if $v^{-1}\ti{\alpha}_{i} \in \DeS{\mu}$ and 
    $(\ti{v}'_{w'})^{-1}\ti{\alpha}_{i} \in \Delta^{-}$}. 
\end{cases}%
\end{split}
\end{equation}
%
\paragraph{Case 1.}
%
Assume that $\pair{v\mu}{\alpha_{i}^{\vee}} \ge 0$; 
note that $v^{-1}\ti{\alpha}_{i} \in \Delta^{+} \cup \DeS{\mu}$. 
Since $\eta_{1}'=f_{i}^{\max}\eta_{1}$ in this case (see \eqref{eq:s1}), 
we see that 
if $\pair{\kappa(\eta_{1})\lambda}{\alpha_{i}^{\vee}} > 0$ (resp., $\le 0$), 
then $\kappa(\eta_{1}')=\mcr{\ti{s}_{i}\kappa(\eta_{1})}^{\J_{\lambda}}$ 
(resp., $\kappa(\eta_{1}')=\kappa(\eta_{1})$) by Lemma~\ref{lem:kap}\,(2). 
Also, we deduce from Lemma~\ref{lem:rec} that
%
%
\begin{equation} \label{eq:cs3-1b}
\pi_{\eta_{1}'} = \pi_{f_{i}^{\max}\eta_{1}} = 
\begin{cases}
f_{i}^{\max}\pi_{\eta_{1}} \cdot 
T_{-\delta_{i0}\kappa(\eta_{1})^{-1}\ti{\alpha}_{0}^{\vee}} 
 & \text{if $\pair{\kappa(\eta_{1})\lambda}{\alpha_{i}^{\vee}} > 0$}, \\[1mm]
f_{i}^{\max}\pi_{\eta_{1}}
 & \text{if $\pair{\kappa(\eta_{1})\lambda}{\alpha_{i}^{\vee}} \le 0$}.
\end{cases}
\end{equation}
Since $w^{-1}\ti{\alpha}_{i} \in \Delta^{+}$, 
it follows from Lemma~\ref{lem:tb}\,(1) and (3) that 
\begin{equation*}
\ti{v}'_{w'} = 
 \begin{cases}
 \ti{s}_{i}\ti{v}_{w} 
  & \text{in Cases 1a and 1c}, \\[1mm]
 \ti{v}_{w} 
  & \text{in Case 1b},
 \end{cases}
\end{equation*}
where 

Case 1a: $v^{-1}\ti{\alpha}_{i} \in \Delta^{+} \setminus \DeS{\mu}^{+}$; 

Case 1b: $v^{-1}\ti{\alpha}_{i} \in \DeS{\mu}$ and $(\ti{v}'_{w'})^{-1}\ti{\alpha}_{i} \in \Delta^{+}$; 

Case 1c: $v^{-1}\ti{\alpha}_{i} \in \DeS{\mu}$ and $(\ti{v}'_{w'})^{-1}\ti{\alpha}_{i} \in \Delta^{-}$; 

\noindent
notice that $\ti{v}_{w}^{-1}\ti{\alpha}_{i} \in \Delta^{+}$ in all of these cases. 

Now, assume first that $\pair{\kappa(\eta_{1})\lambda}{\alpha_{i}^{\vee}} > 0$; 
note that $\kappa(\eta_{1})^{-1}\ti{\alpha}_{i} \in \Delta^{+}$. 
Then we see by Lemmas~\ref{lem:dia} and \ref{lem:wtS} that
%
%
\begin{align}
\wt(\ti{v}'_{w'} \Rightarrow \kappa(\eta_1')) & \equiv
\begin{cases}
\wt(\ti{s}_{i}\ti{v}_{w} \Rightarrow \ti{s}_{i}\kappa(\eta_1)) 
\mod \QSv{\lambda} 
  & \text{in Cases 1a and 1c}, \\[1mm]
\wt(\ti{v}_{w} \Rightarrow \ti{s}_{i}\kappa(\eta_1)) 
\mod \QSv{\lambda} 
  & \text{in Case 1b},
\end{cases} \nonumber \\[2mm]
& = 
\begin{cases}
\wt(\ti{v}_{w} \Rightarrow \kappa(\eta_1)) 
- \delta_{i0}\ti{v}_{w}^{-1}\ti{\alpha}_{0}^{\vee}
+ \delta_{i0}\kappa(\eta_1)^{-1}\ti{\alpha}_{0}^{\vee}
  & \text{in Cases 1a and 1c}, \\[1mm]
\wt(\ti{v}_{w} \Rightarrow \kappa(\eta_1)) 
+ \delta_{i0}\kappa(\eta_1)^{-1}\ti{\alpha}_{0}^{\vee}
  & \text{in Case 1b}. 
\end{cases} \label{eq:cs3-1a}
\end{align}
Combining \eqref{eq:cs3b}, \eqref{eq:cs3-1b}, and \eqref{eq:cs3-1a}, 
along with Lemma~\ref{lem:PiJ}\,(3), 
we obtain \eqref{eq:cs3a} in this case. Assume next that 
$\pair{\kappa(\eta_{1})\lambda}{\alpha_{i}^{\vee}} \le 0$; 
note that $\kappa(\eta_{1})^{-1}\ti{\alpha}_{i} \in \Delta^{-} \cup \DeS{\lambda}$. 
In Case 1b, we have
%
%
\begin{equation} \label{eq:cs3-1d}
\wt(\ti{v}'_{w'} \Rightarrow \kappa(\eta_1')) = 
\wt(\ti{v}_{w} \Rightarrow \kappa(\eta_1)).
\end{equation}
By \eqref{eq:cs3b}, \eqref{eq:cs3-1b}, \eqref{eq:cs3-1d}, 
we obtain \eqref{eq:cs3a} in this case. 
Let us consider Cases 1a and 1c. 
By the same argument as in the proof of Claim~\ref{c:s1}, we see that 
there exists $z \in \WS{\lambda}$ such that 
$(\kappa(\eta_{1})z)^{-1}\ti{\alpha}_{i} \in \Delta^{-}$. 
Then we see by Lemmas~\ref{lem:wtS} and \ref{lem:dia} that
%
%
\begin{align}
\wt(\ti{v}'_{w'} \Rightarrow \kappa(\eta_1')) & = 
\wt(\ti{s}_{i}\ti{v}_{w} \Rightarrow \kappa(\eta_1)) \equiv 
\wt(\ti{s}_{i}\ti{v}_{w} \Rightarrow \kappa(\eta_1)z) \nonumber \\ 
& = \wt(\ti{v}_{w} \Rightarrow \kappa(\eta_1)z) 
- \delta_{i0}\ti{v}_{w}^{-1}\ti{\alpha}_{0}^{\vee} \nonumber \\ 
& \equiv
\wt(\ti{v}_{w} \Rightarrow \kappa(\eta_1)) 
- \delta_{i0}\ti{v}_{w}^{-1}\ti{\alpha}_{0}^{\vee} \mod \QSv{\lambda}. \label{eq:cs3-1c}
\end{align}
By \eqref{eq:cs3b}, \eqref{eq:cs3-1b}, \eqref{eq:cs3-1c}, 
along with Lemma~\ref{lem:PiJ}\,(3), we obtain \eqref{eq:cs3a} also in this case. 

\paragraph{Case 2.}
%
Assume that $v^{-1}\ti{\alpha}_{i} \in \Delta^{-} \setminus \DeS{\mu}^{-}$; 
note that $\pair{v\mu}{\alpha_{i}^{\vee}} < 0$. 
Recall that $\eta_{1}'=f_{i}^{\vp_{i}(\eta)}\eta_{1}$ (see \eqref{eq:s1}); 
by \eqref{eq:vp}, we see that $\vp_{i}(\eta) = 0$ or 
$\vp_{i}(\eta) < \vp_{i}(\eta_{1})$. 
In both cases, we deduce from Lemma~\ref{lem:kap}\,(2) and 
Lemma~\ref{lem:rec} that $\kappa(\eta_{1}')=\kappa(\eta_{1})$ and 
$\pi_{\eta_{1}'}=f_{i}^{\vp_{i}(\eta)}\pi_{\eta_{1}}$. 
Recall that $v'=v$, and hence $\ti{v}'_{w'} = 
\tbmin{v}{\J_{\mu}}{\ti{s}_{i}w}$ in this case. 
Since $w^{-1}\ti{\alpha}_{i} \in \Delta^{+}$, 
we see from Lemma~\ref{lem:tb}\,(2) that 
$\ti{v}'_{w'} = \ti{v}_{w}$. Hence we obtain 
$\wt(\ti{v}'_{w'} \Rightarrow \kappa(\eta_1')) = 
\wt(\ti{v}_{w} \Rightarrow \kappa(\eta_1))$. 
By this equality, \eqref{eq:cs3b}, and 
$\pi_{\eta_{1}'}=f_{i}^{\vp_{i}(\eta)}\pi_{\eta_{1}}$, 
we obtain \eqref{eq:cs3a}. 
Thus we have shown Claim~\ref{c:s3}. \bqed

\vsp

Substituting \eqref{eq:cs1a}, \eqref{eq:cs2a}, \eqref{eq:cs3a} 
into \eqref{eq:s2}, and then using Lemma~\ref{lem:Txi2}, we deduce that 
\begin{equation} \label{eq:s4}
\begin{split}
& \Xi_{\lambda\mu}(f_{i}^{\max}\pi_{\eta} \cdot 
  T_{\wt(w \Rightarrow \kappa(\eta))}) =  \\
& 
\begin{cases}
f_{i}^{\max}\pi_{\eta_{1}} \cdot 
   T_{\wt(\ti{v}_{w} \Rightarrow \kappa(\eta_1)) + \wt(w \Rightarrow \ti{v}_{w})} \otimes 
 f_{i}^{\max}\pi_{\eta_{2}} \cdot 
 T_{\wt(w \Rightarrow \ti{v}_{w})} 
 & \text{if $\pair{v\mu}{\alpha_{i}^{\vee}} > 0$}, \\[2mm]
f_{i}^{\max}\pi_{\eta_{1}} \cdot 
   T_{\wt(\ti{v}_{w} \Rightarrow \kappa(\eta_1)) + \wt(w \Rightarrow \ti{v}_{w})} \otimes 
 \pi_{\eta_{2}} \cdot 
 T_{\wt(w \Rightarrow \ti{v}_{w})}
 & \text{if $\pair{v\mu}{\alpha_{i}^{\vee}} = 0$}, \\[2mm]
f_{i}^{\vp_{i}(\eta)}\pi_{\eta_{1}} \cdot 
   T_{\wt(\ti{v}_{w} \Rightarrow \kappa(\eta_1)) + \wt(w \Rightarrow \ti{v}_{w})} \otimes 
 \pi_{\eta_{2}} \cdot T_{\wt(w \Rightarrow \ti{v}_{w})} 
 & \text{if $\pair{v\mu}{\alpha_{i}^{\vee}} < 0$}.
\end{cases}
\end{split}
\end{equation}
Here we see from \eqref{eq:Txi1a} and \eqref{eq:cl} that 
$\vp_{i}(\pi_{\eta} \cdot 
  T_{\wt(w \Rightarrow \kappa(\eta))}) = \vp_{i}(\eta)$. 
Hence the left-hand side of \eqref{eq:s4} is equal to 
$f_{i}^{\vp_{i}(\eta)}\Xi_{\lambda\mu}
(\pi_{\eta} \cdot T_{\wt(w \Rightarrow \kappa(\eta))})$. 
Similarly, it is easily verified, using \eqref{eq:Txi1a}, \eqref{eq:cl}, 
and the tensor product rule for crystals, that
\begin{equation*}
\vp_{i}(\pi_{\eta_{1}} \cdot 
   T_{\wt(\ti{v}_{w} \Rightarrow \kappa(\eta_1)) + \wt(w \Rightarrow \ti{v}_{w})} \otimes 
\pi_{\eta_{2}} \cdot T_{\wt(w \Rightarrow \ti{v}_{w})}) = \vp_{i}(\eta_{1} \otimes \eta_{2})=\vp_{i}(\eta).
\end{equation*}
Furthermore, we deduce by the tensor product rule for crystals that
the right-hand side of \eqref{eq:s4} is equal to 
\begin{equation*}
\begin{split}
& f_{i}^{\max}(\pi_{\eta_{1}} \cdot 
   T_{\wt(\ti{v}_{w} \Rightarrow \kappa(\eta_1)) + \wt(w \Rightarrow \ti{v}_{w})} \otimes 
\pi_{\eta_{2}} \cdot T_{\wt(w \Rightarrow \ti{v}_{w})}) \\
& \qquad = 
f_{i}^{\vp_{i}(\eta)}(\pi_{\eta_{1}} \cdot 
   T_{\wt(\ti{v}_{w} \Rightarrow \kappa(\eta_1)) + \wt(w \Rightarrow \ti{v}_{w})} \otimes 
\pi_{\eta_{2}} \cdot T_{\wt(w \Rightarrow \ti{v}_{w})}).
\end{split}
\end{equation*}
Therefore, we obtain
\begin{equation*}
f_{i}^{\vp_{i}(\eta)}\Xi_{\lambda\mu}
(\pi_{\eta} \cdot T_{\wt(w \Rightarrow \kappa(\eta))}) = 
f_{i}^{\vp_{i}(\eta)}(\pi_{\eta_{1}} \cdot 
   T_{\wt(\ti{v}_{w} \Rightarrow \kappa(\eta_1)) + \wt(w \Rightarrow \ti{v}_{w})} \otimes 
\pi_{\eta_{2}} \cdot T_{\wt(w \Rightarrow \ti{v}_{w})}),
\end{equation*}
and hence $\Xi_{\lambda\mu}
(\pi_{\eta} \cdot T_{\wt(w \Rightarrow \kappa(\eta))}) = \pi_{\eta_{1}} \cdot 
   T_{\wt(\ti{v}_{w} \Rightarrow \kappa(\eta_1)) + \wt(w \Rightarrow \ti{v}_{w})} \otimes 
\pi_{\eta_{2}} \cdot T_{\wt(w \Rightarrow \ti{v}_{w})}$,
as desired. This completes the proof of Proposition~\ref{prop:s}. 
\end{proof}
%
%
\begin{prop} \label{prop:s2}
Let $\lambda_{1},\,\dots,\,\lambda_{n} \in P^{+}$, and set 
$\lambda:=\lambda_{1} + \cdots + \lambda_{n}$. 
Take $\J_{k}:=\J_{\lambda_{k}}$, $1 \le k \le n$, and $\J=\J_{\lambda}$ as in \eqref{eq:J}. 
Let $v_{k} \in \WSu{\lambda_k}$ for $1 \le k \le n$, and set
\begin{equation*}
\eta:=\Theta_{\lambda_1,\dots,\lambda_n}^{-1}(
\eta_{\lambda_1}^{v_1} \otimes \cdots \otimes \eta_{\lambda_n}^{v_n}) 
\in \QLS(\lambda). 
\end{equation*}
Let $w \in W$. We define 
\begin{equation*}
\begin{cases}
  \ti{v}_{n+1}:=w, \quad 
  \ti{v}_{k}:=\tbmin{v_{k}}{\J_{k}}{\ti{v}_{k+1}} \quad 
  \text{\rm for $1 \le k \le n$}, \\[1mm]
  \xi_{n}:=\wt (\ti{v}_{n+1} \Rightarrow \ti{v}_{n}), \quad
  \xi_{k}:=\xi_{k+1} + \wt (\ti{v}_{k+1} \Rightarrow \ti{v}_{k}) 
  \quad \text{\rm for $1 \le k \le n-1$}.
\end{cases}
\end{equation*}
Then the following equality holds{\rm:}
\begin{equation} \label{eq:s2a}
\Xi_{\lambda_1,\dots,\lambda_n}(\pi_{\eta} \cdot T_{\wt (w \Rightarrow \kappa(\eta))}) = 
(\pi_{\lambda_1}^{v_1} \cdot T_{\xi_1}) \otimes \cdots \otimes 
(\pi_{\lambda_n}^{v_n} \cdot T_{\xi_n}).
\end{equation}
\end{prop}

\begin{proof}
We prove the assertion of the proposition by induction on $n$. 
Assume  that $n=1$. In this case, 
both $\Xi_{\lambda_1,\dots,\lambda_n}$ and 
$\Theta_{\lambda_1,\dots,\lambda_n}$ are the identity map.
Hence, $\eta=\eta_{\lambda_1}^{v_1}$, and 
$\pi_{\eta}=\pi_{\lambda_1}^{v_1}$. By Lemma~\ref{lem:wtS}, 
we have $\xi_{1} = \wt(w \Rightarrow \ti{v}_{1}) \equiv 
\wt(w \Rightarrow v_{1}) = 
\wt(w \Rightarrow \kappa(\eta))$ mod $\QSv{1}$. 
Therefore, we obtain 
$\pi_{\eta} \cdot T_{\wt(w \Rightarrow \kappa(\eta))} = 
 \pi_{\eta} \cdot T_{\xi_1} = \pi_{\lambda_1}^{v_1} \cdot T_{\xi_1}$
by Lemma~\ref{lem:PiJ}\,(3). 
This proves the assertion for the case $n=1$.

Assume that $n > 1$; for simplicity of notation,
we set $\lambda':=\lambda_{1} + \cdots + \lambda_{n-1}$. 
We see from Remarks~\ref{rem:ass1} and \ref{rem:ass2} that 
the following diagrams \eqref{eq:CDa} and \eqref{eq:CDb} are commutative:
\begin{equation} \label{eq:CDa}
\begin{split}
\xymatrix{%
\QLS(\lambda) \ar[r]^-{\Theta_{\lambda_1,\dots,\lambda_n}}
\ar[d]_{\Theta_{\lambda',\lambda_n}} & 
\QLS(\lambda_{1}) \otimes \cdots \otimes \QLS(\lambda_{n}) \\
\QLS(\lambda') \otimes \QLS(\lambda_{n}),
\ar[ur]_{\qquad \Theta_{\lambda_1,\dots,\lambda_{n-1}} \otimes \id} & }
\end{split}
\end{equation}
\begin{equation} \label{eq:CDb}
\begin{split}
\xymatrix{%
\SLS_{0}(\lambda) \ar[r]^-{\Xi_{\lambda_1,\dots,\lambda_n}}
\ar[d]_{\Xi_{\lambda',\lambda_n}} & 
(\SLS(\lambda_{1}) \otimes \cdots \otimes \SLS(\lambda_{n}))_{0} \\
(\SLS(\lambda') \otimes \SLS(\lambda_{n}))_{0}.
\ar[ur]_{\qquad \Xi_{\lambda_1,\dots,\lambda_{n-1}} \otimes \id} & }
\end{split}
\end{equation}
Now, we set 
$\eta':=\Theta_{\lambda_1,\dots,\lambda_{n-1}}^{-1}(
\eta_{\lambda_1}^{v_1} \otimes \cdots \otimes \eta_{\lambda_{n-1}}^{v_{n-1}}) 
\in \QLS(\lambda')$; by the commutative diagram \eqref{eq:CDa}, 
we see that $\Theta_{\lambda',\lambda_{n}}^{-1}(\eta' \otimes 
\eta_{\lambda_{n}}^{v_{n}}) = \eta$. Therefore, we deduce 
from Proposition~\ref{prop:s} that 
\begin{equation}
\Xi_{\lambda',\lambda_{n}}(\pi_{\eta} \cdot T_{\wt(w \Rightarrow \kappa(\eta))}) = 
\pi_{\eta'} \cdot T_{ \wt (\ti{v}_{n} \Rightarrow \kappa(\eta')) + \wt(w \Rightarrow \ti{v}_{n})} 
\otimes \pi_{\lambda_{n}}^{v_{n}} \cdot T_{\wt(w \Rightarrow \ti{v}_{n})};
\end{equation}
note that $\wt(w \Rightarrow \ti{v}_{n}) = \xi_{n}$. 
Also, by the induction hypothesis 
(applied to $\eta' \in \QLS(\lambda')$ and $\ti{v}_{n} \in W$), 
we have
\begin{align*}
& \Xi_{\lambda_1,\dots,\lambda_{n-1}}( 
  \pi_{\eta'} \cdot T_{ \wt (\ti{v}_{n} \Rightarrow \kappa(\eta')) + \wt(w \Rightarrow \ti{v}_{n})}) \\
& \qquad = 
  \Bigl(\Xi_{\lambda_1,\dots,\lambda_{n-1}}( 
  \pi_{\eta'} \cdot T_{ \wt (\ti{v}_{n} \Rightarrow \kappa(\eta')) } ) \Bigr)
  \cdot T_{\wt(w \Rightarrow \ti{v}_{n})} 
  \quad \text{by Lemma~\ref{lem:Txi2}} \\
& \qquad =
  \Bigl(
  (\pi_{\lambda_1}^{v_1} \cdot T_{\xi_1-\wt(w \Rightarrow \ti{v}_{n})}) \otimes \cdots \otimes 
  (\pi_{\lambda_{n-1}}^{v_{n-1}} \cdot T_{\xi_{n-1}-\wt(w \Rightarrow \ti{v}_{n})})
  \Bigr) \cdot T_{\wt(w \Rightarrow \ti{v}_{n})} \\
& \qquad =
  (\pi_{\lambda_1}^{v_1} \cdot T_{\xi_1}) \otimes \cdots \otimes 
  (\pi_{\lambda_{n-1}}^{v_{n-1}} \cdot T_{\xi_{n-1}}). 
\end{align*}
Therefore, by the commutative diagram \eqref{eq:CDb}, we obtain
\begin{align*}
\Xi_{\lambda_1,\dots,\lambda_{n}}(\pi_{\eta} \cdot T_{\wt (w \Rightarrow \kappa(\eta))}) 
& = ((\Xi_{\lambda_1,\dots,\lambda_{n-1}} \otimes \id) \circ 
 \Xi_{\lambda',\lambda_{n}})(\pi_{\eta} \cdot T_{\wt (w \Rightarrow \kappa(\eta))}) \\
& = (\pi_{\lambda_1}^{v_1} \cdot T_{\xi_1}) \otimes \cdots \otimes 
(\pi_{\lambda_n}^{v_n} \cdot T_{\xi_n}), 
\end{align*}
as desired. This proves the proposition. 
\end{proof}

\begin{proof}[Proof of Theorem~\ref{thm:main}.]
Take $N_{\lambda}$, $N_{\mu}$, $N_{\lambda+\mu}$ as in \eqref{eq:N}, 
and let $N$ be a common multiple of $N_{\lambda}$, $N_{\mu}$, $N_{\lambda+\mu}$. 
We deduce from \eqref{eq:CD4} and \eqref{eq:CD5} that 
the following diagrams are commutative:
%
%
\begin{equation} \label{eq:CD4a}
\begin{split}
\xymatrix{%
 \SLS_{0}(\lambda+\mu) 
  \ar[rr]^-{\Xi_{\lambda\mu}} 
  \ar[d]_-{\Sigma_{N}}
  \ar[ddr]^-{\Sigma_{N}'} & &
 (\SLS(\lambda) \otimes \SLS(\mu))_{0} 
 \ar[d]^-{\Sigma_{N} \otimes \Sigma_{N}} \\
 (\SLS(\lambda+\mu)^{\otimes N})_{0} & & 
 (\SLS(\lambda)^{\otimes N} \otimes \SLS(\mu)^{\otimes N})_{0} \\
 & \SLS_{0}(N\lambda+N\mu), \ar[ru]_{\Xi_{\lambda\mu}^{(N)}} 
 \ar[lu]^-{\Xi_{\lambda+\mu}^{(N)}} & 
}
\end{split}
\end{equation}
%
%
\begin{equation} \label{eq:CD5a}
\begin{split}
\xymatrix{%
 \QLS(\lambda+\mu) 
  \ar[rr]^-{\Theta_{\lambda\mu}} 
  \ar[d]_-{\Sigma_{N}}
  \ar[ddr]^-{\Sigma_{N}'} & &
 \QLS(\lambda) \otimes \QLS(\mu) 
 \ar[d]^-{\Sigma_{N} \otimes \Sigma_{N}} \\
 \QLS(\lambda+\mu)^{\otimes N} & & 
 \QLS(\lambda)^{\otimes N} \otimes \QLS(\mu)^{\otimes N} \\
 & \QLS(N\lambda+N\mu), \ar[ru]_{\Theta_{\lambda\mu}^{(N)}} 
 \ar[lu]^-{\Theta_{\lambda+\mu}^{(N)}} & 
}
\end{split}
\end{equation}
By the commutative diagram \eqref{eq:CD4a}, it suffices to show that
\begin{equation} \label{eq:main1a}
\begin{split}
& (\Xi_{\lambda\mu}^{(N)} \circ \Sigma_{N}')
  (\overbrace{ \pi_{\eta} \cdot T_{\wt(w \Rightarrow \kappa(\eta))} }^{%
  \text{ cf. LHS of \eqref{eq:main} }}) \\
& \qquad =
 (\Sigma_{N} \otimes \Sigma_{N})
  (\underbrace{%
  \pi_{\eta_{1}} \cdot 
  T_{\wt(\io{\eta_2}{w} \Rightarrow \kappa(\eta_1))+ \ze{\eta_2}{w}} \otimes 
  \pi_{\eta_{2}} \cdot 
  T_{\wt(w \Rightarrow \kappa(\eta_2))}}_{%
  \text{RHS of \eqref{eq:main}}})
\end{split}
\end{equation}
First we compute the left-hand side of \eqref{eq:main1a}. 
We set $\eta':= \Sigma_{N}'(\eta)  \in \QLS(N\lambda+N\mu)$;
note that $\kappa(\eta')=\kappa(\eta)$.
By the commutative diagrams \eqref{eq:CD} and \eqref{eq:CD3}, 
and the definition of $\Sigma_{N}'$, we see that 
\begin{equation*}
\Sigma_{N}'(\pi_{\eta} \cdot T_{\wt(w \Rightarrow \kappa(\eta))})
  = \Sigma_{N}'(\pi_{\eta}) \cdot T_{\wt(w \Rightarrow \kappa(\eta))}
  = \pi_{\eta'} \cdot T_{\wt(w \Rightarrow \kappa(\eta'))}. 
\end{equation*}
Hence the left-hand side of \eqref{eq:main1a} is identical to 
$\Xi_{\lambda\mu}^{(N)}
 (\pi_{\eta'} \cdot T_{\wt(w \Rightarrow \kappa(\eta'))})$.

Next we compute the right-hand side of \eqref{eq:main1a}. 
Assume that $\eta_{1} \in \QLS(\lambda)$ and $\eta_{2} \in \QLS(\mu)$ are of the forms: 
\begin{equation*}
\eta_{1} = (u_{1},\,\dots,u_{p};\tau_{0},\tau_{1},\dots,\tau_{p}), \qquad
\eta_{2} = (v_{1},\,\dots,v_{s};\sigma_{0},\sigma_{1},\dots,\sigma_{s}), 
\end{equation*}
respectively. We define 
\begin{equation}
\begin{cases}
\tiv{\eta_{2}}{w} = (\ti{v}_{1},\,\dots,\,\ti{v}_{s},\,\ti{v}_{s+1}=w), \\[1.5mm]
\tiv{\eta_{1}}{\ti{v}_{1}} = (\ti{u}_{1},\,\dots,\,\ti{u}_{p},\,\ti{u}_{p+1}=\ti{v}_{1}), 
\end{cases}
\quad \text{and} \quad
\begin{cases}
\tixi{\eta_{2}}{w} = (\ti{\xi}_{1},\,\dots,\,\ti{\xi}_{s}), \\[1.5mm]
\tixi{\eta_{1}}{\ti{v}_{1}} = (\ti{\gamma}_{1},\,\dots,\,\ti{\gamma}_{p}), 
\end{cases}
\end{equation}
as in \eqref{eq:ti1} and \eqref{eq:ti2}, respectively; 
recall that $\io{\eta_2}{w}=\ti{v}_{1}$ and $\ze{\eta_2}{w}=\ti{\xi}_{1}$. 
We claim that
\begin{align}
& \pi_{\eta_1} \cdot T_{\wt(\io{\eta_2}{w} \Rightarrow \kappa(\eta_1))+ \ze{\eta_2}{w}}
  = (u_{1}\PS{\lambda}(t_{\ti{\gamma}_1+\ti{\xi}_1}),\,\dots,
     u_{p}\PS{\lambda}(t_{\ti{\gamma}_p+\ti{\xi}_1});
     \tau_{0},\tau_{1},\dots,\tau_{p}), \label{eq:23a} \\
& \pi_{\eta_2} \cdot T_{\wt(w \Rightarrow \kappa(\eta_{2}))}
  = (v_{1}\PS{\mu}(t_{\ti{\xi}_1}),\,\dots,
     v_{s}\PS{\mu}(t_{\ti{\xi}_s});
     \sigma_{0},\sigma_{1},\dots,\sigma_{s}). \label{eq:23b}
\end{align}
Let us show \eqref{eq:23b}; the proof of \eqref{eq:23a} is similar. 
We define $\bxi{\eta_{2}} = (\xi_{1},\,\dots,\,\xi_{s-1},\xi_{s}=0)$
as in \eqref{eq:bxi}. Then, by Remark~\ref{rem:equiv}, 
we have $\ti{\xi}_{u} \equiv \xi_{u}+\wt^{\J_{\mu}}
(\mcr{w}^{\J_{\mu}} \Rightarrow \kappa(\eta_{2}))$ mod $\QSv{\mu}$ 
for all $1 \le u \le s$. By Lemma~\ref{lem:wtS}, we have 
$\wt^{\J_{\mu}}
(\mcr{w}^{\J_{\mu}} \Rightarrow \kappa(\eta_{2}))
\equiv \wt(w \Rightarrow \kappa(\eta_{2}))$ mod $\QSv{\mu}$, 
and hence $\ti{\xi}_{u} \equiv \xi_{u} + 
\wt(w \Rightarrow \kappa(\eta_{2}))$ mod $\QSv{\mu}$ 
for all $1 \le u \le s$. From these, we see that
\begin{align*}
 \pi_{\eta_2} \cdot T_{\wt(w \Rightarrow \kappa(\eta_{2}))} 
 &   \stackrel{\eqref{eq:pieta}}{=} (v_{1}\PS{\mu}(t_{\xi_1}),\,\dots,
     v_{s}\PS{\mu}(t_{\xi_s}); 
     \sigma_{0},\sigma_{1},\dots,\sigma_{s}) \cdot 
     T_{\wt(w \Rightarrow \kappa(\eta_{2}))} \\
& \stackrel{\eqref{eq:PiJ2}}{=} 
  (v_{1}\PS{\mu}(t_{\xi_1+\wt(w \Rightarrow \kappa(\eta_{2}))}),\,\dots,
     v_{s}\PS{\mu}(t_{\xi_s+\wt(w \Rightarrow \kappa(\eta_{2}))}); 
     \sigma_{0},\sigma_{1},\dots,\sigma_{s}) \\
& = (v_{1}\PS{\mu}(t_{\ti{\xi_1}}),\,\dots,
     v_{s}\PS{\mu}(t_{\ti{\xi_s}}); 
     \sigma_{0},\sigma_{1},\dots,\sigma_{s}) \quad 
   \text{by Lemma~\ref{lem:PiJ}\,(3)}, 
\end{align*}
as desired.
By the definition of $\Sigma_{N}$, 
the right-hand side of \eqref{eq:main1a} is:
\begin{equation} \label{eq:main1c}
\begin{split}
& (\pi_{\lambda}^{u_1} \cdot T_{\ti{\gamma}_1+\ti{\xi}_1})^{\otimes N(\tau_1-\tau_0)} 
 \otimes \cdots \otimes 
(\pi_{\lambda}^{u_p} \cdot T_{\ti{\gamma}_{p}+\ti{\xi}_1})^{\otimes N(\tau_p-\tau_{p-1})} \\
& \qquad \otimes 
(\pi_{\mu}^{v_1} \cdot T_{\ti{\xi}_1})^{\otimes N(\sigma_1-\sigma_0)} 
 \otimes \cdots \otimes 
(\pi_{\mu}^{v_s} \cdot T_{\ti{\xi}_s})^{\otimes N(\sigma_s-\sigma_{s-1})}.
\end{split}
\end{equation}

Now, we see from the commutative diagram \eqref{eq:CD5a} and 
the definition of $\Sigma_{N}$ that 
\begin{align*}
& \Theta_{\lambda\mu}^{(N)}(\eta') =
  (\Theta_{\lambda\mu}^{(N)} \circ \Sigma_{N}')(\eta) = 
  ((\Sigma_{N} \otimes \Sigma_{N}) \circ \Theta_{\lambda\mu})(\eta) = 
  (\Sigma_{N} \otimes \Sigma_{N})(\eta_{1} \otimes \eta_{2}) \\
&  = 
(\eta_{\lambda}^{u_1})^{\otimes N(\tau_1-\tau_0)} 
 \otimes \cdots \otimes 
(\eta_{\lambda}^{u_p})^{\otimes N(\tau_p-\tau_{p-1})} \otimes 
(\eta_{\mu}^{v_1})^{\otimes N(\sigma_1-\sigma_0)} 
 \otimes \cdots \otimes 
(\eta_{\mu}^{v_s})^{\otimes N(\sigma_s-\sigma_{s-1})}.
\end{align*}
Therefore, by applying Proposition~\ref{prop:s2}, 
we deduce that $\Xi_{\lambda\mu}^{(N)}
(\pi_{\eta'} \cdot T_{\wt(w \Rightarrow \kappa(\eta'))})$ 
(which is identical to the left-hand side of \eqref{eq:main1a}, as seen above) 
is identical to the element in \eqref{eq:main1c} 
(which is identical to the right-hand side of \eqref{eq:main1a}, as seen above). 
Thus we have shown \eqref{eq:main1a}, 
thereby completing the proof of Theorem~\ref{thm:main}. 
\end{proof}

%
{\small
}
\end{document}